\documentclass[12pt]{article}

\usepackage[a4paper,margin=0.75in,right=0.75in]{geometry}

\usepackage{amsbsy,amsthm}
\usepackage{amsfonts}
\usepackage{amsmath}
\usepackage{amssymb}
\usepackage{amsthm}
\usepackage{bbm}
\usepackage{booktabs}
\usepackage[font=small,labelfont=bf,labelsep=quad]{caption}
\usepackage{cite}
\usepackage{colortbl,dcolumn}
\usepackage{enumerate}
\usepackage{float}
\usepackage{graphicx}
\usepackage{hyperref}
\usepackage{indentfirst}
\usepackage{latexsym}
\usepackage{mathrsfs}
\usepackage{psfrag}
\usepackage{verbatim}
\usepackage{newtxmath} 

\providecommand{\keywords}[1]{\textbf{\text{Keywords: }} #1}

\numberwithin{equation}{section}

\newtheorem{assumption}{Assumption}[section]
\newtheorem{remark}[assumption]{Remark}
\newtheorem{lemma}[assumption]{Lemma}
\newtheorem{theorem}[assumption]{Theorem}   

\hypersetup{colorlinks=true,                         
linkcolor=red,
anchorcolor=red,
citecolor=red}

\begin{document}

\title{Long time strong convergence analysis of one-step methods for McKean-Vlasov SDEs with superlinear growth coefficients \thanks{
Emails addresses: 222101027@csu.edu.cn; yaozhong@ualberta.ca;sqgan@csu.edu.cn.}}
\author{\normalsize Taiyuan Liu$^{a}$,\quad Yaozhong Hu$^{b}$,\quad Siqing Gan$^{a}$
\\
\footnotesize $^{a}$School of Mathematics and Statistics, HNP-LAMA, Central South University, Changsha, China
\\
\footnotesize $^{b}$Department of Mathematical and Statistical Sciences, AB T6G 2G1, University of Alberta, Edmonton, Canada
}

\maketitle

\begin{abstract}
    This paper presents a strong convergence rate  analysis of general discretization  approximations for McKean-Vlasov SDEs with  super-linear growth    coefficients over infinite time horizon. Under some specified non-globally Lipschitz conditions,  we derive the propagation of chaos,  and the  mean-square convergence rate  over infinite time horizon   for general one-step time discretization schemes for the underlying Mckean-Vlasov SDEs. As an application of the general result  it is  obtained  the  mean-square convergence rate over infinite time horizon for  two numerical schemes: the projected Euler scheme and the backward Euler scheme for Mckean-Vlasov SDEs in non-globally Lipschitz settings. Numerical experiments  are   provided to validate the  theoretical findings.
\end{abstract}


\section{Introduction}\label{2023SIS-section:Introduction}

McKean-Vlasov stochastic differential equations (MV-SDEs), also referred to as distribution-dependent or mean-field SDEs, are characterized by their coefficients that depend on the law of the unknown  solution itself. They  describe large systems where each individual particle  is affected by the aggregate dynamics of the entire group.   
In this work, we focus on the long time convergence analysis of several numerical schemes   for  superlinear 
autonomous $d$-dimensional MV-SDEs of the following form:
\begin{equation}\tag{1.1}\label{1.1}
d X_t=b\left(X_t, \mu_{t}^{X}\right) d t+\sigma\left(X_t, \mu_{t}^{X}\right) d W_t, \quad  t\geq 0
\end{equation}
with initial condition $X_0 \in \mathbb{R}^d$, where $\mu_t^X$ represents the marginal distribution of the process $X$ at time $t$.
 
 Originally introduced by  \cite{mckean1966class, mckean1967propagation}, this category of equations    have  attracted growing attention after the pioneering contribution of  \cite{dawson1983critical}.   They frequently emerge as the asymptotic dynamics of large  interacting particle systems governed by mean-field interactions and   have  applications  in diverse scientific domains, including physics, biology, engineering, quantitative finance, mean-field games, and neuroscience, and have attracted substantial attention in recent years.  
As illustrative examples, MV-SDEs characterize both the optimal paths in mean-field control problems \cite{bensoussan2013mean, sznitman1991topics} and the equilibrium trajectories in mean-field games (e.g.  \cite{ mishura2020existence, wu2022stabilization}). MV-SDEs play a fundamental role in reducing complexity in managing control problems in dispersed systems, for example robotic networks and electrical grids. By emphasizing the mean-field statistical properties of agents rather than modeling their individual interactions, MV-SDEs substantially reduce computational complexity. Moreover, these equations can be used to model  uncertainties in stochastic environment, thereby enabling the development of stable control techniques.  Other practical applications include cooperative agent networks (see \cite{benachour1998nonlinear, bossy1997stochastic}) and related fields such as filtering theory (as explored in \cite{crisan2010approximate}). Additional relevant references  are referred to  \cite{baladron2012mean, bolley2011stochastic, bossy2015clarification, dreyer2011phase, guhlke2018stochastic, mckean1966class}.

Similar to deterministic differential equations, exact solutions to stochastic differential equations are often unattainable or exceedingly complex. This challenge extends to McKean-Vlasov SDEs as well. Consequently, numerical simulation emerges as an indispensable tool for solving MV-SDEs, requiring the development of effective numerical methods. Here, effectiveness typically implies convergence in an appropriate sense, preferably with a high order to enhance computational efficiency, while also preserving key properties of the original system, such as the invariance of certain regions or structural features.
Theorem~1 in \cite{hutzenthaler2011strong} demonstrates that, for standard SDEs featuring non-globally Lipschitz coefficients, the conventional Euler-Maruyama scheme does not maintain finite moments, resulting in divergence in both the strong and weak sense. An analogous divergence phenomenon has been observed in MV-SDEs, commonly referred to as “particle corruption", as detailed in Section~4.1 of \cite{dos2022simulation}. Therefore, the design and analysis of convergent numerical methods under non-globally Lipschitz conditions represent a highly significant and compelling research area.

The practical significance of distribution-dependent stochastic differential equations has motivated extensive research into numerical algorithms for their approximation. Theoretical investigations over finite time horizons have primarily focused on the convergence of numerical schemes and their associated rates. With the condition that the coefficients exhibiting locally Lipschitz continuity with respect to the state variables,  strong convergence analyses of the Euler-Maruyama scheme applied to MV-SDEs have been conducted in \cite{li2023strong}. In cases where the coefficients exhibit super-linear growth in the state variables while maintaining linear growth in the measure variables, several studies have examined the boundedness of moments and convergence characteristics of diverse numerical schemes, with some establishing convergence rates over finite time intervals. Owing to their straightforward algebraic framework and minimal computational expense, explicit methods offer distinct advantages; consequently, several improved explicit schemes have been developed to approximate distribution-dependent SDEs featuring super-linear expansion. Specifically, tamed methods were introduced and analyzed in \cite{dos2022simulation, liu2023tamed, bao2021first, kumar2021explicit, Kumar2022Well-posedness}, where convergence rates were rigorously established. As an alternative to tamed schemes, adaptive Euler methods have been proposed and shown to achieve good  strong convergence rates, as demonstrated in \cite{reisinger2022adaptive}. Modified Euler methods were investigated in \cite{jian2025modified}, which established a global strong convergence rate of $\tfrac{1}{2}$ for MV-SDEs. Implicit methods, valued for their superior stability properties and larger stability regions compared to explicit schemes, are particularly effective in preserving the dissipative structures of the underlying systems over long time horizons. The convergence of backward Euler schemes was studied in \cite{dos2022simulation}, while the split-step backward Euler scheme was analyzed in \cite{chen2022flexible, chen2024euler}, with convergence rates derived.

In the context of high-dimensional distribution-dependent SDEs, the primary challenge is to mitigate explosion of dimensions. To confront this difficulty, multilevel Picard approximations were proposed in \cite{hutzenthaler2022multilevel} for solving high-dimensional McKean-Vlasov equations, substantially reducing the impact of dimensional growth. Additionally, deep learning-based algorithms for approximating solutions to high-dimensional McKean-Vlasov equations have been developed in \cite{han2024learning} and \cite{germain2022numerical}.

It is noteworthy that the long-time numerical analysis of classical  stochastic differential equations holds significant importance across various scientific disciplines, including high-dimensional sampling, Bayesian reasoning, statistical mechanics, and artificial intelligence \cite{dalalyan2017theoretical, hong2019invariant, song2020score}. For SDEs with coefficients that are globally Lipschitz continuous, long time  approximation findings have been derived in \cite{mattingly2010convergence, talay1990second, talay2002stochastic}. In contrast, when the coefficients are  non-global Lipschitz conditions, fewer  studies have examined the long-time approximations of the solution \cite{liu2023backward, WANG2024long-timeStrong, angeli2023uniform, crisan2021uniform}.

The investigation   of  mean-field SDEs over  infinite time horizon is of significant importance, yet it poses substantial challenges. Unlike standard SDEs, MV-SDEs lack the strong Markov property, which precludes the direct application of stopping time technique  and complicates theoretical analysis. In \cite{yuanping2024explicit}, the long-time moment boundedness of truncated Euler methods was analyzed, along with the convergence of the associated invariant measures. The long time  stability of the Euler method was investigated in \cite{liu2023stability}, while \cite{tran2025infinite} provided a long-time convergence analysis for adaptive tamed Euler methods. Furthermore, \cite{bao2024uniform} conducted long-time convergence analyses for backward Euler, tamed Euler, and adaptive Euler schemes. Due to the absence of a unified analytical framework, classical finite-time convergence techniques often fail under these more general conditions. As a result, the long-time convergence analysis of numerical methods for MV-SDEs remains a critical and largely unresolved research area.   
To the best of our knowledge,   strong convergence analysis  for long-time numerical schemes approximating MV-SDEs featuring non-globally Lipschitz  coefficients has not yet been established in prior studies.  Our current work intends to fill this gap. 
Rather than concentrating on particular numerical methods, we propose a comprehensive framework that incorporates a broad range of one-step numerical approximations. Under this framework, we first construct  long-time propagation of chaos and prove their 
long-time strong convergence theorems. A similar methodology was employed in \cite{bao2024uniform} to develop a long-time strong convergence, where  however, the authors  partially relies on global Lipschitz assumptions and yields strong convergence results in the $\mathcal{W}_1$-Wasserstein distance. In contrast, the assumptions adopted here are more general than those in \cite{bao2024uniform}, and moreover  we obtain strong convergence results in the $\mathcal{W}_2$-Wasserstein distance, thereby extending and enhancing the prior findings.
 
The principal contributions of this research are outlined below:
\begin{itemize}
  \item We establish a new long-time version of the propagation of chaos in Theorem~\ref{propagationofchaos}.
  \item 
  For MV-SDEs exhibiting super-linear growth in the drift and diffusion coefficients, we develop a novel framework for the fundamental long-time strong convergence theorem  in Theorems~\ref{theorem 4.2} and~\ref{theorem 4.3}. This framework accommodates a broad range of one-step numerical methods for mean-field SDEs, enabling the derivation of long-time strong convergence rates for these methods. As an application, we derive explicit long-time strong convergence rates for the projected Euler and backward Euler methods   as detailed in Theorems~\ref{theorem5.9} and~\ref{theorem6.9}, respectively.
  
  \item The long-time boundedness of moments for one-step methods applied to MV-SDEs constitutes a critical component in establishing the long-time convergence of these methods. We extend the analytical techniques employed for demonstrating long-time moment boundedness in one-step methods for classical SDEs (as outlined in Lemma~\ref{lemma5.2}) to the setting of MV-SDEs. To overcome the principal challenge in these proofs—namely, the treatment of measure-dependent terms—we employ a judicious combination of Hölder's inequality, Jensen's inequality, and the binomial expansion.
\end{itemize}

The rest of the paper is structured as follows. Section~\ref{Notations and assumptions} introduces some  necessary notation and assumptions. Section~\ref{Long-time behavior of McKean-Vlasov SDEs} establishes the propagation of chaos   and derives the uniform moment boundedness of solutions to the interacting particle system.  Both of these results are over infinite horizon.    Section~\ref{The long-time mean-square convergence results of the one-step method for MV-SDEs} is dedicated to developing a general framework for the long-time mean-square convergence of one-step time-discretization schemes applied to MV-SDEs. Sections~\ref{strong convergence rate of the projected Euler method over infinite time} and~\ref{strong convergence rate of the backward Euler scheme  over infinite time} analyze the long-time strong convergence of two well-known methods—the projected Euler method and the backward Euler method—under non-globally Lipschitz conditions, and derive their respective long-time convergence rates within the proposed framework. Section~\ref{numerical_experiments} presents numerical examples to validate the theoretical results. To facilitate the reading of the paper, we postpone the  detailed proofs  in Section~\ref{Appendix}.

\section{Notations and assumptions}\label{Notations and assumptions}
\textbf{Notation:} Throughout this paper, let $\left(\Omega, \mathcal{F}, \mathbb{P}\right)$  denote a complete  probability space and the associated expectation is denoted by $\mathbb{E}$. Let   
 $\left\{W_t\right\}_{t \geq 0}$ be an  $m$-dimensional Wiener process.  Let  $\left\{ \mathcal{F}_t =\sigma(W_s, 0\le s\le t), t\ge 0\right\}$ be the filtration  generated by the 
Wiener process. We augment   $\mathcal{F}_0$ to contain all $\mathbb{P}$-null sets.       We define $S_{N}=\left\{ 1,2,...,N \right\} $ as the set of positive integers up to  $N$. The inner product and Euclidean norm of vectors in $\mathbb{R}^d$ are denoted by  $\left< \cdot ,\cdot \right>$ and $\left|\cdot\right|$, respectively. For a matrix $ B \in \mathbb{R}^{d \times m} $, $ B^T $ denotes its transpose, and its Hilbert-Schmidt norm is denoted by   $ \left\|B\right\| = \sqrt{\operatorname{trace}\left(B B^T\right)} $. For any $a, b \in \mathbb{R}$, we use $a \wedge b=\min \left\{a, b\right\}$ and $a \vee b=\max \left\{a, b\right\}$.  Let $\mathcal{P}\left(\mathbb{R}^d\right)$ represent the space of all probability measures on $\mathbb{R}^d$. 
The $L^q$-Wasserstein distance between measures $\mu$ and $\nu$ is given by
\begin{equation}\tag{2.1}\label{2.1}
\mathcal{W}_q\left(\mu, \nu\right)=\inf _{\pi \in \mathcal{C}\left(\mu, \nu\right)}\left(\int_{\mathbb{R}^d \times \mathbb{R}^d}\left|x-y\right|^q \pi\left(\mathrm{d} x, \mathrm{d} y\right)\right)^{\frac{1}{q}},
\end{equation}
where $\mathcal{C}\left(\mu, \nu\right)$ is the set of all measures $\pi$ on $\mathbb{R}^{2d}$ such that    $\pi\left(\cdot, \mathbb{R}^d\right)=\mu\left(\cdot\right)$ and $\pi\left(\mathbb{R}^d, \cdot\right)=\nu\left(\cdot\right)$. 

A probability measure $\mu \in \mathcal{P}\left(\mathbb{R}^d\right)$ has a finite $q$-th moment if
\begin{equation}\tag{2.2}\label{2.2}
\mathcal{W}_q\left(\mu\right) := \left( \int_{\mathbb{R}^d} \left|x\right|^q \, \mu\left(dx\right) \right)^{\frac{1}{q}} < \infty.
\end{equation}
We denote by $\mathcal{P}_q\left(\mathbb{R}^d\right)$ the subspace of $\mathcal{P}\left(\mathbb{R}^d\right)$ comprising all measures with finite $q$-th moment:
\begin{equation}\tag{2.3}\label{2.3}
\mathcal{P}_q(\mathbb{R}^d) = \left\{ \mu \in \mathcal{P}(\mathbb{R}^d) : \mathcal{W}_q(\mu) < \infty \right\}.
\end{equation}
In particular, for any $\mu \in \mathcal{P}_2(\mathbb{R}^d)$, it holds that $\mathcal{W}_2(\mu, \delta_0) = \mathcal{W}_2(\mu)$ (see \cite[Lemma~2.3]{GDReis2019Freidlin-Wentzell}), where $\delta_x$ denotes the Dirac measure centered at $x \in \mathbb{R}^d$. Moreover, a direct application of
H\"{o}lder’s inequality yields that, for $p \leq q$,  $\mathcal{W}_p\left( \mu ,\nu \right) \leq \mathcal{W}_q\left( \mu ,\nu \right)$.

For $x:=\left(x_1, \cdots, x_N\right) \in \mathbb{R}^{d\times N}$, where $x_i \in \mathbb{R}^d$ for $1 \leq i \leq N$, define
\begin{equation}\tag{2.4}\label{2.4}
\mu ^{x,N}:=\tfrac{1}{N}\sum_{i=1}^N{\delta _{x_i}}\in \mathcal{P} \left( \mathbb{R}^d \right).
\end{equation}
For $p\geq 1 $, due to
\begin{equation}\tag{2.5}\label{2.5}
\mathcal{W}_{2p}\left( \delta _{x_i} \right) ^{2p}=\int_{\mathbb{R} ^d}\left|x\right|^{2p}\delta _{x_i}\left(\mathrm{d}x\right)=\left|x_i\right|^{2p},
\end{equation}
we can conclude that
\begin{equation}\tag{2.6}\label{2.6}
\mathcal{W}_{2p}\left( \mu ^{x,N} \right) ^{2p}=\int_{\mathbb{R} ^d}\left|x\right|^{2p}\tfrac{1}{N}\sum_{i=1}^N{\delta _{x_i}}\left(\mathrm{d}x\right)=\tfrac{1}{N}\sum_{i=1}^N{\mathcal{W}_{2p}\left( \delta _{x_i} \right) ^{2p}}=\tfrac{1}{N}\sum_{i=1}^N{\left|x_i\right|^{2p}}.
\end{equation}
It is evident that for any $x=\left(x_1, \cdots, x_N\right) \in \mathbb{R}^{d\times N}$ and $y=\left(y_1, \cdots, y_N\right) \in \mathbb{R}^{d\times N}$,
\begin{equation}\tag{2.7}\label{2.7}
\tfrac{1}{N}\sum_{i=1}^N{\left( \delta_{x_i}\times\delta_{y_i}\right)}\in\mathcal{C} \left(\mu^{x,N},\mu^{y,N}\right) ,
\end{equation}
therefore,
\begin{equation}\tag{2.8}\label{2.8}
\mathcal{W}_{2}\left( \mu ^{x,N},\mu ^{y,N} \right)^{2}
\leq \int_{\mathbb{R} ^d\times \mathbb{R} ^d}\left|x-y\right|^2\tfrac{1}{N}\sum_{i=1}^N{\left( \delta _{x_i}\left( dx \right) \times \delta _{y_i}\left( dy \right) \right)}\\
=\tfrac{1}{N}\sum_{i=1}^N{\left| x_i-y_i \right|^2}.
\end{equation}

We focus on the following autonomous $d$-dimensional  distribution-dependent SDEs (MV-SDEs):
\begin{equation}\tag{2.9}\label{MVSDEs}
d X_t=b\left(X_t, \mu_{t}^{X}\right) d t+\sigma\left(X_t, \mu_{t}^{X}\right) d W_t, \quad   t\geq 0,
\end{equation}
where, throughout this paper, $\mu_t^X$ represents the distribution of $X$ at time $t$. The initial value $X_0$ is an $\mathcal{F}_0$-measurable random variable taking values in $\mathbb{R}^d$ and is independent of $W$. We impose the following assumptions on the coefficients
$ 
b: \mathbb{R}^d \times \mathcal{P}_2\left(\mathbb{R}^d\right) \rightarrow \mathbb{R}^d, \quad \sigma : \mathbb{R}^d \times \mathcal{P}_2\left(\mathbb{R}^d\right) \rightarrow \mathbb{R}^{d \times m}  
$.
\begin{assumption}[Contractive monotonicity condition]\label{ass2.1}
There exist positive constants $K_1>K_2\geq 0$ such that, for all $x, y \in \mathbb{R}^d$ and $\mu, \nu \in \mathcal{P}_2\left(\mathbb{R}^d\right)$,
\begin{equation}\tag{2.10}\label{2.10}
2\left\langle x-y, b\left(x, \mu\right)-b\left(y, \nu\right)\right\rangle+\left|\left|\sigma\left(x, \mu\right)-\sigma\left(y, \nu\right)\right|\right|^2 \leq-K_1\left|x-y\right|^2+K_2 \mathcal{W}_2\left(\mu, \nu\right)^2.
\end{equation}
\end{assumption}
\begin{assumption}[Polynomial growth Lipschitz condition]\label{ass2.2}
The drift function $b$ satisfies a polynomial growth Lipschitz condition; specifically, there exist positive constants $L_1 > 0$ and $\kappa \geq 1$ such that, for all $x, y \in \mathbb{R}^d$ and $\mu, \nu \in \mathcal{P}_2\left(\mathbb{R}^d\right)$,
\begin{equation}\tag{2.11}\label{2.11}
\left|b\left( x,\mu \right) -b\left( y,\nu \right) \right|^2\leq L_1\left[ \left( 1+\left|x\right|^{2\kappa -2}+\left|y\right|^{2\kappa -2} \right) \left|x-y\right|^2+\mathcal{W}_2\left( \mu ,\nu \right) ^2 \right].
\end{equation}
\end{assumption}
\begin{assumption}\label{ass2.3}
 There exist positive constant $L_2$ and $q_0 \geq 20$ such that
\begin{equation}\tag{2.12}\label{2.12}
 \mathbb{E}\left[ \left| X_0 \right|^{2q_0} \right] <+\infty,\quad \left| b\left( 0,\delta _0 \right) \right|\leq L_2.
 \end{equation}
\end{assumption}
\begin{assumption}\label{ass2.4}
There exist constants $a_1>a_2\geq 0$, $q^*\geq 4q_0-2$, and $C\geq 0$ such that, for all $x \in \mathbb{R}^d$ and $\mu \in \mathcal{P}_2\left(\mathbb{R}^d\right)$,
\begin{equation}\tag{2.13}\label{2.13}
2\left\langle x, b\left(x, \mu\right)\right\rangle+\left(q^*-1\right)\left|\left|\sigma\left(x, \mu\right)\right|\right|^2 \leq-a_1\left|x\right|^2+a_2 \mathcal{W}_2\left(\mu\right)^2+C.
\end{equation}
\end{assumption}
\begin{remark}\label{remark 2.5}
 Assumptions~\ref{ass2.1}--\ref{ass2.4} imply the existence of positive constants $K_3$, $K_4$, $K_5$, $a_3$, and $a_4$ such that, for all $x, y \in \mathbb{R}^d$ and $\mu, \nu \in \mathcal{P}_2\left(\mathbb{R}^d\right)$,
\begin{equation}\tag{2.14}\label{2.14}
\left\|\sigma \left( x,\mu \right) -\sigma \left( y,\nu \right) \right\|^2\leq K_3\left[ \left( 1+\left|x\right|^{\kappa -1}+\left|y\right|^{\kappa -1} \right) \left|x-y\right|^2+\mathcal{W}_2\left( \mu ,\nu \right) ^2 \right] ,
    \end{equation}
    \begin{equation}\tag{2.15}\label{2.15}
    \left|b\left( x,\mu \right) \right|^2\leq K_4\left( \left|x\right|^{2\kappa}+1 \right) +a_3\mathcal{W}_2\left( \mu \right) ^2 ,
    \end{equation}
    \begin{equation}\tag{2.16}\label{2.16}
    \left\|\sigma \left( x,\mu \right) \right\|^2\leq K_5\left( \left|x\right|^{\kappa+1}+1\right) +a_4\mathcal{W}_2\left( \mu \right) ^2 .
    \end{equation}
\end{remark}
\begin{remark}\label{remark 2.6}
Under Assumption~\ref{ass2.2}, for any fixed $\mu \in \mathcal{P}_2\left(\mathbb{R}^d\right)$, the drift coefficient $b\left(x, \mu\right)$ is continuous with respect to $x \in \mathbb{R}^d$.
\end{remark}
\begin{lemma}[Theorem~2.1 in \cite{Kumar2022Well-posedness}: Existence, uniqueness, and moment bounds on a finite horizon ]\label{lem:Well-posedness}
Under Assumptions~\ref{ass2.1}--\ref{ass2.4}, the MV-SDE \eqref{MVSDEs} admits a unique strong solution. Moreover, the following bounds hold:
\begin{equation}\tag{2.17}\label{2.17}
\mathbb{E}\left[\underset{0\leqslant t\leqslant T}{\sup}\left|X_t\right|^{2q_0}\right]\leq K\,, 
\end{equation}
where $K:=K\left(T,  K_{1},a_{1},E\left[\left|X_0\right|^{2q_0}\right],d,\kappa \right) >0$ is a constant.  
\end{lemma}
\begin{lemma}[Lemma~8.1 in \cite{KIto1964numericalstationarysolutions}: Modified Gronwall's lemma ]\label{lem:Gronwalllemma}
Let $v: \left[0, \infty\right) \to \left[0, \infty\right)$ and $\xi: \left[0, \infty\right) \to \left[0, \infty\right)$ be continuous functions with $v\left(0\right) \geq 0$. Suppose there exists a positive constant $\beta$ such that
\begin{equation}\tag{2.18}\label{2.18}
v\left(t\right)-v\left(s\right) \leq-\beta \int_s^t v\left(r\right) d r+\int_s^t \xi\left(r\right) d r,
\end{equation}
holds for any $0 \leq s < t < \infty$. Then,
\begin{equation}\tag{2.19}\label{2.19}
v\left(t\right) \leq v\left(0\right)+\int_0^t e^{-\beta\left(t-r\right)} \xi\left(r\right) d r.
\end{equation}
\end{lemma}
\section{Main Results}
\subsection{Long-time behavior of McKean-Vlasov SDEs}\label{Long-time behavior of McKean-Vlasov SDEs}
In this section, we present the propagation of chaos over an infinite time horizon and establish the uniform moment bounds of solutions to the interacting particle system over infinite time. We begin by introducing the concept of propagation of chaos. As is well-known, a  primary difficulty in developing numerical methods for MV-SDEs is the discretization of the probability distribution embedded in the coefficients. A standard approach approximates the original MV-SDEs using an interacting particle system, in which the marginal probability distribution embedded in the coefficients of the MV-SDEs is replaced by the empirical distribution derived from the interacting particles. This approximation technique is known as the propagation of chaos.

 To establish the propagation of chaos over an infinite time horizon we  fix an $N \in \mathbb{N}_{+}$, and let $\left(\left\{W_t^i\right\}_{t \geq 0}, X_0^i\right), 1 \leq i \leq N$,  be independent and identically distributed copies of  $\left(\left\{W_t\right\}_{t \geq 0}, X_0\right)$ defined on the same probability space $\left(\Omega, \mathcal{F}, \mathbb{P},\left\{\mathcal{F}_t\right\}_{t \geq 0}\right)$. Consider the following interacting particle system (IPS):
\begin{equation}\tag{3.1}\label{interacting}
    \left\{\begin{array}{l}
\mathrm{d} X_t^{i, N}=b\left(X_t^{i, N}, \mu_t^{X, N}\right) \mathrm{d} t+\sigma\left(X_t^{i, N}, \mu_t^{X, N}\right) \mathrm{d} W_t^i, \quad t \geq 0, \\
\left.X_t^{i, N}\right|_{t=0}=X_0^i, \quad 1 \leq i \leq N,
\end{array}\right.
\end{equation}
where
$$
\mu_t^{X, N}=\tfrac{1}{N} \sum_{i=1}^N \delta_{X_t^{i, N}}
$$
denotes the empirical distribution of the $N$ particles $X_t^{i, N}$, $i=1,2, \cdots, N$.

We also write  the corresponding non-interacting particle system (NIPS):
\begin{equation}\tag{3.2}\label{noninteracting}
\mathrm{d} X_t^i=b\left(X_t^i, \mu_t^{X^i}\right) \mathrm{d} t+\sigma\left(X_t^i, \mu_t^{X^i}\right) \mathrm{d} W_t^i, \quad \forall t \geq 0,
\end{equation}
with the initial condition $X_0^i$ for $1 \leq i \leq N$. By the weak uniqueness established in Lemma~\ref{lem:Well-posedness}, we obtain $\mu_t^{X^i} = \mu_t^X$ for all $t \geq 0$ and $1 \leq i \leq N$. 
\begin{theorem}[Propagation of chaos over an infinite time horizon]\label{propagationofchaos}
Assume that Assumptions~\ref{ass2.1}--\ref{ass2.4} hold with $K_1 > 2K_2$. Then, for any $2 < l \leq q^*$, there exists a constant $C_{d,l}$ such that

\begin{equation}\tag{3.3}\label{3.4}
\underset{1\leqslant i\leqslant N}{\sup}\underset{t\geqslant 0}{\sup}\mathbb{E} \left[ \left| X_{t}^{i}-X_{t}^{i,N} \right|^2 \right] \le \widetilde{K}\Upsilon \left(N\right),
\end{equation}
where 
$\widetilde{K}=\tfrac{2K_2C_{d,l}}{K_1-2K_2}
$ is a positive constant, and
\begin{equation}\tag{3.4}\label{3.5}
\Upsilon \left(N\right):= \left( \mathbb{E} \left[ \left| X_0 \right|^{l} \right] +1 \right) ^{\frac{2}{l^2}}\times\begin{cases}N^{-\frac{1}{2}}+N^{-\frac{{l}-2}{l}}, & d<4 \text { and } {l} \neq 4, \\ N^{-\frac{1}{2}} \log \left(1+N\right)+N^{-\frac{{l}-2}{{l}}}, & d=4 \text { and } {l} \neq 4, \\ N^{-\frac{2}{d}}+N^{-\frac{{l}-2}{l}}, & d>4 \text { and } {l} \neq \frac{d}{d-2}.\end{cases} \\
\end{equation}
\end{theorem}
Next, we establish the uniform moment bound  of solutions to the IPS over an infinite time horizon
\begin{theorem}[Uniform moment boundedness of solutions to the IPS over infinite time]\label{momentboundednessofthesolution}
Under Assumptions~\ref{ass2.1}--\ref{ass2.4}, there exists a positive constant $C_1 = C_1\left(p\right)$, independent of $t$ and $N$, such that
\begin{equation}\tag{3.5}\label{3.25}
\underset{t\geqslant 0}{\mathrm{sup}}\underset{1\leqslant j\leqslant N}{\sup} \mathbb{E}\left[\left|X^{j,N}_{t}\right|^{2 p}\right] \leq C_1\underset{1\leqslant j\leqslant N}{\sup} \mathbb{E}\left[\left(1+\left|X_0^{j}\right|^{2 p}\right)\right]=C_1\mathbb{E}\left[\left(1+\left|X_0\right|^{2 p}\right)\right], \quad 1 \leq p \leq q_0.
\end{equation}
\end{theorem}

\subsection{Long-time mean-square convergence   of the one-step method}\label{The long-time mean-square convergence results of the one-step method for MV-SDEs}

In this section, we aim to establish the long-time mean-square convergence result for general one-step time discretization schemes applied to the MV-SDEs. We begin by presenting a general long-time strong convergence theorem for interacting particle systems (\ref{interacting}) with super-linear growth.  This result, together  with  earlier results on propagation of chaos   in the infinite-time setting, 
will produce a long-time mean-square convergence   for one-step time-discretization schemes applied to MV-SDEs \eqref{MVSDEs}.

To approximate the interacting particle system (\ref{interacting}), we adopt a uniform step size $h=\tfrac{1}{M}$ with  $M \geq 1$. Let $X^{j,N}\left(t, x^{j,N} ; s\right)$ for $t \leq s<\infty$ refer to the solution of (\ref{interacting}) with  the initial condition $X^{j,N}\left(t, x^{j,N} ; t\right)=x^{j,N}$. In particular, for $t \geq 0$, we write $X^{j,N}\left(t\right)$ to refer to the solution to the SDEs (\ref{interacting}) with initial value $X^{j,N}\left(0\right)=X^{j,N}_0$. A general one-step (local) approximation $Y^{j,N}\left(t, x^{j,N} ; t+h\right)$ to $X^{j,N}\left(t, x^{j,N} ; t+h\right)$, when  $s=t+h $ 
is sufficiently close to $t$, can be expressed as
\begin{equation}\tag{3.6}\label{4.1}
Y^{j,N}\left(t, x^{j,N} ; t+h\right)=x^{j,N}+\varPsi^{j,N}\left(t, x^{j,N}, h ; \zeta ^{j,N}_t\right),
\end{equation}
where $\varPsi^{j,N}$ maps $\left[0, \infty\right) \times \mathbb{R}^d \times\left(0, 1\right) \times \mathbb{R}^m$ to $\mathbb{R}^d$, and $\zeta ^{j,N}_t$ is a random variable defined on the probability space $\left(\Omega, \mathcal{F}, \mathbb{P}\right)$ that possesses sufficiently high-order moments. 
An example of such a mapping $\varPsi$ is  the 
Euler-Maruyama scheme, given by
\begin{equation}\tag{3.7}\label{4.2}
\varPsi^{j,N}\left(t, x^{j,N}, h ; \zeta ^{j,N}_t\right)=hb\left(x^{j,N},\mu^{x^{j,N},N} _t \right)+\sigma\left(x^{j,N},\mu^{x^{j,N},N}_t\right) \Delta W^{j}\left(h\right),
\end{equation}
where $\Delta W^{j}\left(h\right)=W^{j}\left(t+h\right)-W^{j}\left(t\right)$.
With this local one-step approximation, the original solution can be approximated by an explicit scheme  on the mesh grids $\left\{t_k=k h, k \geq 0\right\}$
\begin{equation}\tag{3.8}\label{4.3}
Y^{j,N}_0=X^{j,N}_0, \quad Y^{j,N}_{k+1}=Y^{j,N}_k+\varPsi^{j,N}\left(t_k, Y^{j,N}_k, h ; \zeta^{j,N} _k\right), k \geq 0,
\end{equation}
where $\zeta^{j,N}_k$ for $k \geq 0$ are independent of $Y^{j,N}_0, Y^{j,N}_1, \cdots, Y^{j,N}_{k-1}, \zeta^{j,N}_0, \zeta^{j,N}_1, \cdots, \zeta^{j,N}_{k-1}$. An alternative formulation can be written as:
\begin{equation}\tag{3.9}\label{4.4}
Y^{j,N}_0=X^{j,N}_0, \quad Y^{j,N}_{k+1}=Y^{j,N}\left(t_k, Y^{j,N}_k ; t_{k+1}\right)=Y^{j,N}\left(t_0, Y^{j,N}_0 ; t_{k+1}\right), k \geq 0.
\end{equation}

We present a general  long-time strong convergence theorem for interacting
particle systems as detailed below.
\begin{theorem}[A general long-time strong convergence theorem for interacting particle systems]\label{theorem 4.2}
Assume the following conditions hold:\\
$\left(\mathrm{H}_1\right)$ Assumptions \ref{ass2.1}--\ref{ass2.4} are satisfied.\\
$\left(\mathrm{H}_2\right)$ The local 
 one-step approximation $Y^{j,N}\left(t_0, X^{j,N}_0 ; h\right)$, as defined in (\ref{4.1}), achieves local weak and strong errors of order $q_1>1$ and $q_2\in \left(\tfrac{1}{2}, q_1-\tfrac{1}{2}\right)$, respectively:  there exist constants $1 \leq \eta_1 \leq q_0$, $1 \leq \eta_2 \leq q_0$, and $0<h_0\leq 1$
such that, for all $0<h \leq h_0$, 
\begin{equation}\tag{3.10}\label{4.19}
\underset{1\leqslant j\leqslant N}{\sup}\left|\mathbb{E}\left[X^{j,N}\left(t_0, X^{j,N}_0 ; h\right)-Y^{j,N}\left(t_0, X^{j,N}_0 ; h\right)\right]\right| \leq C_2\mathbb{E}\left[\left(1+\left|X_0\right|^{\eta_1}\right)\right] h^{q_1}, 
\end{equation}
\begin{equation}\tag{3.11}\label{4.20}
\underset{1\leqslant j\leqslant N}{\sup}\left\{\mathbb{E}\left[\left|X^{j,N}\left(t_0, X^{j,N}_0 ; h\right)-Y^{j,N}\left(t_0, X^{j,N}_0 ; h\right)\right|^{2}\right]\right\}^{\frac{1}{2}}\leq C_3\left\{\mathbb{E}\left[\left(1+\left|X_0\right|^{2\eta_2}\right) \right]\right\}^{\frac{1}{2}}h^{q_2},
\end{equation}
where $C_2, C_3>0$ are constants independent of 
 $h, N$.\\
$\left(\mathrm{H}_3\right)$ The global approximation $Y^{j,N}_k$, as defined in (\ref{4.3}), possesses finite moments. Specifically, there exist constants $h_0>0$, $\eta_3 \geq 1$, and $C_4>0$ such that, for all $0<h \leq h_0$, $1 \leq p\eta_3 \leq q_0$, and $k \geq 0$,
\begin{equation}\tag{3.12}\label{4.21}
\underset{1\leqslant j\leqslant N}{\sup}\mathbb{E}\left[\left|Y^{j,N}_k\right|^{2 p}\right] \leq C_4\mathbb{E}\left[\left(1+\left|X_0\right|^{2 p\eta_3}\right)\right],
\end{equation}
where $C_4$ is independent of  $h, k,N$.\\
Then, there exists a constant $\lambda:=\max \left\{\eta_1\eta_3 ,\left(\tfrac{\kappa-1}{2}+\eta_2\right)\eta_3 \right\} \leq q_0$  such that, for all  $h \leq h_1:=\min \left\{\tfrac{1}{2\left(K_1-K_2\right)}, h_0\right\}$, 
\begin{equation}\tag{3.13}\label{4.22}
\underset{k\geqslant0}{\sup}\underset{1\leqslant j\leqslant N}{\sup}\left\{\mathbb{E}\left[\left|X^{j,N}_k-Y^{j,N}_k\right|^{2}\right]\right\}^{\frac{1}{2}}\le C_5\left\{\mathbb{E}\left[\left(1+\left|X_0\right|^{2 \lambda}\right)\right]\right\}^{\frac{1}{2}} h^{q_2-\frac{1}{2}},
\end{equation}
where $C_{5}$ is a constant independent of $h, k,N,d$.
\end{theorem}

We now present the long-time mean-square convergence
results for the one-step method applied to the MV-SDEs (\ref{MVSDEs}).
\begin{theorem}[Long-time mean-square convergence results for one-step methods applied to MV-SDEs]\label{theorem 4.3}
Let $X^{i}$ denote the solution to (\ref{noninteracting}), and $Y^{i,N}$ be the one-step numerical approximation to (\ref{interacting}). Assume that Assumptions \ref{ass2.1}--\ref{ass2.4} hold with $K_1>2 K_2$, and that conditions $\left(\mathrm{H}_2\right)$ and $\left(\mathrm{H}_3\right)$ are satisfied. Then, for any $2<l\leq q^*$ and $h \leq h_1:=\min \left\{\tfrac{1}{2\left(K_1-K_2\right)}, h_0\right\}$, the global mean-square error satisfies
\begin{equation}\tag{3.14}\label{4.40}
\underset{k\geqslant0}{\sup}\underset{1\leqslant i\leqslant N}{\sup}\mathbb{E}\left[\left|X_k^i-Y_k^{i, N}\right|^2\right] 
\leq C_6 \begin{cases}N^{-\frac{1}{2}}+N^{-\frac{{l}-2}{l}}+h^{2q_2-1}, & d<4 \text { and } {l} \neq 4, \\ N^{-\frac{1}{2}} \log \left(1+N\right)+N^{-\frac{{l}-2}{{l}}}+h^{2q_2-1}, & d=4 \text { and } {l} \neq 4, \\ N^{-\frac{2}{d}}+N^{-\frac{{l}-2}{l}}+h^{2q_2-1}, & d>4 \text { and } {l} \neq \frac{d}{d-2} ,\end{cases}
\end{equation}
where $C_6$ is a constant independent of $h$, $k$, $N$.
\end{theorem}

\subsection{Strong convergence rate of the projected Euler method over infinite time}\label{strong convergence rate of the projected Euler method over infinite time}
In this section, we apply the general long-time mean-square convergence results for one-step methods to MV-SDEs to establish a strong convergence rate for the projected Euler scheme under standard conditions. The projected Euler method for classical SDEs has been well   explored  (see, e.g., \cite{Wolf-Jurgen2016project}). For the interacting particle system (\ref{interacting}), it takes the following form:
\begin{equation}\tag{3.15}\label{5.1}
\left\{\begin{array}{l}
\bar{Y}_{k}^{j,N}:=\psi\left(Y_{k-1}^{j,N}\right)
, \\
Y^{j,N}_{k+1}:=\bar{Y}^{j,N}_k+h b\left(\bar{Y}^{j,N}_k,\mu^{\bar{Y}^{j,N}_k,N}_k\right)+\sigma\left(\bar{Y}^{j,N}_k,\mu^{\bar{Y}^{j,N}_k,N}_k\right) \Delta W^{j}_{k},
\end{array}\right.
\end{equation}
where $\Delta W^{j}_{k}:=W^{j}\left(t_{k+1}\right)-W^{j}\left(t_{k}\right)$, $k \geq 0$, $Y^{j,N}_0=X^{j}_0$, $\mu^{\bar{Y}^{j,N}_k,N}_k:=\tfrac{1}{N} \sum_{j=1}^N \delta_{\bar{Y}^{j,N}_k}$ for $j\in S_{N}$, and the projection operator $\psi: \mathbb{R}^d \to \mathbb{R}^d$ is defined by
$\psi\left(x\right) = \min \left\{ 1, h^{-\frac{1}{2\left(\kappa + 2\right)}} \left|x\right|^{-1} \right\} x$, with $\kappa$ as given in (\ref{2.11}).
\begin{remark}\label{ass5.1}
The projection operator $\psi$ satisfies the following conditions:
\begin{equation}\tag{3.16}\label{5.2}
\left|\psi\left(x\right)\right| \leq h^{-\frac{1}{2\left(\kappa+2\right)}}, 
\end{equation}
\begin{equation}\tag{3.17}\label{5.3}
\left|\psi\left(x\right)-\psi\left(y\right)\right|  \leq\left|x-y\right|,
\end{equation}
and $\psi(0)=0 \in \mathbb{R}^d$. Setting $y=0$ in (\ref{5.3}) yields:
\begin{equation}\tag{3.18}\label{5.4}
\left|\psi\left(x\right)\right| \leq\left|x\right|.
\end{equation}
\end{remark}

To establish the strong convergence rate for the method \eqref{5.1}, we divide the analysis into two parts. The first part establishes the boundedness of the numerical solution in the $2p$-th moment. 

\begin{theorem}[Moment bounds for the projected Euler scheme]\label{momentboundsforPEM}
Under Assumptions \ref{ass2.1}--\ref{ass2.4}, the following holds for any $1 \leq p < \left\lfloor \frac{q_0}{2} \right\rfloor$ and $0<h \leq h_2:=\min \left\{h_1, \tfrac{1}{p\left(a_1-a_2\right)}\right\}$ 
\begin{equation}\tag{3.19}\label{5.16}
\underset{1\leqslant j\leqslant N}{\sup}\mathbb{E}\left[\left|Y^{j,N}_k\right|^{2 p}\right] \leq K_6 \underset{1\leqslant j\leqslant N}{\sup}\mathbb{E}\left[\left(1+\left|X_0^{j}\right|^{4 p}\right)\right]= K_6\mathbb{E}\left[\left(1+\left|X_0\right|^{4 p}\right)\right],
\end{equation}
where $K_{6}$ is a constant independent of $h, k, N$.
\end{theorem}

The  one-step approximation of the projected Euler scheme (\ref{5.1}) can be reformulated as follows:
\begin{equation}\tag{3.20}\label{5.63}
\begin{aligned}
Y^{j,N}\left( t,x^{j,N};t+h \right) :=&\psi\left(x^{j,N}\right)+\int_t^{t+h}{b\left( \psi\left(x^{j,N}\right),\mu ^{\psi\left(x^{j,N}\right),N} \right)}ds\\
&+\int_t^{t+h}{\sigma \left( \psi\left(x^{j,N}\right),\mu ^{\psi\left(x^{j,N}\right),N} \right)}dW^j\left( s \right). 
\end{aligned}
\end{equation}

The second part focuses on establishing the local weak and strong errors, as specified in (\ref{4.19}) and (\ref{4.20}).

\begin{theorem}\label{theorem5.7}
Under Assumptions \ref{ass2.1}--\ref{ass2.4}, the following bounds hold for any $1\leq \kappa \leq \tfrac{q_0-11}{9}$ and $0<h \leq h_2$:
\begin{equation}\tag{3.21}\label{5.85}
\underset{1\leqslant j\leqslant N}{\sup}\left|\mathbb{E}\left[X^{j,N}\left(t, x^{j,N} ; t+h\right)-Y^{j,N}\left(t, x^{j,N} ; t+h\right)\right]\right|  \leq K_{7}\mathbb{E}\left[\left(1+\left|x^{j,N}\right|^{ 4\kappa+6}\right)\right] h^{\frac{3}{2}},
\end{equation}
    \begin{equation}\tag{3.22}\label{5.86}
    \underset{1\leqslant j\leqslant N}{\sup} \left\{\mathbb{E}\left[\left|X^{j,N}\left(t, x^{j,N} ; t+h\right)-Y^{j,N}\left(t, x^{j,N} ; t+h\right)\right|^{2}\right]\right\}^{\frac{1}{2}} \leq K_{8}\left\{\mathbb{E}\left[\left(1+\left|x^{j,N}\right|^{8\kappa+12 }\right)\right]\right\}^{\frac{1}{2}} h,
    \end{equation}
    where $K_{7}$ and $K_{8}$ are constants independent of $h, N$.
\end{theorem} 

Theorem \ref{theorem 4.2} establishes the strong convergence rate of the projected Euler method (\ref{5.1}) for the interacting particle system (\ref{interacting}) over infinite time. Specifically, the subsequent statement is valid:
  \begin{theorem}\label{theorem5.8}
Suppose Assumptions \ref{ass2.1}--\ref{ass2.4} hold. Then, the projected Euler scheme (\ref{5.1}), when applied to the interacting particle system (\ref{interacting}) with super-linear growth, achieves a long-time strong convergence rate of order $\tfrac{1}{2}$. Specifically, for any $0<h \leq h_2$ and $1\leq \kappa \leq \tfrac{q_0-11}{9}$, 
\begin{equation}\tag{3.23}\label{5.83}
\underset{k\geqslant0}{\sup}\underset{1\leqslant j\leqslant N}{\sup}\left\{\mathbb{E}\left[\left|X^{j,N}_k-Y^{j,N}_k\right|^{2}\right]\right\}^{\frac{1}{2}}\leq C_{7}\left\{\mathbb{E} \left[ \left( 1+\left| X_0 \right|^{18\kappa+22} \right) \right]\right\}^{\frac{1}{2}} h^{\frac{1}{2}},
\end{equation}
where $C_{7}$ is a constant independent of $h,k,N,d$.
\end{theorem}

This theorem, together  with Theorem \ref{propagationofchaos}, establishes the long-time mean-square convergence  of the projected Euler scheme  (\ref{5.1}) applied to MV-SDEs (\ref{MVSDEs}). In particular, it holds that

\begin{theorem}\label{theorem5.9}
Let $X^{i}$ be a solution of  (\ref{noninteracting}), and $Y^{i,N}$ be the projected Euler scheme for (\ref{interacting}). Assume that Assumptions~\ref{ass2.1}--\ref{ass2.4} are satisfied with $K_1 > 2K_2$. Then, the projected Euler method (\ref{5.1}), when applied to the MV-SDEs (\ref{MVSDEs}), exhibits long-time mean-square convergence. More precisely, for any $2<l\leq q^*$, $1\leq \kappa \leq \tfrac{q_0-11}{9}$ and $0<h \leq h_2$,  
\begin{equation}\tag{3.24}\label{5.84}
\underset{k\geqslant0}{\sup}\underset{1\leqslant i\leqslant N}{\sup}\mathbb{E}\left[\left|X_k^i-Y_k^{i, N}\right|^2\right] 
\leq C_8 \begin{cases}N^{-\frac{1}{2}}+N^{-\frac{{l}-2}{l}}+h, & d<4 \text { and } {l} \neq 4, \\ N^{-\frac{1}{2}} \log \left(1+N\right)+N^{-\frac{{l}-2}{{l}}}+h, & d=4 \text { and } {l} \neq 4, \\ N^{-\frac{2}{d}}+N^{-\frac{{l}-2}{l}}+h, & d>4 \text { and } {l} \neq \frac{d}{d-2} ,\end{cases}
\end{equation}
where $C_8$ is a constant independent of $h,k,N$.
\end{theorem}
\subsection{Strong convergence rate of the backward Euler scheme  over infinite time }\label{strong convergence rate of the backward Euler scheme  over infinite time}
In this section, we apply the general long-time mean-square convergence theorem for one-step schemes in the context of MV-SDEs. This enables us to investigate the strong convergence rate of the backward Euler method under specific conditions. The backward Euler scheme for the interacting particle system (\ref{interacting}) takes the following form:
\begin{equation}\tag{3.25}\label{6.1}
\hat{X}_{k+1}^{i,N}=\hat{X}_{k}^{i,N}+hb\left( \hat{X}_{k+1}^{i,N},\mu _{k}^{X,N} \right) +\sigma \left( \hat{X}_{k}^{i,N},\mu _{k}^{X,N} \right) \Delta W_{k}^{i},
\end{equation}
where $ \Delta W_{k}^i:=W^i\left(t_{k+1}\right)-W^i\left(t_{k}\right)$, $k\geq0$, $\hat{X}_0^{i, N}={X}_{0}^{i}$, and $ \mu_k^{X, N}\left(dx\right):=\tfrac{1}{N} \sum_{j=1}^N \delta_{\hat{X}_{k}^{j,N}}\left(d x\right)$, $i\in S_{N}$.

To apply the general  long-time mean-square convergence theorem  for one-step methods   to MV-SDEs, the initial step involves establishing the moment  bound  for the numerical solution. For this purpose, we impose the following assumption. 

\begin{assumption}\label{ass6.1}
Suppose that the diffusion coefficient of MV-SDEs (\ref{MVSDEs}) obeys a global Lipschitz condition. Specifically, there is a constant $0<L_{\sigma}<\tfrac{a_1-a_2}{5\left( 2q_0-1 \right)}$ such that
\begin{equation}\tag{3.26}\label{6.2}
\left|\left|\sigma \left( x,\mu \right) -\sigma \left( y,\nu \right) \right|\right|^2\leq L_{\sigma}\left[ \left|x-y\right|^2+\mathcal{W}_2\left( \mu ,\nu \right) ^2 \right].
\end{equation}
\end{assumption}

\begin{theorem}[Moment bounds for the backward Euler scheme]\label{momentboundsforBEM}
Under Assumptions \ref{ass2.1}--\ref{ass2.4} and \ref{ass6.1}, the following holds for any $1 \leq p < \left\lfloor \frac{q_0}{2} \right\rfloor$, and $0<h \leq h_2$  \begin{equation}\tag{3.27}\label{6.5}
\underset{1\leqslant j\leqslant N}{\sup}\mathbb{E}\left[\left|\hat{X}^{j,N}_k\right|^{2 p}\right] \leq K_9 \underset{1\leqslant j\leqslant N}{\sup}\mathbb{E}\left[\left(1+\left|X_0^{j}\right|^{4 p}\right)\right]\le  K_9\mathbb{E}\left[\left(1+\left|X_0\right|^{4 p}\right)\right],
\end{equation}
where $K_9$ is a constant independent of $h, k, N$.
\end{theorem}


The one-step approximation of the backward Euler-Maruyama method is given by
\begin{equation}\tag{3.28}\label{6.22}
\begin{aligned}
\hat{X}^{j,N}\left( t_0,X^{j,N}_0;t_0+h \right) :=&X^{j,N}_0+\int_{t_0}^{t_0+h}{b\left( \hat{X}^{j,N}\left( t_0,X^{j,N}_0;t_0+h \right) ,\mu _{t_0}^{X,N} \right)}ds\\
&+\int_{t_0}^{t_0+h}{\sigma \left( X^{j,N}_0,\mu _{t_0}^{X,N} \right)}dW^j\left( s \right).\\
\end{aligned}
\end{equation} 

\begin{theorem}\label{theorem6.7}
   Assuming that Assumptions~ \ref{ass2.1}--\ref{ass2.4} and \ref{ass6.1} hold, the following estimates are valid for any $1 \leq\kappa\leq \tfrac{2q_0-23}{17}$, and $0<h \leq h_2$   
\begin{equation}\tag{3.29}\label{6.44}
\underset{1\leqslant j\leqslant N}{\mathrm{sup}}\left|\mathbb{E}\left [X^{j,N}\left( t_0,X^{j,N}_0;t_0+h \right) -\hat{X}^{j,N}\left( t_0,X^{j,N}_0;t_0+h \right) \right]\right|\le K_{10}\mathbb{E} \left[\left( 1+\left|X_0\right|^{8\kappa+2}\right) \right] h^{\frac{3}{2}},
\end{equation}
\begin{equation}\tag{3.30}\label{6.45}
\underset{1\leqslant j\leqslant N}{\mathrm{sup}}\left\{ \mathbb{E} \left[ \left|X^{j,N}\left( t_0,X^{j,N}_0;t_0+h \right) -\hat{X}^{j,N}\left( t_0,X^{j,N}_0;t_0+h \right)\right |^2 \right] \right\} ^{\frac{1}{2}}\le K_{11}\left\{ \mathbb{E} \left[ \left(1+\left|X_0\right|^{16\kappa +24}\right) \right] \right\} ^{\frac{1}{2}}h,
\end{equation}
where $K_{10}$ and $K_{11}$ are constants independent of $h, N$.
\end{theorem}

Building on these preparations, we can now establish the strong convergence rate of the backward Euler method (\ref{6.1}) for the interacting particle system (\ref{interacting}) over infinite time.  

\begin{theorem}\label{theorem6.8}
Assume that Assumptions~\ref{ass2.1}--\ref{ass2.4} and~\ref{ass6.1} are satisfied. Then, the backward Euler scheme  (\ref{6.1}), when applied to the interacting particle system (\ref{interacting}) with super-linear growth, achieves a long-time strong convergence rate of order $\frac{1}{2}$. Specifically,  for any $0<h \leq h_2$ and $1 \leq \kappa \leq \tfrac{2q_0-23}{17}$,  
\begin{equation}\tag{3.31}\label{6.42}
\underset{k\geqslant0}{\sup}\underset{1\leqslant j\leqslant N}{\sup}\left\{\mathbb{E} \left[ \left|X^{j,N}_k-\hat{X}^{j,N}_k\right|^2 \right]\right\}^{\frac{1}{2}}\le C_{9}\left\{\mathbb{E} \left[ \left( 1+\left| X_0 \right|^{17\kappa+23} \right) \right]\right\}^{\frac{1}{2}} h^{\frac{1}{2}},
\end{equation}
where $C_{9}$ is a constant independent of $h, k,N,d$.
\end{theorem}

Combining Theorem \ref{theorem6.8} with Theorem \ref{propagationofchaos} yields the long-time mean-square convergence of the backward Euler method (\ref{6.1}) applied to the MV-SDEs (\ref{MVSDEs}). Specifically, it holds that
\begin{theorem}\label{theorem6.9}
Let $X^{i}$ be a solution of  (\ref{noninteracting}), and $\hat{X}^{i,N}$ be the backward Euler scheme  for (\ref{interacting}). Assume Assumptions \ref{ass2.1}--\ref{ass2.4} and \ref{ass6.1} are satisfied with $K_1>2 K_2$. Then, the backward Euler scheme (\ref{6.1}), when applied to the MV-SDEs (\ref{MVSDEs}), exhibits long-time mean-square convergence. More precisely,  for any $2<l\leq q^*$, $0<h \leq h_2$ and $1 \leq \kappa \leq \tfrac{2q_0-23}{17}$, 
\begin{equation}\tag{3.32}\label{6.43}
\underset{k\geqslant0}{\sup}\underset{1\leqslant i\leqslant N}{\sup}\mathbb{E}\left[\left|X_k^i-\hat{X}_k^{i, N}\right|^2\right]
\leq C_{10} \begin{cases}N^{-\frac{1}{2}}+N^{-\frac{{l}-2}{l}}+h, & d<4 \text { and } {l} \neq 4, \\ N^{-\frac{1}{2}} \log \left(1+N\right)+N^{-\frac{{l}-2}{{l}}}+h, & d=4 \text { and } {l} \neq 4, \\ N^{-\frac{2}{d}}+N^{-\frac{{l}-2}{l}}+h, & d>4 \text { and } {l} \neq \frac{d}{d-2} ,\end{cases}
\end{equation}
where $C_{10}$ is independent of $h,k,N$.
\end{theorem}
\section{Numerical experiments}
\label{numerical_experiments}
In this section, we demonstrate the preceding theoretical findings through numerical simulations of MV-SDEs with super-linear growth. Newton's iteration procedure is utilized to resolve the implicit algebraic equations that emerge from the backward Euler method. We provide two examples to demonstrate the long-time mean-square convergence rates of the projected Euler and backward Euler schemes.

To approximate the marginal distribution $\mu _{t_k}^{X}$ at every time point $t_k$,  where $k$ is a natural number, using its empirical measure $
\mu _{t_k}^{X,N}=\frac{1}{N}\sum_{i=1}^N{\delta _{X_{_{t_k}}^{i,N}}}
$, we simulate    $N=5000$ particles.

Since the exact solution of the MV-SDEs with super-linear growth is not given explicitly, we approximate it using a numerical reference solution computed using a significantly finer time step. The reference values are obtained employing a Monte Carlo technique integrated with numerical approximations and a reference step size $ h_{\text{ref}} = 2^{-16} $. We employ Monte Carlo simulations with 10,000 sample paths to calculate the root mean square error (RMSE) for time steps $ h \in \{ 2^{-12}, 2^{-11}, 2^{-10}, 2^{-9}, 2^{-8} \} $, defined as
\begin{equation}\tag{4.1}\label{7.1}
RMSE:=\sqrt{\tfrac{1}{N}\sum_{j=1}^N{\left| Y_{T}^{j,N}-X_{T}^{j,N} \right|^2}},
\end{equation}
at the terminal time $ T = 16 $, where $ X_T^{j,N} $ is the reference solution calculated with respect to $ h_{\text{ref}} = 2^{-16} $.

\smallskip
\noindent 
\textbf{Example 4.1}  Consider the one-dimensional MV-SDE given by 
\begin{equation*}\tag{4.2}\label{7.2}
\begin{cases}
	\mathrm{d}X_t=\left( aX_t\left( -b-\left|X_t\right| \right) +c\mathbb{E} \left[ \left|X_t\right| \right] \right) \mathrm{d}t+eX_t\mathrm{d}W_t, t\in \left( 0,T \right] ,\\
	X_0=x,\\
\end{cases}
\end{equation*}
with parameter values $a=5$, $b=2$, $c=0.25$, $e=0.125$, $x=1$.

Define
$b\left( x,\mu \right): =5x\left( -2-\left|x\right| \right) +\tfrac{1}{4}\int_{\mathbb{R}}{\left|x\right|}\mu \left( dx \right) $ and $\sigma\left( x,\mu \right): =\tfrac{x}{8}$.  
To begin, we verify that these coefficients satisfy Assumptions \ref{ass2.1}--\ref{ass2.4} and \ref{ass6.1}.  For any $x,y\in \mathbb{R}$, $\mu ,\nu \in \mathcal{P} _2\left( \mathbb{R} \right)$,
\begin{equation}\tag{4.3}\label{7.3}
\begin{aligned}
&2\left\langle x-y,b\left(x,\mu \right)-b\left(y,\nu \right)\right\rangle +\left\| \sigma\left (x,\mu \right)-\sigma \left(y,\nu \right)\right\| ^2\\
\leq& -19\left|x-y\right|^2-10\left( x-y \right) \left[ x\left|x\right|-y\left|y\right| \right] +2\times \tfrac{1}{4}\left( x-y \right) \left( \int_{\mathbb{R}}{\left|x\right|}\mu \left( dx \right) -\int_{\mathbb{R}}{\left|y\right|}\nu \left( dy \right) \right)\\
\leq& -\tfrac{303}{16}\left|x-y\right|^2+\mathcal{W}_2\left( \mu ,\nu \right) ^2
,
\end{aligned}
\end{equation}
where we have applied the fact that $10\left( x-y \right) \left( x\left|x\right|-y\left|y\right| \right) \geq 0$ and $\mathcal{W}_1\left( \mu ,\delta _0 \right) \leq \mathcal{W}_1\left( \mu ,\nu \right) +\mathcal{W}_1\left( \nu ,\delta _0 \right) $.
Thus, Assumption \ref{ass2.1} holds with $K_1=\tfrac{303}{16}$ and $K_2=1$.

It is straightforward to verify that 
\begin{equation}\tag{4.4}\label{7.4}
\begin{aligned}
\left|b\left( x,\mu \right) -b\left( y,\nu \right) \right|^2\leq 600\left[ \left( 1+\left|x\right|^2+\left|y\right|^2 \right) \left|x-y\right|^2+\mathcal{W}_2\left( \mu ,\nu \right) ^2 \right] ,
\end{aligned}
\end{equation}
thus, Assumption \ref{ass2.2}  holds with $L_1=600$ and $\kappa=2$.

Furthermore, since $\left| b\left( 0,\delta _0 \right) \right|=0$ and by selecting $q_{0}=40$, Assumption \ref{ass2.3} holds.

Next, we verify Assumption  \ref{ass2.4}. We see 
\begin{equation}\tag{4.5}\label{7.5}
\begin{aligned}
2\left\langle x,b\left(x,\mu \right)\right\rangle +\left(q^*-1\right)\left\| \sigma \left(x,\mu \right)\right\| ^2
\leq -\left( 20-\tfrac{q^*}{64} \right) x^2+4\mathcal{W}_2\left( \mu \right) ^2.
\end{aligned}
\end{equation}
By setting $q^*=512$, we obtain $a_1= 20-\frac{q^*}{64}=12$ and $a_2=4$. Therefore,  Assumption \ref{ass2.4} holds.

Moreover, there is a constant $0<L_{\sigma}=\tfrac{1}{50}<\tfrac{a_1-a_2}{5\left( 2q_0-1 \right)}=\tfrac{8}{395}$ such that 
\begin{equation}\tag{4.6}\label{7.6}
\left\|\sigma \left( x,\mu \right) -\sigma \left( y,\nu \right) \right\|^2\leq \tfrac{1}{50}\left[ \left|x-y\right|^2+\mathcal{W}_2\left( \mu ,\nu \right) ^2 \right].
\end{equation} Thus, Assumption \ref{ass6.1} holds.

The projected Euler scheme applied to (\ref{7.2}) takes the following form:
\begin{equation}\tag{4.7}\label{7.7}
\left\{\begin{array}{l}
	Y_{0}^{j,N}=1,\\
	\bar{Y}_{k}^{j,N}=\min \left\{ 1,h^{-\frac{1}{8}}\left| Y_{k}^{j,N} \right|^{-1} \right\} Y_{k}^{j,N},\\
	Y_{k+1}^{j,N}=\bar{Y}_{k}^{j,N}+h\left[ 5\bar{Y}_{k}^{j,N}\left( -2-\left|\bar{Y}_{k}^{j,N}\right| \right) +\frac{1}{4}\frac{1}{N}\sum_{i=1}^N{\delta _{\left|\bar{Y}_{k}^{i,N}\right|}} \right]+\frac{1}{8}\bar{Y}_{k}^{j,N}\Delta W_{k}^{j}.\\
\end{array}\right.
\end{equation}

The backward Euler scheme applied to (\ref{7.2}) takes the following form:
\begin{equation}\tag{4.8}\label{7.8}
\left\{\begin{array}{l}
	\hat{X}_{0}^{i,N}=1,\\
	\hat{X}_{k+1}^{i,N}=\hat{X}_{k}^{i,N}+h\left[ 5\hat{X}_{k+1}^{i,N}\left( -2-\left|\hat{X}_{k+1}^{i,N}\right| \right) +\frac{1}{4}\frac{1}{N}\sum_{j=1}^N{\delta _{\left|\hat{X}_{k}^{j,N}\right|}} \right] +\frac{1}{8}\hat{X}_{k}^{i,N}\Delta W_{k}^{i}.\\
\end{array}\right.
\end{equation}

\begin{figure}[htp]
    \centering
    \begin{minipage}{0.48\linewidth}
        \centering
        \includegraphics[width = 1\linewidth]{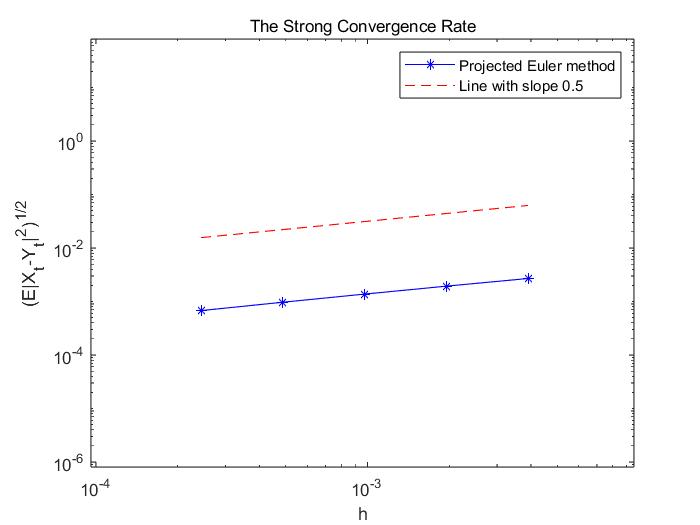}
        \caption{ Long time strong errors for PEM }
        \label{2023SIS-fig.1}
    \end{minipage}
    \begin{minipage}{0.48\linewidth}
        \centering
        \includegraphics[width = 1\linewidth]{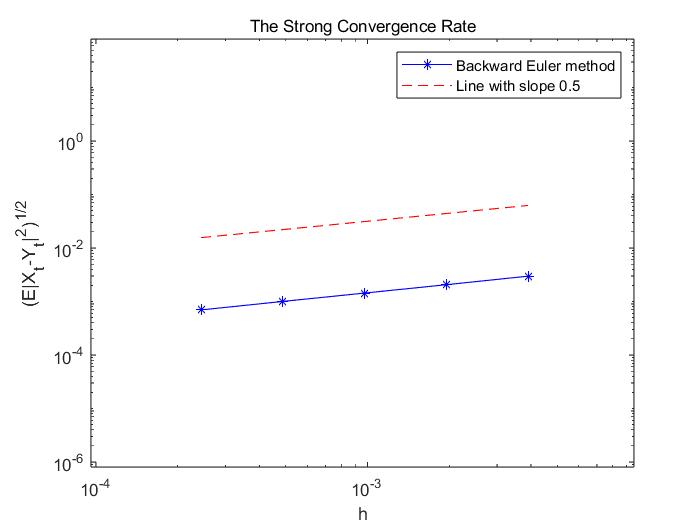}
        \caption{ Long time strong errors for BEM  }
        \label{2023SIS-fig.2}
    \end{minipage}
\end{figure}

Figure \ref{2023SIS-fig.1} illustrates the root mean square error between the  projected Euler method (\ref{7.7}) and the reference solution  as a function of the time step $h$. Similarly, Figure \ref{2023SIS-fig.2} depicts the root mean square error between the backward Euler method (\ref{7.8}) and the reference solution  as a function of the time step $h$. As anticipated, Figures  \ref{2023SIS-fig.1} and \ref{2023SIS-fig.2} show that the long-time strong convergence rates of both the projected Euler method (\ref{7.7}) and the backward Euler method (\ref{7.8}) are close to $\tfrac{1}{2}$ with regard to the time step size.

\smallskip
\noindent \textbf{Example  4.2}  Consider an   MV-SDE where the drift coefficient incorporates trigonometric functions:
\begin{equation}\tag{4.9}\label{7.9}
\begin{cases}
	\mathrm{d}X_t=\left( -aX_{t}^{3}-bX_t\left| X_t \right|-cX_t+d\sin X_t+e\mathbb{E} \left[ \left| X_t \right| \right] \right) \mathrm{d}t+\left( fX_t+g\mathbb{E} \left[ \left| X_t \right| \right] \right) \mathrm{d}W_t,t\in \left( 0,T \right] ,\\
	X_0=x,\\
\end{cases}
\end{equation}
with parameter values $a=20$, $b=20$, $c=5$, $d=0.25$, $e=0.1$, $f=0.1$, $g=0.05$, $x=1$.

By direct computation, Assumption \ref{ass2.1} holds with $K_1=9$ and $K_2=\tfrac{3}{20}$. Assumption \ref{ass2.2} holds with $L_1=30000$ and $\kappa=3$. Setting $q_{0}=40$ and noting that $\left| b\left( 0,\delta _0 \right) \right|=0$, Assumption \ref{ass2.3} holds. With $q^*=158$, we obtain $a_1=\tfrac{651}{100}$ and $a_2=\tfrac{177}{200}$, thereby fulfilling Assumption \ref{ass2.4}. Besides, there is a constant $0<L_{\sigma}=\tfrac{1}{80}<\tfrac{a_1-a_2}{5\left( 2q_0-1 \right)}=\tfrac{9}{632}$ such that Assumption \ref{ass6.1} holds.

We construct the projected Euler and backward Euler schemes  for the interacting particle system  associated with (\ref{7.9}). The projected Euler method is formulated as:
\begin{equation}\tag{4.10}\label{7.10}
\left\{ \begin{array}{l}
	Y_{0}^{j,N}=1,\\
	\bar{Y}_{k}^{j,N}=\min \left\{ 1,h^{-\frac{1}{10}}\left| Y_{k}^{j,N} \right|^{-1} \right\} Y_{k}^{j,N},\\
	\begin{aligned}  Y_{k+1}^{j,N}=&\bar{Y}_{k}^{j,N}+h\left[ -20\left( \bar{Y}_{k}^{j,N} \right) ^3-20\bar{Y}_{k}^{j,N}\left| \bar{Y}_{k}^{j,N} \right|-5\bar{Y}_{k}^{j,N}+\tfrac{1}{4}\sin \bar{Y}_{k}^{j,N}+\tfrac{1}{10}\tfrac{1}{N}\sum_{i=1}^N{\delta _{\left| \bar{Y}_{k}^{i,N} \right|}} \right] \\
    &+\left( \tfrac{1}{10}\bar{Y}_{k}^{j,N}+\tfrac{1}{20}\tfrac{1}{N}\sum_{i=1}^N{\delta _{\left| \bar{Y}_{k}^{i,N} \right|}} \right) \Delta W_{k}^{j}.\\
	\end{aligned}
\end{array} \right. 
\end{equation}

The backward Euler method is given by:
\begin{equation}\tag{4.11}\label{7.11}
\left\{ \begin{array}{l}
	\hat{X}_{0}^{i,N}=1,\\
	\begin{aligned} \hat{X}_{k+1}^{i,N}=&\hat{X}_{k}^{i,N}+h\left[ -20\left( \hat{X}_{k+1}^{i,N} \right) ^3-20\hat{X}_{k+1}^{i,N}\left| \hat{X}_{k+1}^{i,N} \right|-5\hat{X}_{k+1}^{i,N}+\tfrac{1}{4}\sin \hat{X}_{k+1}^{i,N}+\tfrac{1}{10}\tfrac{1}{N}\sum_{j=1}^N{\delta _{\left| \hat{X}_{k}^{j,N} \right|}} \right] \\
    &+\left( \tfrac{1}{10}\hat{X}_{k}^{i,N}+\tfrac{1}{20}\tfrac{1}{N}\sum_{j=1}^N{\delta _{\left| \hat{X}_{k}^{j,N} \right|}} \right) \Delta W_{k}^{i}.\\
    \end{aligned}
\end{array} \right. 
\end{equation}
\begin{figure}[htp]
    \centering
    \begin{minipage}{0.48\linewidth}
        \centering
        \includegraphics[width = 1\linewidth]{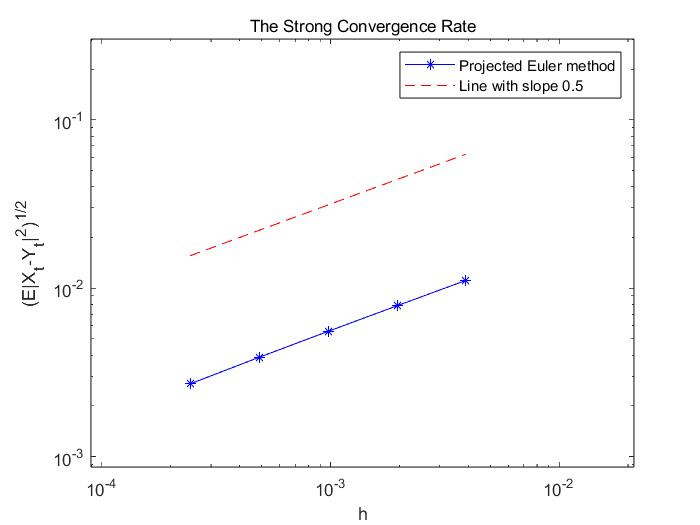}
        \caption{ Long time strong errors for PEM }
        \label{2023SIS-fig.3}
    \end{minipage}
    \begin{minipage}{0.48\linewidth}
        \centering
        \includegraphics[width = 1\linewidth]{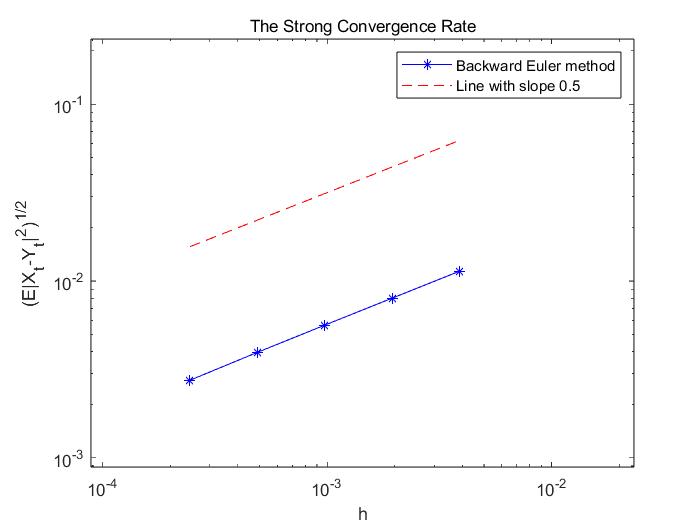}
        \caption{ Long time strong errors for BEM  }
        \label{2023SIS-fig.4}
    \end{minipage}
\end{figure}

Figure \ref{2023SIS-fig.3} plots the RMSE between the  projected Euler method (\ref{7.10}) and the reference solution as a function of $h$, while Figure \ref{2023SIS-fig.4} presents the RMSE for the backward Euler method (\ref{7.11}) relative to the reference solution, also as a function
of $h$. As expected, both figures indicate that the long-time strong convergence rates of  the projected Euler scheme (\ref{7.10}) together with the backward Euler scheme (\ref{7.11}) are close to 0.5 with regard to the time step size.

\section{Appendix}\label{Appendix}
\subsection{Proof of Theorem \ref{propagationofchaos}}
 Prior to proving the propagation of chaos, we first recall the following lemma.
\begin{lemma}(\cite[Theorem 1]{Fournier2015empirical}  Moment estimates for empirical measures)\label{lem:Momentestimatesformeasure}
Let $\left\{Z_m\right\}_{m \geq 1}$ be a series of independent and identically distributed random variables on $\mathbb{R}^d$ with common distribution $\mu \in \mathcal{P}_{l}\left(\mathbb{R}^d\right)$. Define the empirical measure as $\mu^N=\tfrac{1}{N} \sum_{j=1}^N \delta_{Z_j}$. Subsequently, for any $0<r<l$, there is a constant $C_{r, l, d}$ such that, for every $N \geq 1$,
\begin{equation}\tag{5.1}\label{estimatesformeasure}
\mathbb{E}\left[\mathcal{W}^r_r\left(\mu^N, \mu\right)\right] \leq C_{r, l, d} \mathcal{W}_{l}\left( \mu \right)^{\frac{r}{l}} \begin{cases}N^{-\frac{1 }{2} }+N^{-\frac{l-r}{l}}, & \text { if } r>\frac{d}{2} \text { and } l \neq 2 r, \\ N^{-\frac{1 }{2}} \log \left(1+N\right)+N^{-\frac{l-r}{l}}, & \text { if } r=\frac{d}{2}  \text { and } l \neq 2 r, \\ N^{-\frac{r}{d}}+N^{-\frac{l-r}{l}}, & \text { if } 0<r<\frac{d}{2} \text { and } l \neq \frac{d}{d-r}.\end{cases}
\end{equation}
\end{lemma}
We can now complete the proof of Theorem \ref{propagationofchaos}.
\begin{proof}[Proof of Theorem \ref{propagationofchaos}]
By applying It{\^o}'s formula and by using the properties of the It{\^o} integral, for each $1 \leq j \leq N$, it is inferred from equations (\ref{interacting}) and (\ref{noninteracting}) that

\begin{equation}\tag{5.2}\label{3.9}
\begin{aligned}
\mathbb{E}\left[\left|X_{t} ^j-X_{t}^{j, N}\right|^2\right]= & 2 \mathbb{E}\left[\int_0^{t}\left\langle X_s^j-X_s^{j, N}, b\left(X_s^j, \mu_s^{X^j}\right)-b\left(X_s^{j, N}, \mu_s^{X, N}\right)\right\rangle d s \right]\\
& +\mathbb{E}\left[\int_0^{t}\left\|\sigma\left(X_s^j, \mu_s^{X^j}\right)-\sigma\left(X_s^{j, N}, \mu_s^{X, N}\right)\right\|^2 d s\right].
\end{aligned}
\end{equation}
Next, applying Assumption~\ref{ass2.1}, we obtain
\begin{equation}\tag{5.3}\label{3.10}
\mathbb{E}\left[\left|X_{t} ^j-X_{t}^{j, N}\right|^2\right]\leq \mathbb{E}\left[\int_0^{t}-K_1 \left|X_s^j-X_s^{j, N}\right|^2+K_2 \mathcal{W}_2\left(\mu_s^{X^j}, \mu_s^{X, N}\right)^2 d s \right].\\
\end{equation}
Denote $\mu_t^N:=\tfrac{1}{N} \sum_{i=1}^N \delta_{X_t^i}$.  It is clear that
\begin{equation}\tag{5.4}\label{3.11}
\begin{aligned}
& \mathbb{E}\left[\left|X_t^j-X_t^{j, N}\right|^2 \right] \\
\leq &\mathbb{E}\left[\int_0^{t} -K_1 \left|X_s^j-X_s^{j, N}\right|^2 +\tfrac{2 K_2}{N} \sum_{i=1}^N \left|X_s^i-X_s^{i, N}\right|^2  +2 K_2 \mathcal{W}_2\left(\mu_s^{X^j}, \mu_s^N\right)^2 d s\right],\\
\end{aligned}
\end{equation}
where we have used the inequality 
\begin{equation}\tag{5.5}\label{3.12}
\mathbb{E}\left[\mathcal{W}_2\left(\mu_t^{X^j}, \mu_t^{X, N}\right)^2\right] \leq 2 \mathbb{E}\left[\mathcal{W}_2\left(\mu_t^{X^j}, \mu_t^N\right)^2\right]+2 \mathbb{E}\left[\mathcal{W}_2\left(\mu_t^N, \mu_t^{X, N}\right)^2\right].
\end{equation}
Summing (\ref{3.11}) over $j = 1$ to $N$  and dividing by $ N $ gives
\begin{equation}\tag{5.6}\label{3.13}
\begin{aligned}
&\tfrac{1}{N}\sum_{j=1}^N{\mathbb{E} \left[ \left| X_{t}^{j}-X_{t}^{j,N} \right|^2 \right]}\\
\leq&\mathbb{E}   \int_0^t\left[ {-}\tfrac{K_1}{N}\sum_{j=1}^N{\left| X_{s}^{j}-X_{s}^{j,N} \right|^2}+\tfrac{2K_2}{N}\sum_{i=1}^N{\left| X_{s}^{i}-X_{s}^{i,N} \right|^2}+\tfrac{2K_2}{N}\sum_{j=1}^N{\mathcal{W} _2\left( \mu _{s}^{X^j},\mu _{s}^{N} \right) ^2}  \right]ds .\\
\end{aligned}
\end{equation}
By setting $r=2$ in Lemma \ref{lem:Momentestimatesformeasure}, we obtain that, for any $2<l\leq q^*$ and $t \geq 0$, a constant $ C_{d,l} $ exists that relies solely on $ d $ and $ l $, such that
\begin{equation}\tag{5.7}\label{3.15}
\mathbb{E}\left[\mathcal{W}_2\left(\mu_t^{X^j}, \mu_t^N\right)^2\right] \leq C_{d, {l}}\mathcal{W} _{l}\left( \mu _{t}^{X^j} \right) ^{\frac{2}{l}} \eta\left(N\right)\leq C_{d, {l}}\left(\mathbb{E}\left[\left|X^j_t\right|^{l}\right]\right)^{\frac{2}{l^2}} \eta \left(N\right),
\end{equation}
where 
\begin{equation}\tag{5.8}\label{3.16}
\eta \left(N\right):= \begin{cases}N^{-\frac{1}{2}}+N^{-\frac{{l}-2}{l}}, & d<4 \text { and } {l} \neq 4, \\ N^{-\frac{1}{2}} \log \left(1+N\right)+N^{-\frac{{l}-2}{l}}, & d=4 \text { and } {l} \neq 4, \\ N^{-\frac{2}{d}}+N^{-\frac{{l}-2}{{l}}}, & d>4 \text { and } {l} \neq \frac{d}{d-2}.\end{cases}
\end{equation}
Given 
\begin{equation}\tag{5.9}\label{3.14}
\mathcal{W} _{l}\left( \mu _{t}^{X^j} \right) =\left( \int_{\mathbb{R} ^d}\left|x\right|^{l}\mu _{t}^{X^j}\left(\mathrm{d}x\right) \right) ^{\frac{1}{l}}=\left( \mathbb{E}\left[ \left|X_{t}^{j}\right|^{l} \right] \right) ^{\frac{1}{l}} \,, 
\end{equation} 
we aim to show that
\begin{equation}\tag{5.10}\label{3.17}
\underset{t\geqslant 0}{\mathrm{sup}}\left(\mathbb{E}\left[\left|X^j_t\right|^{l}\right]\right)^{\frac{2}{l^2}} \leq C_{l}\left(1+\mathbb{E}\left[\left|X^j_0\right|^{l}\right]\right)^{\frac{2}{{l^2}}},
\end{equation}
where $C_{{l}}$ is a constant depending only on $ {l}$.
In fact, for any $\alpha>0$ and $t\geq 0$, employing It{\^o}'s lemma together with Assumption \ref{ass2.4}  yields
\begin{equation}\tag{5.11}\label{3.18}
\begin{aligned}
& e^{\alpha  t} \mathbb{E}\left[\left|X_t^j\right|^{l}\right] \\
\leq & \mathbb{E}\left[\left|X_0^j\right|^{l}\right]+\int_0^t \alpha  e^{\alpha s} \mathbb{E}\left[\left|X_s^j\right|^{l}\right] d s \\
& +\tfrac{l}{2} \mathbb{E}\left[\int_0^t e^{\alpha  s}\left|X_s^j\right|^{{l} -2}\left(2\left\langle X_s^j, b\left(X_s^j, \mu_s^{X^j}\right)\right\rangle+\left({l}-1\right)\left\|\sigma\left(X_s^j, \mu_s^{X^j}\right)\right\|^2\right) d s\right] \\
\leq & \mathbb{E}\left[\left|X_0^j\right|^{l}\right]+\int_0^t \alpha  e^{\alpha  s} \mathbb{E}\left[\left|X_s^j\right|^{l}\right] d s \\
& +\tfrac{{l}}{2} \mathbb{E}\left[\int_0^t e^{\alpha  s}\left|X_s^j\right|^{{l}-2}\left(-a_1\left|X_s^j\right|^2+a_2 \mathcal{W}_2\left(\mu_s^{X^j}\right)^2+C\right) d s\right] \\
\leq & \mathbb{E}\left[\left|X_0^j\right|^{l}\right]+\int_0^t \alpha e^{\alpha s} \mathbb{E}\left[\left|X_s^j\right|^{l}\right] d s+\tfrac{{l}}{2} \int_0^t e^{\alpha s}\left(-a_1 \mathbb{E}\left[\left|X_s^j\right|^{l}\right]\right.\\
&\left. +a_2 \mathbb{E}\left[\left|X_s^j\right|^{l}\right]+\tfrac{C\left({l}-2\right)\varepsilon_{1}}{{l}} \mathbb{E}\left[\left|X_s^j\right|^{l}\right]\right) d s+ \tfrac{C}{ \varepsilon ^{\frac{l}{2}-1}_{1}}\int_0^t e^{\alpha s} d s\\
\leq & \mathbb{E}\left[\left|X_0^j\right|^{l}\right]+\int_0^t e^{\alpha s}\left(\alpha-\tfrac{{l} a_1}{2}+\tfrac{{l} a_2}{2}+\tfrac{C\left({l}-2\right)\varepsilon_{1}}{2}\right) \mathbb{E}\left[\left|X_s^j\right|^{l} \right]d s+ \tfrac{C}{\alpha \varepsilon ^{\frac{l}{2}-1}_{1}}\left[e^{\alpha t}-1\right] ,
\end{aligned}
\end{equation}
where we have used the fact
\begin{equation}\tag{5.12}\label{3.19}
\mathbb{E}\left[ \mathcal{W}_2\left( \mu _{s}^{X^j} \right) ^{l} \right] \leq \mathbb{E}\left[ \mathcal{W}_{l}\left( \mu _{s}^{X^j} \right) ^{l} \right] =\mathbb{E}\left[ \int_{\mathbb{R} ^d}{\left|x\right|^{l}\mu _{s}^{X^j}\left( dx \right)} \right] =\mathbb{E}\left[ \left|X_{s}^{j}\right|^{l} \right] ,
\end{equation}
and   Young's inequality
\begin{equation}\tag{5.13}\label{3.20}
ab\leq \varepsilon_{1} \tfrac{l-2}{l}a^{\frac{l}{l-2}}+\tfrac{1}{\varepsilon ^{\frac{l}{2}-1}_{1}}\tfrac{2}{l}b^{\frac{l}{2}}\,. 
\end{equation}
Applying Gronwall's lemma with $\varepsilon_{1} =\tfrac{l\left(a_1-a_2\right)}{2\left( l-2 \right) C}$ and $\alpha =\tfrac{l\left( a_1-a_2 \right)}{4}$ yields (\ref{3.17}). 
Consequently,     incorporating \eqref{3.15} and \eqref{3.17} into \eqref{3.13}, and noting that $\mathbb{E} \left[\left|X_0\right|^{l}\right]=\mathbb{E} \left[\left|X_0^{j}\right|^{l}\right]$ and $\Upsilon \left(N\right)=\left(1+\mathbb{E}\left[\left|X_0\right|^{l}\right]\right)^{\frac{2}{{l^2}}}\eta \left(N\right)$, we obtain
\begin{equation}\tag{5.14}\label{3.21}
\begin{aligned}
&\tfrac{1}{N}\sum_{j=1}^N{\mathbb{E} \left[ \left| X_{t}^{j}-X_{t}^{j,N} \right|^2 \right]}\\
\leq& -\left( K_1-2K_2 \right) \tfrac{1}{N}\sum_{j=1}^N{\int_0^t{\mathbb{E} \left[ \left| X_{s}^{j}-X_{s}^{j,N} \right|^2 \right]}}ds+\int_0^t{2K_2C_{d,l}\Upsilon \left(N\right)ds}.\\
\end{aligned}
\end{equation}
Define
\begin{equation}\tag{5.15}
v\left(t\right)= \tfrac{1}{N}\sum_{j=1}^N{\mathbb{E} \left[ \left| X_{t}^{j}-X_{t}^{j,N} \right|^2 \right]} , \beta=K_1-2K_2 , \xi\left(u\right)=2K_2C_{d,l}\Upsilon \left(N\right).
\end{equation}
Applying the modified Gronwall's lemma (Lemma \ref{lem:Gronwalllemma}), we conclude that, for all $t>0$,
\begin{equation}\tag{5.16}\label{3.22}
\tfrac{1}{N}\sum_{j=1}^N{\mathbb{E}\left[ \left| X_{t}^{j}-X_{t}^{j,N} \right|^2\right]} \leq \tfrac{2K_2C_{d,l}\Upsilon \left(N\right)}{K_1-2K_2}.
\end{equation}
Since the  sequence  $\left( X_{t}^{i} ,  X_{t}^{i,N} \right) _{1\leqslant i\leqslant N}$  is  identically distributed, this implies that $\left\{ \left| X_{t}^{i}-X_{t}^{i,N} \right|^2 \right\} _{1\leqslant i\leqslant N}$ is also identically distributed. Therefore, for any $1\leq i\leq N$,
\begin{equation}\tag{5.17}\label{3.23}
\mathbb{E} \left[ \left| X_{t}^{i}-X_{t}^{i,N} \right|^2 \right] =\tfrac{1}{N}\sum_{j=1}^N{\mathbb{E} \left[ \left| X_{t}^{j}-X_{t}^{j,N} \right|^2 \right]}\le \tfrac{2K_2C_{d,l}\Upsilon \left(N\right)}{K_1-2K_2}.
\end{equation}
Taking the supremum over $i$ in (\ref{3.23}) and setting $\widetilde{K}=\tfrac{2K_2C_{d,l}}{K_1-2K_2}
$, we deduce that
\begin{equation}\tag{5.18}\label{3.24}
\underset{1\leqslant i\leqslant N}{\sup}\underset{t\geqslant 0}{\sup}\mathbb{E} \left[ \left| X_{t}^{i}-X_{t}^{i,N} \right|^2 \right] \leq \widetilde{K}\Upsilon \left(N\right).
\end{equation}
The theorem is thus proved.  
\end{proof}
\subsection{Proof of Theorem \ref{momentboundednessofthesolution}}
Under Assumptions \ref{ass2.1}--\ref{ass2.4}, we establish uniform moment bounds for the solutions of the interacting particle system (\ref{interacting}) over infinite time.
\begin{proof}[Proof of Theorem \ref{momentboundednessofthesolution}] 
By applying It\^o's lemma in conjunction with the inner product inequality, for any $1\leq j\leq N$ and $\beta>0$, we obtain  
\begin{equation}\tag{5.19}\label{3.27}
\begin{aligned}
e^{\beta t}\left|X^{j,N}_{t}\right|^{2p} \leq & \left|X^j_0\right|^{2p}+\beta \int_0^t{e^{\beta s}\left|X^{j,N}_{s}\right|^{2p}ds}\\
&+2p \int_0^t{e^{\beta s}}\left|X^{j,N}_{s}\right|^{2p-2}\left\langle X^{j,N}_{s}, b\left( X^{j,N}_{s} ,\mu _{s}^{X,N} \right) 
\right\rangle \mathrm{d} s \\
& +2p \int_0^t{e^{\beta s}}\left|X^{j,N}_{s}\right|^{2p-2}\left\langle X^{j,N}_{s}, \sigma\left( X^{j,N}_{s} ,\mu _{s}^{X,N} \right)  \mathrm{d} W_s^j\right\rangle \\
& +p\left(2 p-1\right) \int_0^t{e^{\beta s}}\left|X^{j,N}_{s}\right|^{2p-2}\left\|\sigma\left( X^{j,N}_{s} ,\mu _{s}^{X,N} \right) \right\|^2 \mathrm{d} s.
\end{aligned}
\end{equation}
Since $q_0\geq 20$, it follows that $2p \leq 2q_0\leq 4q_0-2\leq q^*$.
By applying Assumption \ref{ass2.4} in conjunction with the characteristics of the It\^o integral, we derive
\begin{equation}\tag{5.20}\label{3.31}
\begin{aligned}
&e^{\beta t}\mathbb{E}\left[\left|X^{j,N}_{t}\right|^{2p}\right]\\
\leq&\mathbb{E}\left[\left|X_0^{j}\right|^{2p}\right]+\beta \mathbb{E}\left[\int_0^{t}{e^{\beta s}\left|X^{j,N}_{s}\right|^{2p}ds}\right]\\
&+p\mathbb{E}\left[\int_0^{t}{e^{\beta s}\left|X^{j,N}_{s}\right|^{2p-2}}\left( -a_1\left|X^{j,N}_{s}\right|^2+a_2\mathcal{W}_2\left( \mu _{s}^{X,N} \right) ^2+C \right) \mathrm{d}s\right].
\end{aligned}
\end{equation}
By applying (\ref{2.6}) and Young's inequality, we obtain
\begin{equation}\tag{5.21}\label{3.32}
 \begin{aligned}
\mathbb{E} \left[ \left|X^{j,N}_{s}\right|^{2p-2}\mathcal{W} _2\left( \mu _{s}^{X,N} \right) ^2 \right] &\leq \tfrac{p-1}{p}\mathbb{E} \left[ \left|X^{j,N}_{s}\right|^{2p} \right] +\tfrac{1}{p}\mathbb{E} \left[ \mathcal{W} _2\left( \mu _{s}^{X,N} \right) ^{2p} \right] \\
&\leq \tfrac{p-1}{p}\mathbb{E} \left[ \left|X^{j,N}_{s}\right|^{2p} \right] +\tfrac{1}{p}\mathbb{E} \left[ \mathcal{W} _{2p}\left( \mu _{s}^{X,N} \right) ^{2p} \right]\\
&\leq \tfrac{p-1}{p}\mathbb{E} \left[ \left|X^{j,N}_{s}\right|^{2p} \right] +\tfrac{1}{p}\mathbb{E} \left[ \left|X^{j,N}_{s}\right|^{2p} \right]\\
&=\mathbb{E} \left[ \left|X^{j,N}_{s}\right|^{2p} \right].\\
\end{aligned}
\end{equation}
Furthermore, by utilizing (\ref{3.32}) and applying Young's inequality (specifically, (\ref{3.20}) with $l=2p$), we obtain
\begin{equation}\tag{5.22}\label{3.33}
\begin{aligned}
e^{\beta t}\mathbb{E} \left[ \left|X^{j,N}_{t}\right|^{2p} \right] \le &\mathbb{E} \left[ \left| X^{j}_0 \right|^{2p} \right] +\int_0^t{e^{\beta s}\mathbb{E} \left[\left|X^{j,N}_{s}\right|^{2p}\right]}\left[ \beta -a_1p+a_2p+C\varepsilon_{1} \left( p-1 \right) \right] \mathrm{d}s\\
&+\tfrac{C}{\varepsilon_{1} ^{p-1}}\int_0^t{e^{\beta s}ds}.\\
\end{aligned}
\end{equation}
Subsequently, by setting $\varepsilon_{1} =\tfrac{p\left(a_1-a_2\right)}{2\left(p-1 \right) C}$ and $\beta = \tfrac{\left(a_1-a_2\right)p}{2}$, we obtain, for any $1\leq j\leq N$, 
\begin{equation}\tag{5.23}\label{3.34}
\mathbb{E} \left[ \left|X^{j,N}_{t}\right|^{2p} \right] \le C_p\left( \mathbb{E} \left[ \left| X^{j}_0 \right|^{2p} \right] +1 \right).
\end{equation}
Taking the supremum over $t$ and $j$ in (\ref{3.34}) completes the proof of the theorem.
\end{proof}

\subsection{Proof of Theorem \ref{theorem 4.2}}
To derive the long-time mean-square convergence results for one-step methods applied to the MV-SDEs (\ref{MVSDEs}), we need first to establish    the next lemma.
\begin{lemma}\label{lemma 4.1}
    Suppose that Assumptions~\ref{ass2.1}--\ref{ass2.4} hold, and that the joint sequence $\left\{ \left( x^{j,N},y^{j,N} \right) \right\} _{1\leq j\leq N}$ is identically distributed. For $s\ge t$, define 
\begin{equation}\tag{5.24}\label{4.5}
\mathcal{H} _{t,x,y}^{j,N}\left(s\right):=X^{j,N}\left(t,x;s \right)-X^{j,N}\left(t,y;s\right)-x+y\,. 
\end{equation}
Then, for any $t \geq 0$, $1 \leq\kappa\leq q_0$, and $1 \leq j \leq N$,
\begin{equation}\tag{5.25}\label{4.6}
 \underset{1\leqslant j\leqslant N}{\sup}\mathbb{E}\left[\left|X^{j,N}\left(t, x^{j,N} ; t+h\right)-X^{j,N}\left(t, y^{j,N} ; t+h\right)\right|^{2}\right] \leq\mathbb{E}\left[\left|x^{j,N}-y^{j,N}\right|^{2}\right] e^{-\left(K_1-K_2\right) h}, 
 \end{equation}
 \begin{equation}\tag{5.26}\label{4.7}
\underset{1\leqslant j\leqslant N}{\sup}\mathbb{E}\left[\left|\mathcal{H}^{j,N}_{t, x^{j,N}, y^{j,N}}\left(t+h\right)\right|^{2}\right] \leq C_{11}\mathbb{E}\left[\left(1+\left|x^{j,N}\right|^{2 \kappa-2}+\left|y^{j,N}\right|^{2 \kappa-2}\right)^{\frac{1}{2}}\left|x^{j,N}-y^{j,N}\right|^{2}\right] h,
\end{equation}
where $C_{11} $ is a positive constant independent of $j,t, h, N $.
\end{lemma}
\begin{proof}  For $s \geq t$,  denote  \begin{equation}\tag{5.27}\label{4.8}
\mathcal{P}^{j,N}_{t, x^{j,N}, y^{j,N}}\left(s\right):=X^{j,N}\left(t, x^{j,N} ; s\right)-X^{j,N}\left(t,y^{j,N} ; s\right).
\end{equation}
We can write 
\begin{equation}\tag{5.28}\label{4.9}
\mathcal{H}^{j,N}_{t, x^{j,N}, y^{j,N}}\left(s\right)=\mathcal{P}^{j,N}_{t, x^{j,N}, y^{j,N}}\left(s\right)-\left(x^{j,N}-y^{j,N}\right).
\end{equation}
 Moreover, we denote $\mu _{s}^{X,N}:=\tfrac{1}{N}\sum_{j=1}^N{\delta _{X^{j,N}\left(t,x^{j,N};s\right)}}$ and $\nu _{s}^{X,N}:=\tfrac{1}{N}\sum_{j=1}^N{\delta _{X^{j,N}\left(t,y^{j,N};s\right)}}$. Since $X^{j,N}\left(t, x^{j,N} ; s\right)$ and $X^{j,N}\left(t, y^{j,N} ; s\right)$ are driven by the same Brownian motion, it follows that
\begin{equation}\tag{5.29}\label{4.10}
\begin{aligned}
&\mathrm{d} \mathcal{H}^{j,N}_{t, x^{j,N}, y^{j,N}}\left(s\right)=\mathrm{d} \mathcal{P}^{j,N}_{t, x^{j,N}, y^{j,N}}\left(s\right)\\
=&\left(b\left(X^{j,N}\left(t, x^{j,N} ; s\right), \mu _{s}^{X,N}\right)-b\left(X^{j,N}\left(t, y^{j,N} ; s\right), \nu _{s}^{X,N}\right)\right) \mathrm{d} s\\
&+\left(\sigma\left(X^{j,N}\left(t, x^{j,N} ; s\right), \mu _{s}^{X,N}\right)-\sigma\left(X^{j,N}\left(t, y^{j,N} ; s\right), \nu _{s}^{X,N}\right)\right) \mathrm{d} W^{j}\left(s\right).
\end{aligned}
\end{equation}
   Applying It{\^o}'s formula in a manner analogous to the derivation of (\ref{3.27}), and utilizing Assumption \ref{ass2.1} along with the characteristics of the It{\^o} integral, we derive, for all $t \leq s \leq t+h$  and $1\leq j\leq N$,
\begin{equation}\tag{5.30}\label{4.12}
\begin{aligned}
& \mathbb{E}\left[e^{\left(K_1-K_2\right)\left(s-t\right)}\left|\mathcal{P}^{j,N}_{t, x^{j,N}, y^{j,N}}\left(s\right)\right|^{2}\right]-\mathbb{E}\left[\left|\mathcal{P}^{j,N}_{t, x^{j,N}, y^{j,N}}\left(t\right)\right|^{2}\right] \\
\leq & \int_t^s  e^{\left(K_1-K_2\right)\left(\theta-t\right)}\left\{\left(K_1-K_2\right)\mathbb{E}\left[\left|\mathcal{P}^{j,N}_{t, x^{j,N}, y^{j,N}}\left(\theta\right)\right|^{2}\right]\right\}\mathrm{d} \theta\\ 
&+\int_t^s  e^{\left(K_1-K_2\right)\left(\theta-t\right)}\left\{-K_1\mathbb{E}\left[\left|\mathcal{P}^{j,N}_{t, x^{j,N}, y^{j,N}}\left(\theta\right)\right|^{2}\right]+K_2\mathbb{E}\left[\mathcal{W}_2\left(\mu _{\theta}^{X,N},\nu _{\theta}^{X,N}\right)^2\right]\right\}\mathrm{d} \theta\\
\leq & \int_t^s  e^{\left(K_1-K_2\right)\left(\theta-t\right)}\left\{\left(K_1-K_2\right)\mathbb{E}\left[\left|\mathcal{P}^{j,N}_{t, x^{j,N}, y^{j,N}}\left(\theta\right)\right|^{2}\right]\right\}\mathrm{d} \theta\\ 
&+\int_t^s  e^{\left(K_1-K_2\right)\left(\theta-t\right)}\left\{-K_1\mathbb{E}\left[\left|\mathcal{P}^{j,N}_{t, x^{j,N}, y^{j,N}}\left(\theta\right)\right|^{2}\right]+\tfrac{K_2}{N}\sum_{i=1}^N{\mathbb{E} \left[ \left|\mathcal{P}_{t,x^{i,N},y^{i,N}}^{i,N}\left( \theta \right) \right|^2\right]}\right\}\mathrm{d} \theta.
\end{aligned}
\end{equation}
Since the joint sequence $
\left\{ \left( x^{j,N},y^{j,N} \right) \right\} _{1\leq j\leq N}$ is   identically distributed,
the sequence of initial differences $
\left\{  x^{j,N}-y^{j,N}  \right\} _{1\leq j\leq N}$ is also identically distributed. Furthermore, the sequence of processes $\left\{ \left( X^{j,N}\left( t,x^{j,N};s \right) ,X^{j,N}\left( t,y^{j,N};s \right) \right) \right\} _{1\le j\le N}$ satisfies  the same interacting particles system (\ref{interacting}), differing only in their initial values, which is identically distributed. 
As a result, the sequence of difference processes 
 $\left\{ X^{j,N}\left( t,x^{j,N};s \right)-X^{j,N}\left( t,y^{j,N};s \right) \right\} _{1\leq j\leq N}$ is identically distributed. This implies that  the sequence $\left\{ \left|\mathcal{P} _{t,x^{i,N},y^{i,N}}^{i,N}\left( s\right) \right|^2 \right\} _{1 \leq i\leq N}$ is also identically distributed and therefore, for any $1\leq j\leq N$, we have  
\begin{equation}\tag{5.31}\label{4.14}
\begin{aligned}
&\mathbb{E}\left[e^{\left(K_1-K_2\right)\left(s-t\right)}\left|\mathcal{P}^{j,N}_{t, x^{j,N}, y^{j,N}}\left(s\right)\right|^{2}\right]-\mathbb{E}\left[\left|\mathcal{P}^{j,N}_{t, x^{j,N}, y^{j,N}}\left(t\right)\right|^{2}\right] \\
\leq & \int_t^s  e^{\left(K_1-K_2\right)\left(\theta-t\right)}\left[\left(K_1-K_2\right)\mathbb{E}\left[\left|\mathcal{P}^{j,N}_{t, x^{j,N}, y^{j,N}}\left(\theta\right)\right|^{2}\right]\right]\mathrm{d} \theta\\ 
&-\int_t^s  e^{\left(K_1-K_2\right)\left(\theta-t\right)}\left[\left(K_1-K_2\right)\mathbb{E}\left[\left|\mathcal{P}^{j,N}_{t, x^{j,N}, y^{j,N}}\left(\theta\right)\right|^{2}\right]\right]\mathrm{d} \theta\\
\leq&0.
\end{aligned}
\end{equation}
Taking the supremum over $ j $ in (\ref{4.14}) yields
\begin{equation}\tag{5.32}\label{4.15}
\underset{1\leqslant j\leqslant N}{\sup}\mathbb{E}\left[e^{\left(K_1-K_2\right)\left(s-t\right)}\left|\mathcal{P}^{j,N}_{t, x^{j,N}, y^{j,N}}\left(s\right)\right|^{2}\right] \leq \underset{1\leqslant j\leqslant N}{\sup}\mathbb{E}\left[\left|\mathcal{P}^{j,N}_{t, x^{j,N}, y^{j,N}}\left(t\right)\right|^{2}\right]=\mathbb{E}\left[\left|x^{j,N}-y^{j,N}\right|^{2}\right].
\end{equation}
Finally, setting $s=t+h$ completes the proof of inequality (\ref{4.6}).
 
By applying It{\^o}'s formula in conjunction with (\ref{4.9}), we obtain, for $\theta \geq 0$ and $t \geq 0$,
\begin{equation}\tag{5.33}\label{4.16}
\begin{aligned}
&\left|\mathcal{H}^{j,N}_{t, x^{j,N}, y^{j,N}}\left(t+\theta\right)\right|^{2} \\= &2  \int_t^{t+\theta} \left[\left\langle\mathcal{P}^{j,N}_{t, x^{j,N}, y^{j,N}}\left(s\right), b\left(X^{j,N}\left(t, x^{j,N} ;s\right), \mu _{s}^{X,N}\right)-b\left(X^{j,N}\left(t, y^{j,N} ; s\right), \nu _{s}^{X,N}\right)\right\rangle\right. \\
 &+\left.\tfrac{1}{2}\left|\left|\sigma\left(X^{j,N}\left(t, x^{j,N} ;s\right), \mu _{s}^{X,N}\right)-\sigma\left(X^{j,N}\left(t, y^{j,N} ; s\right), \nu _{s}^{X,N}\right)\right|\right|^2\right] \mathrm{d} s \\
 &-2 \int_t^{t+\theta}\left\langle x^{j,N}-y^{j,N},b\left(X^{j,N}\left(t, x^{j,N} ;s\right), \mu _{s}^{X,N}\right)-b\left(X^{j,N}\left(t, y^{j,N} ; s\right), \nu _{s}^{X,N}\right)\right\rangle \mathrm{d} s \\
&+2 \int_t^{t+\theta}\left\langle\mathcal{H}^{j,N}_{t, x^{j,N}, y^{j,N}}\left(s\right),\left(\sigma\left(X^{j,N}\left(t, x^{j,N} ;s\right), \mu _{s}^{X,N}\right)-\sigma\left(X^{j,N}\left(t, y^{j,N} ; s\right), \nu _{s}^{X,N}\right)\right) \mathrm{d} W^{j}\left(s\right)\right\rangle.
\end{aligned}
\end{equation}
Subsequently, noting that the sequence $\left\{ \left|\mathcal{P} _{t,x^{i,N},y^{i,N}}^{i,N}\left( s\right) \right|^2 \right\} _{1 \leq i\leq N}$ is identically distributed, we apply the Cauchy-Schwarz inequality, H{\"o}lder's inequality, Assumption \ref{ass2.1} and \ref{ass2.2}, along with (\ref{2.8}), to establish that for any $1\leq j\leq N$:
\begin{equation}\tag{5.34}\label{4.17}
\begin{aligned}
&\mathbb{E} \left[ \left| \mathcal{H} _{t,x^{j,N},y^{j,N}}^{j,N}\left(t+\theta \right) \right|^2 \right]\\
\le &\int_t^{t+\theta}{\left\{ -K_1\mathbb{E} \left[ \left|\mathcal{P} _{t,x^{j,N},y^{j,N}}^{j,N}\left( s \right) \right|^2 \right] +K_2\mathbb{E} \left[\mathcal{W} _2\left( \mu _{s}^{X,N},\nu _{s}^{X,N} \right) ^2\right] \right\}}ds\\
&-2\int_t^{t+\theta}{\mathbb{E}\left[ \left\langle x^{j,N}-y^{j,N},b\left(X^{j,N}\left(t,x^{j,N};s\right),\mu _{s}^{X,N}\right)-b\left(X^{j,N}\left(t,y^{j,N};s\right),\nu _{s}^{X,N}\right)\right\rangle \right]} \mathrm{d}s\\
\le &\int_t^{t+\theta}{\left\{ -K_1\mathbb{E} \left[ \left|\mathcal{P} _{t,x^{j,N},y^{j,N}}^{j,N}\left( s \right) \right|^2 \right] +\tfrac{K_2}{N}\sum_{i=1}^N{\mathbb{E} \left[ \left|\mathcal{P} _{t,x^{i,N},y^{i,N}}^{i,N}\left(s\right)\right|^2 \right]} \right\}}ds\\
&+2\int_t^{t+\theta}{\mathbb{E}\left[ \left|x^{j,N}-y^{j,N}\right|\left|b\left(X^{j,N}\left(t,x^{j,N};s\right),\mu _{s}^{X,N}\right)-b\left(X^{j,N}\left(t,y^{j,N};s\right),\nu _{s}^{X,N}\right)\right| \right]} \mathrm{d}s\\
\le &\int_t^{t+\theta}{\left\{ -K_1\mathbb{E} \left[ \left|\mathcal{P} _{t,x^{j,N},y^{j,N}}^{j,N}\left( s \right) \right|^2 \right] +K_2\mathbb{E} \left[ \left|\mathcal{P} _{t,x^{j,N},y^{j,N}}^{j,N}\left(s\right)\right|^2 \right] \right\}}ds\\
+&2\sqrt{L_1}\int_t^{t+\theta}\mathbb{E}\left[\left|x^{j,N}-y^{j,N}\right|\left\{ \left(1+\left|X^{j,N}\left(t,x^{j,N};s\right)\right|^{2\kappa -2}+\left|X^{j,N}\left(t,y^{j,N};s\right)\right|^{2\kappa -2}\right) \right.\right.\\
&\times\left. \left.\left| \mathcal{P} _{t,x^{j,N},y^{j,N}}^{j,N}\left( s \right) \right|^2+\mathcal{W} _2\left( \mu _{s}^{X,N},\nu _{s}^{X,N} \right) ^2 \right\} ^{\frac{1}{2}}\right]\mathrm{d}s\\
\le &2\sqrt{L_1}\int_t^{t+\theta}\mathbb{E}\left[\left|x^{j,N}-y^{j,N}\right|\mathbb{E} \left[ \left(1+\left|X^{j,N}\left(t,x^{j,N};s\right)\right|^{2\kappa -2}+\left|X^{j,N}\left(t,y^{j,N};s\right)\right|^{2\kappa -2}\right)^{\frac{1}{2}}\right. \right.\\
&\times\left. \left. \left. \left| \mathcal{P} _{t,x^{j,N},y^{j,N}}^{j,N}\left( s \right) \right|\right\rvert\ \mathcal{F} _t  \right] \right] +\mathbb{E} \left[ \left|x^{j,N}-y^{j,N}\right|\mathcal{W} _2\left( \mu _{s}^{X,N},\nu _{s}^{X,N} \right) \right] \mathrm{d}s\\
\le&2\sqrt{L_1}\int_t^{t+\theta}\mathbb{E}\left[\left|x^{j,N}-y^{j,N}\right|\left\{ \mathbb{E} \left[\left.\left(1+\left|X^{j,N}\left(t,x^{j,N};s\right)\right|^{2\kappa -2}+\left|X^{j,N}\left(t,y^{j,N};s\right)\right|^{2\kappa -2}\right)\right\rvert\ \mathcal{F} _t \right] \right\} ^{\frac{1}{2}}\right.\\
&\times\left.  \left\{ \mathbb{E} \left[\left. \left| \mathcal{P} _{t,x^{j,N},y^{j,N}}^{j,N}\left( s \right) \right|^2 \right\rvert\ \mathcal{F} _t  \right] \right\} ^{\frac{1}{2}} \right] +\left\{ \mathbb{E} \left[ \left|x^{j,N}-y^{j,N}\right|^2 \right] \right\} ^{\frac{1}{2}}\left\{ \mathbb{E} \left[ \left| \mathcal{P} _{t,x^{j,N},y^{j,N}}^{j,N}\left( s \right) \right|^2 \right] \right\} ^{\frac{1}{2}}\mathrm{d}s.\\
\end{aligned}
\end{equation}
From the above  (\ref{4.17}), along with the conditional version of (\ref{3.25}) and (\ref{4.14}), and (\ref{4.14}), it follows that
\begin{equation}\tag{5.35}\label{4.18}
\begin{aligned}
&\mathbb{E} \left[ \left| \mathcal{H} _{t,x^{j,N},y^{j,N}}^{j,N}\left(t+\theta \right) \right|^2 \right]\\
\le& C^{\prime}\int_t^{t+\theta}{\mathbb{E} \left[ \left|x^{j,N}-y^{j,N}\right|\left( 1+\left|x^{j,N}\right|^{2\kappa -2}+\left|y^{j,N}\right|^{2\kappa -2} \right) ^{\frac{1}{2}}\left|x^{j,N}-y^{j,N}\right| \right] e^{ -\frac{K_1-K_2}{2}\left( s-t \right) }}\\
&+\mathbb{E} \left[ \left|x^{j,N}-y^{j,N}\right|^2 \right] e^{ -\frac{K_1-K_2}{2}\left( s-t \right) }\mathrm{d}s\\
\le& \bar{C}\mathbb{E} \left[ \left|x^{j,N}-y^{j,N}\right|^2\left( 1+\left|x^{j,N}\right|^{2\kappa -2}+\left|y^{j,N}\right|^{2\kappa -2} \right) ^{\frac{1}{2}} \right] \int_t^{t+\theta}{e^{ -\frac{K_1-K_2}{2}\left( s-t \right) }}\mathrm{d}s\\
\le& C_{11}\mathbb{E} \left[ \left( 1+\left| x^{j,N} \right|^{2\kappa -2}+\left| y^{j,N} \right|^{2\kappa -2} \right) ^{\frac{1}{2}}\left| x^{j,N}-y^{j,N} \right|^2 \right] \theta.\\
\end{aligned}
\end{equation}
Setting $\theta=h$ and taking the supremum over $j$  establishes inequality (\ref{4.7}),
thereby Lemma \ref{lemma 4.1} is proved.  
\end{proof} 
We can now complete the proof of Theorem \ref{theorem 4.2}.
\begin{proof}[ Proof of Theorem \ref{theorem 4.2}]
To emphasize the reliance on the initial condition, we denote the solution to the interacting particle system (\ref{interacting}) by $X^{j,N}\left(t_0, X^{j,N}_0 ; t_0+t\right)$. The error between the exact solution and its numerical approximation is expressed as 
\begin{equation}\tag{5.36}\label{4.23}
\begin{aligned}
\rho^{j,N}_{k+1}: & =X^{j,N}\left(t_0, X^{j,N}_0 ; t_{k+1}\right)-Y^{j,N}\left(t_0, X^{j,N}_0 ; t_{k+1}\right) \\
 &=X^{j,N}\left(t_k, X^{j,N}_k ; t_{k+1}\right)-Y^{j,N}\left(t_k, Y^{j,N}_k ; t_{k+1}\right) \\
 &=X^{j,N}\left(t_k, X^{j,N}_k ; t_{k+1}\right)-X^{j,N}\left(t_k, Y^{j,N}_k ; t_{k+1}\right)\\
&\qquad\quad +X^{j,N}\left(t_k, Y^{j,N}_k ; t_{k+1}\right)-Y^{j,N}\left(t_k, Y^{j,N}_k ; t_{k+1}\right) \\
&: =\mathcal{P}^{j,N}_{t_k, X^{j,N}_k, Y^{j,N}_k}\left(t_{k+1}\right)+\mathcal{V}^{j,N}_{k+1}.
\end{aligned}
\end{equation}
We note that the first component on the right side of (\ref{4.23}) represents the difference between the exact solutions  at time $t_{k+1}$ resulting from distinct initial conditions at time $t_k$. This term can be further decomposed as follows: 
\begin{equation}\tag{5.37}\label{4.24}
\begin{aligned}
\mathcal{P}^{j,N}_{t_k, X^{j,N}_k, Y^{j,N}_k}\left(t_{k+1}\right): & =X^{j,N}\left(t_k, X^{j,N}_k ; t_{k+1}\right)-X^{j,N}\left(t_k, Y^{j,N}_k ; t_{k+1}\right) \\
& :=\rho^{j,N}_k+\mathcal{H}^{j,N}_{t_k, X^{j,N}_k, Y^{j,N}_k}\left(t_{k+1}\right),
\end{aligned}
\end{equation}
where $\mathcal{H}$ is as defined in (\ref{4.5}).
The second component on the right side of (\ref{4.23}) represents the one-step error at the $\left(k+1\right)$-th step, which takes the following form: 
\begin{equation}\tag{5.38}\label{4.25}
\mathcal{V}^{j,N}_{k+1}:=X^{j,N}\left(t_k, Y^{j,N}_k ; t_{k+1}\right)-Y^{j,N}\left(t_k, Y^{j,N}_k ; t_{k+1}\right).
\end{equation}
Therefore, 
\begin{equation}\tag{5.39}\label{4.26}
\begin{aligned}
 \underset{1\leqslant j\leqslant N}{\sup}\mathbb{E}\left[\left|\rho^{j,N}_{k+1}\right|^{2}\right] 
= & \underset{1\leqslant j\leqslant N}{\sup}\mathbb{E}\left[\left|\mathcal{P}^{j,N}_{t_k, X^{j,N}_k, Y^{j,N}_k}\left(t_{k+1}\right)+\mathcal{V}^{j,N}_{k+1}\right|^{2}\right] \\
= &\underset{1\leqslant j\leqslant N}{\sup} \mathbb{E}\left[\left|\mathcal{P}^{j,N}_{t_k, X^{j,N}_k, Y^{j,N}_k}\left(t_{k+1}\right)\right|^2+2\left\langle\mathcal{P}^{j,N}_{t_k, X^{j,N}_k, Y^{j,N}_k}\left(t_{k+1}\right), \mathcal{V}^{j,N}_{k+1}\right\rangle+\left|\mathcal{V}^{j,N}_{k+1}\right|^2\right] \\
\leq & \underset{1\leqslant j\leqslant N}{\sup}\mathbb{E}\left[\left|\mathcal{P}^{j,N}_{t_k, X^{j,N}_k, Y^{j,N}_k}\left(t_{k+1}\right)\right|^{2}\right]+2\underset{1\leqslant j\leqslant N}{\sup}\mathbb{E}\left[\left\langle\rho^{j,N}_k, \mathcal{V}^{j,N}_{k+1}\right\rangle\right]\\
&\qquad  +2\underset{1\leqslant j\leqslant N}{\sup}\mathbb{E}\left[\left\langle\mathcal{H}^{j,N}_{t_k, X^{j,N}_k, Y^{j,N}_k}\left(t_{k+1}\right), \mathcal{V}^{j,N}_{k+1}\right\rangle\right]+\underset{1\leqslant j\leqslant N}{\sup}\mathbb{E}\left[\left|\mathcal{V}^{j,N}_{k+1}\right|^2\right].
\end{aligned}
\end{equation}
The sequence of exact solutions at the $k$-th step of the interacting particle system, $\left\{ X_{k}^{j,N} \right\} _{1\leq j\leq N}$,
and the corresponding sequence of one-step approximations, $\left\{ Y_{k}^{j,N} \right\} _{1\leq j\leq N}$, are identically distributed. Since the symmetry inherent in the construction of the one-step approximation for the interacting particle system, which preserves the system's overall symmetry. For instance, the projected Euler and backward Euler schemes, as time-discretization methods discussed later, exhibit this symmetry. Consequently, the joint sequences $\left\{ \left( X_{k}^{j,N},Y_{k}^{j,N} \right) \right\} _{1\leq j\leq N}$ are also identically distributed.\\
For the first component on the right side of (\ref{4.26}),  
we derive 
\begin{equation}\tag{5.40}\label{4.27}
\underset{1\leqslant j\leqslant N}{\sup}\mathbb{E}\left[\left|\mathcal{P}^{j,N}_{t_k, X^{j,N}_k, Y^{j,N}_k}\left(t_{k+1}\right)\right|^{2}\right] \leq \mathbb{E}\left[\left|\rho^{j,N}_k\right|^{2}\right] e^{-\left(K_1-K_2\right)h}.
\end{equation}
The conditional version of (\ref{4.19}) takes the following form:
\begin{equation}\tag{5.41}\label{4.28}
\begin{aligned}
 \left|\mathbb{E} \left[\left.\mathcal{V}^{j,N} _{k+1}\right\rvert\ \mathcal{F} _{t_k} \right]\right|&=\left|\mathbb{E} \left[\left. X^{j,N}\left( t_k,Y_{k}^{j,N};t_{k+1} \right) -Y^{j,N}\left( t_k,Y_{k}^{j,N};t_{k+1} \right) \right\rvert\ \mathcal{F} _{t_k} \right]\right|\\
 &\leq C_2 \left( 1+\left|Y_{k}^{j,N}\right|^{\eta _1} \right)  h^{q_1}.
 \end{aligned}
\end{equation}
Due to the $\mathcal{F}_{t_k}$-measurability of $\rho^{j,N}_k$, along with the conditional version of (\ref{4.19}) and the application of the inner product inequality, we derive the following result for the second component on the right side of (\ref{4.26}):
\begin{equation}\tag{5.42}\label{4.29}
\begin{aligned}
2\underset{1\leqslant j\leqslant N}{\sup}\mathbb{E} \left[ \left \langle \rho^{j,N} _k,\mathcal{V}^{j,N} _{k+1} \right \rangle \right] 
&= 2\underset{1\leqslant j\leqslant N}{\sup}\mathbb{E} \left[ \mathbb{E} \left[\left. \left\langle \rho^{j,N} _k,\mathcal{V}^{j,N} _{k+1} \right \rangle \right\rvert\ \mathcal{F} _{t_k} \right] \right]\\
&\leq 2\underset{1\leqslant j\leqslant N}{\sup}\mathbb{E} \left[ \left |\mathbb{E} \left[ \left.\rho^{j,N} _k\mathcal{V}^{j,N} _{k+1}\right\rvert\ \mathcal{F} _{t_k} \right] \right|\right] \\
&\leq 2\underset{1\leqslant j\leqslant N}{\sup}\mathbb{E} \left[\left|\rho ^{j,N}_k\right|\left| \mathbb{E} \left[\left. \mathcal{V}^{j,N} _{k+1}\right\rvert\ \mathcal{F} _{t_k} \right]\right |\right] \\
&\leq 2C_2\underset{1\leqslant j\leqslant N}{\sup}\mathbb{E} \left[ \left| \rho^{j,N} _k \right|\left( 1+\left| Y_{k}^{j,N} \right|^{\eta_1} \right) \right] h^{q_1}.
\end{aligned}
\end{equation}
Subsequently, by applying the conditional version of (\ref{4.20}), we obtain 
\begin{equation}\tag{5.43}\label{4.30}
\begin{aligned}
\left\{ \mathbb{E} \left[\left. \left|\mathcal{V}^{j,N} _{k+1} \right|^2\right\rvert\ \mathcal{F} _{t_k} \right] \right\} ^{\frac{1}{2}}&=\left\{ \mathbb{E} \left[ \left.\left| X^{j,N}\left( t_k,Y_{k}^{j,N};t_{k+1} \right) -Y^{j,N}\left( t_k,Y_{k}^{j,N};t_{k+1} \right) \right|^2\right\rvert\ \mathcal{F} _{t_k} \right] \right\} ^{\frac{1}{2}}\\
&\leq C_3 \left( 1+\left| Y_{k}^{j,N} \right|^{2\eta _2}  \right) ^{\frac{1}{2}}h^{q_2}\,. 
\end{aligned}
\end{equation}
The conditional version of (\ref{4.7}), given by 
 \begin{equation}\tag{5.44}\label{4.31}
\left\{ \mathbb{E} \left[\left. \left| \mathcal{H} _{t_k,X_{k}^{j,N},Y_{k}^{j,N}}^{j,N}\left(t_{k+1}\right) \right|^2\right\rvert\ \mathcal{F} _{t_k} \right] \right\} ^{\frac{1}{2}}\le {C_{11}}^{\frac{1}{2}} \left( 1+\left|X_{k}^{j,N}\right|^{2\kappa -2}+\left|Y_{k}^{j,N}\right|^{2\kappa -2} \right) ^{\frac{1}{4}}\left|\rho^{j,N} _k\right|  h^{\frac{1}{2}},
\end{equation}
together with H\"{o}lder's inequality, implies 
\begin{equation}\tag{5.45}\label{4.32}
\begin{aligned}
 2 \underset{1\leqslant j\leqslant N} {\sup}\mathbb{E}&\left[\left\langle\mathcal{H}^{j,N}_{t_k, X^{j,N}_k, Y^{j,N}_k}\left(t_{k+1}\right), \mathcal{V}^{j,N}_{k+1}\right\rangle\right] \\
= & 2  \underset{1\leqslant j\leqslant N}{\sup}\mathbb{E}\left[\mathbb{E}\left[\left.\left\langle\mathcal{H}^{j,N}_{t_k, X^{j,N}_k, Y^{j,N}_k}\left(t_{k+1}\right), \mathcal{V}^{j,N}_{k+1}\right\rangle \right\rvert\ \mathcal{F}_{t_k}\right]\right] \\
\leq & 2 \underset{1\leqslant j\leqslant N}{\sup} \mathbb{E}\left[\left\{\mathbb{E}\left[\left.\left|\mathcal{H}^{j,N}_{t_k, X^{j,N}_k, Y^{j,N}_k}\left(t_{k+1}\right)\right|^{2} \right\rvert\ \mathcal{F}_{t_k}\right]\right\}^{\frac{1}{2}}\left\{\mathbb{E}\left[\left.\left|\mathcal{V}^{j,N}_{k+1}\right|^{2} \right\rvert\ \mathcal{F}_{t_k}\right]\right\}^{\frac{1}{2}}\right] \\
\leq & 2C^{\frac{1}{2}}_{11}C_3\underset{1\leqslant j\leqslant N}{\sup}\mathbb{E}\left[\left|\rho^{j,N}_k\right|\left(1+\left|X^{j,N}_k\right|^{2 \kappa-2}+\left|Y^{j,N}_k\right|^{2 \kappa-2}\right)^{\frac{1}{4}}\left(1+\left|Y^{j,N}_k\right|^{2\eta_2}\right)^{\frac{1}{2}}\right] h^{q_2+\frac{1}{2}}.
\end{aligned}
\end{equation}
Moreover,
\begin{equation}\tag{5.46}\label{4.33}
\underset{1\leqslant j\leqslant N}{\sup}\mathbb{E} \left[\left |\mathcal{V}^{j,N}_{k+1}\right|^2 \right] \le C^2_{3}\underset{1\leqslant j\leqslant N}{\sup}\mathbb{E} \left[ \left( 1+\left| Y_{k}^{j,N} \right|^{2\eta _2} \right) \right] h^{2q_2}.
\end{equation}
Substituting (\ref{4.27}), (\ref{4.29}), (\ref{4.32}) and (\ref{4.33}) into (\ref{4.26}) yields 
\begin{equation}\tag{5.47}\label{4.34}
\begin{aligned}
 \underset{1\leqslant j\leqslant N}{\sup}\mathbb{E}&\left[\left|\rho^{j,N}_{k+1}\right|^{2}\right]\\ 
\leq & \underset{1\leqslant j\leqslant N}{\sup}\mathbb{E}\left[\left|\rho^{j,N}_k\right|^{2}\right] e^{-\left(K_1-K_2\right)h} +2C_2 \underset{1\leqslant j\leqslant N}{\sup}\mathbb{E}\left[\left|\rho^{j,N}_k\right|\left(1+\left|Y^{j,N}_k\right|^{\eta_1}\right)\right] h^{q_1} \\
& +2C^{\frac{1}{2}}_{11}C_3\underset{1\leqslant j\leqslant N}{\sup}\mathbb{E}\left[\left|\rho^{j,N}_k\right|\left(1+\left|X^{j,N}_k\right|^{2 \kappa-2}+\left|Y^{j,N}_k\right|^{2 \kappa-2}\right)^{\frac{1}{4}} \left(1+\left|Y^{j,N}_k\right|^{2\eta_2}\right)^{\frac{1}{2}}\right]h^{q_2+\frac{1}{2}} \\
&+C^2_{3}\underset{1\leqslant j\leqslant N}{\sup}\mathbb{E} \left[ \left( 1+\left| Y_{k}^{j,N} \right|^{2\eta _2} \right) \right] h^{2q_2}.
\end{aligned}
\end{equation}
Next, by applying Young's inequality $ab\leq \varepsilon_{2} \tfrac{a^2}{2}+\tfrac{1}{\varepsilon_{2}}\tfrac{b^2}{2}$ to (\ref{4.34}) and noting that $q_1 \geq q_2+\tfrac{1}{2}$, we derive 
\begin{equation}\tag{5.48}\label{4.35}
\begin{aligned}
 \underset{1\leqslant j\leqslant N}{\sup}\mathbb{E}&\left[\left|\rho^{j,N}_{k+1}\right|^{2}\right]\\
\leq & \underset{1\leqslant j\leqslant N}{\sup}\mathbb{E}\left[\left|\rho^{j,N}_k\right|^{2}\right] e^{-\left(K_1-K_2\right)h} \\
&+\tfrac{\left(K_1-K_2\right) h}{4} \underset{1\leqslant j\leqslant N}{\sup}\mathbb{E}\left[\left|\rho^{j,N}_k\right|^2\right]+K^{*} \underset{1\leqslant j\leqslant N}{\sup}\mathbb{E}\left[\left(1+\left|Y^{j,N}_k\right|^{\eta_1}\right)^{2}\right] h^{2q_2} \\
&+K^{*} \underset{1\leqslant j\leqslant N}{\sup}\mathbb{E}\left[\left(1+\left|X^{j,N}_k\right|^{2 \kappa-2}+\left|Y^{j,N}_k\right|^{2 \kappa-2}\right)^{\frac{1}{2}}\left(1+\left|Y^{j,N}_k\right|^{2\eta_2}\right)\right] h^{2 q_2} \\
&+\tfrac{\left(K_1-K_2\right) h}{4} \underset{1\leqslant j\leqslant N}{\sup}\mathbb{E}\left[\left|\rho^{j,N}_k\right|^2\right] +K^{*} \underset{1\leqslant j\leqslant N}{\sup}\mathbb{E}\left[\left(1+\left|Y^{j,N}_k\right|^{2\eta_2}\right)\right] h^{2q_2},\\
\end{aligned}
\end{equation}
where $K^{*}$ is a positive constant independent of $h,N, k$. \ Subsequently, under the assumption $0<h \leq \tfrac{1}{2\left(K_1-K_2\right)}$ and utilizing the inequality $e^{-x} \leq 1-x+\tfrac{x^2}{2}$ for $0<x<1$, we obtain
\begin{equation}\tag{5.49}\label{4.36}
e^{-\left( K_1-K_2 \right) h}\leq 1-\left( K_1-K_2 \right) h+\tfrac{\left( K_1-K_2 \right) ^2}{2}\tfrac{1}{2\left( K_1-K_2 \right)}h= 1-\tfrac{3}{4}\left( K_1-K_2 \right) h.
\end{equation}
 Finally, by combining (\ref{3.25}), (\ref{4.21}), and (\ref{4.36}), we conclude that
\begin{equation}\tag{5.50}\label{4.37}
\begin{aligned} 
&\underset{1\leqslant j\leqslant N}{\sup}\mathbb{E}\left[\left|\rho^{j,N}_{k+1}\right|^{2}\right] \\
\leq & \left\{1-\tfrac{3\left(K_1-K_2\right)}{4}h+\tfrac{\left(K_1-K_2\right)}{2}h\right\} \underset{1\leqslant j\leqslant N}{\sup}\mathbb{E}\left[\left|\rho^{j,N}_k\right|^{2}\right]+K^{*} \underset{1\leqslant j\leqslant N}{\sup}\mathbb{E}\left[\left(1+\left|X_0^{j}\right|^{2\lambda}\right)\right] h^{2q_2} \\
\leq & \left\{1-\tfrac{\left(K_1-K_2\right)}{4}h\right\} \underset{1\leqslant j\leqslant N}{\sup}\mathbb{E}\left[\left|\rho^{j,N}_k\right|^{2}\right]+K^{*} \underset{1\leqslant j\leqslant N}{\sup}\mathbb{E}\left[\left(1+\left|X_0^{j}\right|^{2 \lambda}\right)\right] h^{2q_2},
\end{aligned}
\end{equation}
where $\lambda=\max \left\{\eta_1\eta_3 ,\left(\tfrac{\kappa-1}{2}+\eta_2\right)\eta_3 \right\}$.
On the other hand, we obtain \begin{equation}\tag{5.51}\label{4.38}
\underset{1\leqslant j\leqslant N}{\sup}\mathbb{E}\left[\left|\rho^{j,N}_k\right|^{2}\right] \leq C^{2}_5\underset{1\leqslant j\leqslant N}{\sup}\mathbb{E}\left[\left(1+\left|X_0^{j}\right|^{2 \lambda}\right)\right] h^{2\left(q_2-\frac{1}{2}\right)}=C^{2}_5\mathbb{E}\left[\left(1+\left|X_0\right|^{2 \lambda}\right)\right] h^{2\left(q_2-\frac{1}{2}\right)},
\end{equation}
where the constant $C_5$ is independent of $h,k,N$. 
Combining this with \eqref{4.37} yields  
\begin{equation}\tag{5.52}\label{4.39}
\underset{1\leqslant j\leqslant N}{\sup}\left\{\mathbb{E}\left[\left|X^{j,N}_k-Y^{j,N}_k\right|^{2}\right]\right\}^{\frac{1}{2}} \leq C_{5}\left\{\mathbb{E}\left[\left(1+\left|X_0\right|^{2 \lambda}\right)\right]\right\}^{\frac{1}{2}} h^{q_2-\frac{1}{2}}.
\end{equation}
Consequently, we deduce (\ref{4.22}). Theorem~\ref{theorem 4.2} is thus proved. 
\end{proof} 
\subsection{Proof of Theorem \ref{momentboundsforPEM}}
For this part, we establish the boundedness of high-order moments for the projected Euler method. For this purpose, we first present the following lemma:
\begin{lemma}\label{lemma5.2} 
Consider a sequence of random variables $\left\{Z_k\right\}_{k \in \mathbb{N}}$. Suppose there are positive constants $m$ and $n$ such that $1-mh>0$  and
\begin{equation}\tag{5.53}\label{5.5}
\mathbb{E}\left[\left|Z_k\right|^{2 p}\right] \leq\left(1-m h\right)\mathbb{E}\left[\left|Z_{k-1}\right|^{2 p}\right]+n h\mathbb{E}\left[\left|Z_{k-1}\right|^{2 p-2}\right]
\end{equation}
holds for all $p \geq 1$ and $k \in \mathbb{N}$. Then,
\begin{equation}\tag{5.54}\label{5.6}
\mathbb{E}\left[\left|Z_k\right|^{2 p}\right] \leq C_{12}\mathbb{E}\left[\left(1+\left|Z_0\right|^{2 p}\right)\right],
\end{equation}
where $Z_0$ is the initial value and $C_{12}$ is a constant depending only on $p, m, n$.
\end{lemma}
\begin{proof} We prove (\ref{5.6}) by induction on $p$.  For $p = 1$, inequality (\ref{5.5}) yields 
\begin{equation}\tag{5.55}\label{5.7}
\begin{aligned}
\mathbb{E} \left[ \left| Z_k \right|^2 \right] &\le \left(1-mh\right)^k\mathbb{E} \left[ \left| Z_0 \right|^2 \right] +\left( 1+\left(1-mh\right)+...+\left(1-mh\right)^{k-1} \right) nh
\\
&=\left(1-mh\right)^k\mathbb{E} \left[ \left| Z_0 \right|^2 \right] +\tfrac{1-\left(1-mh\right)^k}{mh}nh
\\
&\leq C_{13}\mathbb{E} \left[ \left( 1+\left| Z_0 \right|^2 \right) \right]. \\
\end{aligned}
\end{equation}
Assume now that, for some $p_0 \geq 1$, inequality (\ref{5.6}) holds for any $p \leq p_0$ and all $k \in \mathbb{N}$; that is,
\begin{equation}\tag{5.56}\label{5.8}
\mathbb{E}\left[\left|Z_k\right|^{2 p}\right] \leq C_{12}\mathbb{E}\left[\left(1+\left|Z_0\right|^{2 p}\right)\right]\,. 
\end{equation}
For $p_0\le p\le p_0+1$, notice that
\begin{equation}\tag{5.57}\label{5.9}
\mathbb{E} \left[ \left|Z_{k-1}\right|^{2p-2} \right]\le 
C_{12} \mathbb{E} \left [ 1+\left|Z_0\right|^{2p-2}\right]\le
C_{12}' \mathbb{E} \left [ 1+\left|Z_0\right|^{2p}\right]\,. 
\end{equation}
Combining (\ref{5.5}) and (\ref{5.9}), we obtain
\begin{equation}\tag{5.58}\label{5.10}
\begin{aligned}
 \mathbb{E} \left[ \left| Z_k \right|^{2p} \right]& \le \left(1-mh\right)\mathbb{E} \left[ \left| Z_{k-1} \right|^{2p} \right] +nh\mathbb{E} \left[ \left| Z_{k-1} \right|^{2p-2} \right] 
\\
&\le \left(1-mh\right)\mathbb{E} \left[ \left| Z_{k-1} \right|^{2p} \right] +C_{13}nh\mathbb{E} \left[ \left( 1+\left| Z_0 \right|^{2p} \right) \right] 
\\
&\le \left(1-mh\right)^k\mathbb{E} \left[ \left| Z_0 \right|^{2p} \right] +C_{13}nh\left( 1+\left(1-mh\right)+...+\left(1-mh\right)^{k-1} \right) \mathbb{E} \left[ \left( 1+\left| Z_0 \right|^{2p} \right) \right] 
\\
&=\left(1-mh\right)^k\mathbb{E} \left[ \left| Z_0 \right|^{2p}\right] +C_{13}nh\tfrac{1-\left(1-mh\right)^k}{mh}\mathbb{E} \left[ \left( 1+\left| Z_0 \right|^{2p} \right) \right] 
\\
&\leq C_{14}\mathbb{E} \left[ \left( 1+\left| Z_0 \right|^{2p}\right) \right].\\
\end{aligned}
\end{equation}
Therefore, by mathematical induction, we conclude that for all $p \in \mathbb{N}$,
\begin{equation}\tag{5.59}\label{5.11}
\mathbb{E}\left[\left|Z_k\right|^{2 p}\right] \leq C_{12}\mathbb{E}\left[\left(1+\left|Z_0\right|^{2 p}\right)\right],
\end{equation}
where $C_{12}$ depends only on $p, m, n$.
\end{proof}

We now return to our scheme. By applying (\ref{2.6}), (\ref{2.15}), and (\ref{2.16}), we obtain:
\begin{equation}\tag{5.60}\label{5.12}
\begin{aligned}
 \left|b\left(\bar{Y}^{j,N}_{k-1},\mu^{\bar{Y}^{j,N}_{k-1},N}_{k-1}\right)\right|^2&\leq K_4\left( \left|\bar{Y}^{j,N}_{k-1}\right|^{2\kappa}+1 \right) +a_3\mathcal{W}_2\left(\mu^{\bar{Y}^{j,N}_{k-1},N}_{k-1}\right) ^2 \\
&= K_4\left( \left|\bar{Y}^{j,N}_{k-1}\right|^{2\kappa}+1 \right) +\tfrac{a_3}{N}\sum_{j=1}^N{\left|\bar{Y}^{j,N}_{k-1}\right|^2},
\end{aligned}
\end{equation}
\begin{equation}\tag{5.61}\label{5.13}
\begin{aligned}
\left|\left| \sigma\left(\bar{Y}^{j,N}_{k-1},\mu^{\bar{Y}^{j,N}_{k-1},N}_{k-1}\right) \right| \right|^2&\leq K_5\left( \left|\bar{Y}^{j,N}_{k-1}\right|^{\kappa+1}+1 \right) +a_4\mathcal{W}_2\left(\mu^{\bar{Y}^{j,N}_{k-1},N}_{k-1}\right) ^2 \\
&= K_5\left( \left|\bar{Y}^{j,N}_{k-1}\right|^{\kappa+1}+1 \right) +\tfrac{a_4}{N}\sum_{j=1}^N{\left|\bar{Y}^{j,N}_{k-1}\right|^2}.
\end{aligned}
\end{equation}

Taking the supremum over $ j $ in (\ref{5.12}) and using (\ref{5.2}), we obtain
\begin{equation}\tag{5.62}\label{5.14}
\begin{aligned}
\underset{1\leqslant j\leqslant N}{\sup}\left|b\left(\bar{Y}^{j,N}_{k-1},\mu^{\bar{Y}^{j,N}_{k-1},N}_{k-1}\right)\right|^2
&\leq K_4\left(\underset{1\leqslant j\leqslant N}{\sup} \left|\bar{Y}^{j,N}_{k-1}\right|^{2\kappa}+1 \right) +a_3\underset{1\leqslant j\leqslant N}{\sup}\left|\bar{Y}^{j,N}_{k-1}\right|^2\\
&= K_4\left(\underset{1\leqslant j\leqslant N}{\sup}\left|\psi \left(Y_{k-1}^{j,N}\right)\right|^{2\kappa}+1 \right) +a_3\underset{1\leqslant j\leqslant N}{\sup}\left|\psi \left(Y_{k-1}^{j,N}\right)\right|^2\\
&\leq K_4\left(h^{-\frac{2\kappa}{2\left( \kappa +2 \right)}}+1\right)+a_3h^{-\frac{1}{ \kappa +2 }}\\
&\leq K_{12} h^{-1}+a_5.
\end{aligned}
\end{equation}

Taking the supremum over $ j $ in (\ref{5.13}), and utilizing 
(\ref{5.2}), we obtain
\begin{equation}\tag{5.63}\label{5.15}
\begin{aligned}
\underset{1\leqslant j\leqslant N}{\sup}\left|\left|\sigma\left(\bar{Y}^{j,N}_{k-1},\mu^{\bar{Y}^{j,N}_{k-1},N}_{k-1}\right)\right| \right|^2
&\leq K_5\left( \underset{1\leqslant j\leqslant N}{\sup}\left|\bar{Y}^{j,N}_{k-1}\right|^{\kappa+1}+1 \right) +a_4\underset{1\leqslant j\leqslant N}{\sup}\left|\bar{Y}^{j,N}_{k-1}\right|^2\\
&= K_5\left( \underset{1\leqslant j\leqslant N}{\sup}\left|\psi \left(Y_{k-1}^{j,N}\right)\right|^{\kappa+1}+1 \right) +a_4\underset{1\leqslant j\leqslant N}{\sup}\left|\psi \left(Y_{k-1}^{j,N}\right)\right|^2\\
&\leq K_5\left(h^{-\frac{\kappa+1}{2\left( \kappa +2 \right)}}+1\right)+a_4h^{-\frac{1}{ \kappa +2 }}\\
&\leq K_{13} h^{-\frac{1}{2}}+a_6.
\end{aligned}
\end{equation}
 
The preceding arguments will yield the moment bounds for the projected Euler scheme. At this point, we are able to  finalize the proof of Theorem~\ref{momentboundsforPEM}.

\begin{proof}[Proof of Theorem~\ref{momentboundsforPEM}]
For $p=1$, by applying Assumption \ref{ass2.4}, along with (\ref{2.6}), (\ref{5.1}), (\ref{5.4}), and (\ref{5.14}), and exploiting the properties of Brownian motion, we obtain, for all $1\leq j\leq N$,\\
\begin{equation}\tag{5.64}\label{5.17}
\begin{aligned}
&\mathbb{E} \left[\left| Y_{k}^{j,N} \right|^2\right] \\
\leq& \mathbb{E} \left[\left| \bar{Y}_{k-1}^{j,N} \right|^2\right]+h^{2}\mathbb{E} \left[\left|b\left(\bar{Y}^{j,N}_{k-1},\mu^{\bar{Y}^{j,N}_{k-1},N}_{k-1}\right) \right|^2\right]\\
&+h\mathbb{E} \left[2\left \langle \bar{Y}_{k-1}^{j,N},b\left( \bar{Y}_{k-1}^{j,N},\mu _{k-1}^{\bar{Y}_{k-1}^{j,N},N} \right) \right \rangle +\left|\left| \sigma \left( \bar{Y}_{k-1}^{j,N},\mu _{k-1}^{\bar{Y}_{k-1}^{j,N},N} \right)  \right|\right|^2\right]\\
\le&  \mathbb{E} \left[\left| \bar{Y}_{k-1}^{j,N} \right|^2\right]+h^{2}\left(K_{12} h^{-1}+a_5\right)\\
&+h\left\{ -a_1\mathbb{E} \left[\left|\bar{Y}_{k-1}^{j,N}\right|^2\right]+a_2\mathbb{E} \left[\mathcal{W}_2\left( \mu _{k-1}^{\bar{Y}_{k-1}^{j,N},N} \right) ^2\right]+C \right\}\\
 \leq&\left(1-\left(a_1-a_2\right) h\right)\mathbb{E}\left[\left|\bar{Y}^{j,N}_{k-1}\right|^2\right]+K_{14} h \\
\leq&\left(1-\left(a_1-a_2\right) h\right)\mathbb{E}\left[\left|Y_{k-1}^{j,N}\right|^2\right]+K_{14} h.\\
\end{aligned}
\end{equation}
By applying recursion and subsequently taking the supremum over $ j $, we conclude that inequality (\ref{5.16}) holds for the case $p=1$.

For any positive integer  $p \geq 2$,  applying (\ref{5.1}) and (\ref{5.14}) yields
\begin{equation}\tag{5.65}\label{5.18}
\begin{aligned}
&\mathbb{E} \left[\left. \left( 1+\left| Y_{k}^{j,N} \right|^2 \right) ^p\right\rvert\ \mathcal{F} _{t_{k-1}} \right] \\
\leq& \mathbb{E} \left[ \left( 1+\left| \bar{Y}_{k-1}^{j,N} \right|^2+2\left \langle \bar{Y}_{k-1}^{j,N},b\left( \bar{Y}_{k-1}^{j,N},\mu _{k-1}^{\bar{Y}_{k-1}^{j,N},N} \right) h \right \rangle +\left| \sigma \left( \bar{Y}_{k-1}^{j,N},\mu _{k-1}^{\bar{Y}_{k-1}^{j,N},N} \right) \Delta W_{k-1}^{j} \right|^2\right.\right.\\
&\qquad +\left.\left.\left.2\left \langle \bar{Y}_{k-1}^{j,N}+b\left( \bar{Y}_{k-1}^{j,N},\mu _{k-1}^{\bar{Y}_{k-1}^{j,N},N} \right) h,\sigma \left( \bar{Y}_{k-1}^{j,N},\mu _{k-1}^{\bar{Y}_{k-1}^{j,N},N} \right) \Delta W_{k-1}^{j} \right \rangle +K_{15}h \right) ^p\right\rvert\ \mathcal{F} _{t_{k-1}} \right] \\
=& \mathbb{E} \left[ \left( 1+\left| \bar{Y}_{k-1}^{j,N} \right|^2 \right) ^p\left( 1+\tfrac{2\left\langle \bar{Y}_{k-1}^{j,N},b\left( \bar{Y}_{k-1}^{j,N},\mu _{k-1}^{\bar{Y}_{k-1}^{j,N},N} \right) h \right \rangle +\left| \sigma \left( \bar{Y}_{k-1}^{j,N},\mu _{k-1}^{\bar{Y}_{k-1}^{j,N},N} \right) \Delta W_{k-1}^{j} \right|^2}{1+\left| \bar{Y}_{k-1}^{j,N} \right|^2}\right.\right.\\
&\qquad +\left.\left.\left.\tfrac{ 2\left\langle \bar{Y}_{k-1}^{j,N}+b\left( \bar{Y}_{k-1}^{j,N},\mu _{k-1}^{\bar{Y}_{k-1}^{j,N},N} \right) h,\sigma \left( \bar{Y}_{k-1}^{j,N},\mu _{k-1}^{\bar{Y}_{k-1}^{j,N},N} \right) \Delta W_{k-1}^{j} \right\rangle}{1+\left| \bar{Y}_{k-1}^{j,N} \right|^2}+\tfrac{K_{15}h}{1+\left| \bar{Y}_{k-1}^{j,N} \right|^2}\right) ^p\right\rvert\ \mathcal{F} _{t_{k-1}} \right]. \\
\end{aligned}
\end{equation}
We define
\begin{equation}\tag{5.66}\label{5.19}
\begin{aligned}
\xi^{j,N}:=&\tfrac{2\left\langle \Bar{Y}^{j,N}_{k-1},b\left(\bar{Y}^{j,N}_{k-1},\mu^{\bar{Y}^{j,N}_{k-1},N}_{k-1}\right) h\right\rangle}{1+\left|\Bar{Y}^{j,N}_{k-1}\right|^2}+\tfrac{\left|\sigma\left(\bar{Y}^{j,N}_{k-1},\mu^{\bar{Y}^{j,N}_{k-1},N}_{k-1}\right) \Delta W_{k-1}^{j}\right|^2}{1+\left|\Bar{Y}^{j,N}_{k-1}\right|^2}\\
&+\tfrac{2\left\langle \Bar{Y}^{j,N}_{k-1}+b\left(\bar{Y}^{j,N}_{k-1},\mu^{\bar{Y}^{j,N}_{k-1},N}_{k-1}\right) h , \sigma\left(\bar{Y}^{j,N}_{k-1},\mu^{\bar{Y}^{j,N}_{k-1},N}_{k-1}\right)\Delta W_{k-1}^{j}\right\rangle}{1+\left|\Bar{Y}^{j,N}_{k-1}\right|^2}\\
&:=\xi^{j,N}_1+\xi^{j,N}_2+\xi^{j,N}_3.
\end{aligned}
\end{equation}
Therefore, leveraging the $\mathcal{F}_{t_{k-1}}$-measurability of $\bar{Y}_{k-1}^{j,N}$ and employing the binomial expansion, we derive from (\ref{5.18}), for any $1 \leq j \leq N$,\\
\begin{equation}\tag{5.67}\label{5.20}
\begin{aligned}
&\mathbb{E} \left[\left. \left( 1+\left| Y_{k}^{j,N} \right|^2 \right) ^p\right\rvert\ \mathcal{F} _{t_{k-1}} \right]\\
\le &\mathbb{E} \left[\left. \left(1+\xi ^{j,N}+\tfrac{K_{15}h}{1+\left| \bar{Y}_{k-1}^{j,N} \right|^2}\right)^p\right\rvert\ \mathcal{F} _{t_{k-1}} \right] \left( 1+\left| \bar{Y}_{k-1}^{j,N} \right|^2 \right) ^p
\\
\le &\mathbb{E} \left[\left. \left(1+\xi ^{j,N}\right)^p\right\rvert\ \mathcal{F} _{t_{k-1}} \right] \left( 1+\left| \bar{Y}_{k-1}^{j,N} \right|^2 \right) ^p+pK_{16}h\left( 1+\left| \bar{Y}_{k-1}^{j,N} \right|^2 \right) ^{p-1}\mathbb{E}\left[\left. \left(1+\xi ^{j,N}\right)^{p-1}\right\rvert\ \mathcal{F} _{t_{k-1}} \right] 
\\
\le &\left\{1+p\mathbb{E}\left[\left. \xi^{j,N} \right\rvert\ \mathcal{F}_{t_{k-1}}\right]+\tfrac{p\left(p-1\right)}{2}\mathbb{E}\left[\left.\left(\xi^{j,N}\right)^2\right\rvert\ \mathcal{F}_{t_{k-1}}\right]+\sum_{r=3}^p{C_{p}^{r}\mathbb{E}\left[\left.\left(\xi^{j,N}\right)^r\right\rvert\ \mathcal{F}_{t_{k-1}}\right]}\right\}\\
&\times\left( 1+\left| \bar{Y}_{k-1}^{j,N} \right|^2 \right) ^p+K_{16}ph\left\{1+\left(p-1\right)\mathbb{E}\left[\left. \xi^{j,N} \right\rvert\ \mathcal{F}_{t_{k-1}}\right]\right.\\
&+\left.\tfrac{\left(p-1\right)\left(p-2\right)}{2}\mathbb{E}\left[\left.\left(\xi^{j,N}\right)^2\right\rvert\ \mathcal{F}_{t_{k-1}}\right]+\sum_{r=3}^p{C_{p-1}^{r}\mathbb{E}\left[\left.\left(\xi^{j,N}\right)^r\right\rvert\ \mathcal{F}_{t_{k-1}}\right]}\right\}\left( 1+\left| \bar{Y}_{k-1}^{j,N} \right|^2 \right) ^{p-1}.
\end{aligned}
\end{equation}
Exploiting the properties of Brownian motion, we derive
\begin{equation}\tag{5.68}\label{5.21}
\begin{aligned}
&\mathbb{E}\left[\left.\xi^{j,N} \right\rvert\ \mathcal{F}_{t_{k-1}}\right]\\
=&\mathbb{E}\left[\left.\tfrac{2\left\langle \Bar{Y}^{j,N}_{k-1}, b\left(\bar{Y}^{j,N}_{k-1},\mu^{\bar{Y}^{j,N}_{k-1},N}_{k-1}\right)h\right\rangle+\left|\sigma\left(\bar{Y}^{j,N}_{k-1},\mu^{\bar{Y}^{j,N}_{k-1},N}_{k-1}\right) \Delta W_{k-1}^{j}\right|^2}{1+\left|\bar{Y}^{j,N}_{k-1}\right|^2} \right\rvert\ \mathcal{F}_{t_{k-1}}\right] \\
&+\mathbb{E}\left[\left.\tfrac{2\left\langle\bar{Y}^{j,N}_{k-1}+b\left(\bar{Y}^{j,N}_{k-1},\mu^{\bar{Y}^{j,N}_{k-1},N}_{k-1}\right)h, \sigma\left(\bar{Y}^{j,N}_{k-1},\mu^{\bar{Y}^{j,N}_{k-1},N}_{k-1}\right) \Delta W_{k-1}^{j}\right\rangle}{1+\left|\Bar{Y}^{j,N}_{k-1}\right|^2} \right\rvert\ \mathcal{F}_{t_{k-1}}\right]\\
=&\tfrac{2\left\langle\bar{Y}^{j,N}_{k-1}, b\left(\bar{Y}^{j,N}_{k-1},\mu^{\bar{Y}^{j,N}_{k-1},N}_{k-1}\right)h\right\rangle+\left|\left|\sigma\left(\bar{Y}^{j,N}_{k-1},\mu^{\bar{Y}^{j,N}_{k-1},N}_{k-1}\right)\right|\right|^2 h}{1+\left|\Bar{Y}^{j,N}_{k-1}\right|^2}.\\
\end{aligned}
\end{equation}
Subsequently, we obtain
\begin{equation}\tag{5.69}\label{5.22}
\begin{aligned}
 &\mathbb{E}\left[\left.\left(\xi^{j,N}\right)^2  \right\rvert\ \mathcal{F}_{t_{k-1}}\right] \\
=  &\mathbb{E}\left[\left.\left(\left(\xi^{j,N}_1\right)^2+\left(\xi^{j,N}_2\right)^2+\left(\xi^{j,N}_3\right)^2+2 \xi^{j,N}_1 \xi^{j,N}_2+2 \xi^{j,N}_1 \xi^{j,N}_3+2 \xi^{j,N}_2 \xi^{j,N}_3\right)  \right\rvert\ \mathcal{F}_{t_{k-1}}\right].
\end{aligned}
\end{equation}
We now evaluate each component on the right side of (\ref{5.22}). Applying (\ref{5.14}), the following estimate holds for the initial component on the right side of (\ref{5.22}):
\begin{equation}\tag{5.70}\label{5.23}
\begin{aligned}
\mathbb{E}\left[\left.\left(\xi^{j,N}_1\right)^2  \right\rvert\ \mathcal{F}_{t_{k-1}}\right] & \leq \mathbb{E}\left[\left.\tfrac{4\left|\bar{Y}^{j,N}_{k-1}\right|^2\left|b\left(\bar{Y}^{j,N}_{k-1},\mu^{\bar{Y}^{j,N}_{k-1},N}_{k-1}\right)h\right|^2}{\left(1+\left|\bar{Y}^{j,N}_{k-1}\right|^2\right)^2}  \right\rvert\ \mathcal{F}_{t_{k-1}}\right] \\
& \leq \tfrac{4\left|\bar{Y}^{j,N}_{k-1}\right|^2\left(K_{12} h+a_5 h^2\right)}{\left(1+\left|\bar{Y}^{j,N}_{k-1}\right|^2\right)^2}    \leq \tfrac{K h}{1+\left|\Bar{Y}^{j,N}_{k-1}\right|^2}.
\end{aligned}
\end{equation}
For the second component on the right side of (\ref{5.22}), applying (\ref{5.15}) yields
\begin{equation}\tag{5.71}\label{5.24}
\begin{aligned}
\mathbb{E}\left[\left.\left(\xi^{j,N}_2\right)^2  \right\rvert\ \mathcal{F}_{t_{k-1}}\right] & =\mathbb{E}\left[\left.\tfrac{\left|\sigma\left(\bar{Y}^{j,N}_{k-1},\mu^{\bar{Y}^{j,N}_{k-1},N}_{k-1}\right) \Delta W_{k-1}^{j}\right|^4}{\left(1+\left|\bar{Y}^{j,N}_{k-1}\right|^2\right)^2}  \right\rvert\ \mathcal{F}_{t_{k-1}}\right] \\
& \leq \tfrac{3\left(K_{13} h^{-\frac{1}{2}}+a_6\right)^2 h^2}{\left(1+\left|\bar{Y}^{j,N}_{k-1}\right|^2\right)^2}   \leq \tfrac{K h}{1+\left|\bar{Y}^{j,N}_{k-1}\right|^2}.
\end{aligned}
\end{equation}
Utilizing (\ref{5.14}) and (\ref{5.15}), we observe that the third and fourth components appearing on the right side of (\ref{5.22}) satisfy the following estimates:
\begin{equation}\tag{5.72}\label{5.25}
\begin{aligned}
\mathbb{E}\left[\left.\left(\xi^{j,N}_3\right)^2  \right\rvert\ \mathcal{F}_{t_{k-1}}\right] & \leq\mathbb{E}\left[\left.\tfrac{4\left|\bar{Y}^{j,N}_{k-1}+b\left(\bar{Y}^{j,N}_{k-1},\mu^{\bar{Y}^{j,N}_{k-1},N}_{k-1}\right)h\right|^2\left|\sigma\left(\bar{Y}^{j,N}_{k-1},\mu^{\bar{Y}^{j,N}_{k-1},N}_{k-1}\right) \Delta W_{k-1}^{j}\right|^2}{\left(1+\left|\bar{Y}^{j,N}_{k-1}\right|^2\right)^2} \right\rvert\, \mathcal{F}_{t_{k-1}}\right] \\
&\leq \tfrac{8\left( \left|\bar{Y}_{k-1}^{j,N}\right|^2+h^2\left|b\left( \bar{Y}_{k-1}^{j,N},\mu _{k-1}^{\bar{Y}_{k-1}^{j,N},N} \right) \right|^2 \right) \left|\left|\sigma \left( \bar{Y}_{k-1}^{j,N},\mu _{k-1}^{\bar{Y}_{k-1}^{j,N},N} \right) \right|\right|^2h}{\left( 1+\left|\bar{Y}_{k-1}^{j,N}\right|^2 \right) ^2}\\
& \leq \tfrac{8h\left|\left|\sigma\left(\bar{Y}^{j,N}_{k-1},\mu^{\bar{Y}^{j,N}_{k-1},N}_{k-1}\right)\right|\right|^2}{1+\left|\bar{Y}^{j,N}_{k-1}\right|^2}+\tfrac{K h}{1+\left|\bar{Y}^{j,N}_{k-1}\right|^2},
\end{aligned}
\end{equation}
and
\begin{equation}\tag{5.73}\label{5.26}
\begin{aligned}
\mathbb{E}\left[\left.2 \xi^{j,N}_1 \xi^{j,N}_2  \right\rvert\ \mathcal{F}_{t_{k-1}}\right] & =\mathbb{E}\left[\left.\tfrac{4\left\langle\bar{Y}^{j,N}_{k-1}, b\left(\bar{Y}^{j,N}_{k-1},\mu^{\bar{Y}^{j,N}_{k-1},N}_{k-1}\right)h\right\rangle\left|\sigma\left(\bar{Y}^{j,N}_{k-1},\mu^{\bar{Y}^{j,N}_{k-1},N}_{k-1}\right) \Delta W_{k-1}^{j}\right|^2}{\left(1+\left|\bar{Y}^{j,N}_{k-1}\right|^2\right)^2} \right\rvert\, \mathcal{F}_{t_{k-1}}\right] \\
& \leq \tfrac{K h}{1+\left|\bar{Y}^{j,N}_{k-1}\right|^2}.
\end{aligned}
\end{equation}
The properties of Brownian motion imply
\begin{equation}\tag{5.74}\label{5.27}
\mathbb{E}\left[\left.2 \xi^{j,N}_1 \xi^{j,N}_3  \right\rvert\ \mathcal{F}_{t_{k-1}}\right]=\mathbb{E}\left[\left.2 \xi^{j,N}_2 \xi^{j,N}_3  \right\rvert\ \mathcal{F}_{t_{k-1}}\right]=0.
\end{equation}
Combining (\ref{5.23})--(\ref{5.27}) yields
\begin{equation}\tag{5.75}\label{5.28}
\mathbb{E}\left[\left.\left(\xi^{j,N}\right)^2  \right\rvert\ \mathcal{F}_{t_{k-1}}\right] \leq \tfrac{8 h \left|\left|\sigma\left(\bar{Y}^{j,N}_{k-1},\mu^{\bar{Y}^{j,N}_{k-1},N}_{k-1}\right)\right|\right|^2}{1+\left|\bar{Y}^{j,N}_{k-1}\right|^2}+\tfrac{K h}{1+\left|\bar{Y}^{j,N}_{k-1}\right|^2}.
\end{equation}  
For $r \geq 3$, we can assert that
\begin{equation}\tag{5.76}\label{5.29}
\begin{aligned}
\mathbb{E} \left[\left. \left(\xi ^{j,N}\right)^r \right\rvert\ \mathcal{F} _{t_{k-1}} \right] &=\mathbb{E} \left[\left. \left( \xi _{1}^{j,N}+\xi _{2}^{j,N}+\xi _{3}^{j,N} \right) ^r\right\rvert\ \mathcal{F} _{t_{k-1}} \right] 
\\
&=\sum_{r_0+s_0+l_0=r}{\tfrac{r!}{r_0!s_0!l_0!}}\mathbb{E} \left[\left.\left( \xi _{1}^{j,N} \right) ^{r_0}\left( \xi _{2}^{j,N} \right) ^{s_0}\left( \xi _{3}^{j,N} \right) ^{l_0} \right\rvert\ \mathcal{F} _{t_{k-1}} \right] 
\\
&\le \tfrac{Kh}{1+\left|\bar{Y}_{k-1}^{j,N}\right|^2}.\\
\end{aligned}
\end{equation}
If $ l_0 = 2n_0 + 1 $ with $ n_0 \in \mathbb{N} $, the properties of Brownian motion imply 
\begin{equation}\tag{5.77}\label{5.30}
\mathbb{E} \left[\left. \left( \xi _{1}^{j,N} \right) ^{r_0}\left( \xi _{2}^{j,N} \right) ^{s_0}\left( \xi _{3}^{j,N} \right) ^{2n_0+1}\right\rvert\ \mathcal{F} _{t_{k-1}} \right] =0\,. 
\end{equation}
If $l_0=2n_0$ with $n_0 \in \mathbb{N}$, we can utilize the fact that $\mathbb{E} \left[\left. \left| \Delta W_{k-1}^{j} \right|^{2s}\right\rvert\ \mathcal{F} _{t_{k-1}} \right] =\left( 2s-1 \right) !!h^s$, together with (\ref{5.14}), (\ref{5.15}), and (\ref{5.19}), to obtain
\begin{equation}\tag{5.78}\label{5.31}
\begin{aligned}
&\mathbb{E} \left[\left. \left( \xi _{1}^{j,N} \right) ^{r_0}\left( \xi _{2}^{j,N} \right) ^{s_0}\left( \xi _{3}^{j,N} \right) ^{2n_0}\right\rvert\ \mathcal{F} _{t_{k-1}} \right] 
\\
\le& \tfrac{K_{18}\left| \bar{Y}_{k-1}^{j,N} \right|^{r_0+2n_0}\left| b\left( \bar{Y}_{k-1}^{j,N},\mu _{k-1}^{\bar{Y}_{k-1}^{j,N},N} \right) \right|^{r_0}\left|\left| \sigma \left( \bar{Y}_{k-1}^{j,N},\mu _{k-1}^{\bar{Y}_{k-1}^{j,N},N} \right) \right|\right| ^{2s_0+2n_0}h^{r_0+s_0+n_0}}{\left( 1+\left| \bar{Y}_{k-1}^{j,N} \right|^2 \right) ^r}
\\
&+\tfrac{K_{18}\left| \bar{Y}_{k-1}^{j,N} \right|^{r_0}\left| b\left( \bar{Y}_{k-1}^{j,N},\mu _{k-1}^{\bar{Y}_{k-1}^{j,N},N} \right) \right|^{r_0+2n_0}\left|\left| \sigma \left( \bar{Y}_{k-1}^{j,N},\mu _{k-1}^{\bar{Y}_{k-1}^{j,N},N} \right) \right|\right| ^{2s_0+2n_0}h^{r_0+s_0+3n_0}}{\left( 1+\left| \bar{Y}_{k-1}^{j,N} \right|^2 \right) ^r}
\\
\le& \tfrac{K_{19}\left| \bar{Y}_{k-1}^{j,N} \right|^{r_0+2n_0}\left( K_{12}h^{-1}+a_5 \right) ^{\frac{r_0}{2}}\left( K_{13}h^{-\frac{1}{2}}+a_6 \right) ^{s_0+n_0}h^{r_0+s_0+n_0}}{\left( 1+\left| \bar{Y}_{k-1}^{j,N} \right|^2 \right) ^r}
\\
&+\tfrac{K_{19}\left| \bar{Y}_{k-1}^{j,N} \right|^{r_0}\left( K_{12}h^{-1}+a_5 \right) ^{\frac{r_0}{2}+n_0}\left( K_{13}h^{-\frac{1}{2}}+a_6 \right) ^{s_0+n_0}h^{r_0+s_0+3n_0}}{\left( 1+\left| \bar{Y}_{k-1}^{j,N} \right|^2 \right) ^r}
\\
\le& \tfrac{K_{20}h^{\frac{r_0+s_0+n_0}{2}}}{\left( 1+\left| \bar{Y}_{k-1}^{j,N} \right|^2 \right) ^{\frac{r+s_0}{2}}}+\tfrac{K_{20}h^{\frac{r_0+s_0+3n_0}{2}}}{\left( 1+\left| \bar{Y}_{k-1}^{j,N} \right|^2 \right) ^{r-\frac{r_0}{2}}}
\le  \tfrac{K_{20}h^{\frac{r_0+s_0+n_0}{2}}}{\left( 1+\left| \bar{Y}_{k-1}^{j,N} \right|^2 \right) ^{\frac{r+s_0}{2}}}
= \tfrac{K_{20}h^{\frac{r-n_0}{2}}}{\left( 1+\left| \bar{Y}_{k-1}^{j,N} \right|^2 \right) ^{\frac{r+s_0}{2}}}.\\
\end{aligned}
\end{equation}
Since $ r \geq 3 $ and $ r, r_0, s_0, n_0 \in \mathbb{N} $ with $ r = r_0 + s_0 + 2n_0 $, it follows that $ \tfrac{r + s_0}{2} \geq 1 $ and $ \tfrac{r - n_0}{2} \geq 1 $. Hence, 
\begin{equation}\tag{5.79}\label{5.32}
\mathbb{E} \left[\left. \left( \xi _{1}^{j,N} \right) ^{r_0}\left( \xi _{2}^{j,N} \right) ^{s_0}\left( \xi _{3}^{j,N} \right) ^{2n_0}\right\rvert\ \mathcal{F} _{t_{k-1}} \right] \le \tfrac{Kh}{1+\left|\bar{Y}_{k-1}^{j,N}\right|^2}.\\
\end{equation}
Therefore, the assertion (\ref{5.29}) holds. 

By combining (\ref{2.6}), (\ref{5.21}), (\ref{5.28}), and (\ref{5.29}), and  leveraging Assumption \ref{ass2.4}, we derive from (\ref{5.20}) that 
\begin{equation}\tag{5.80}\label{5.33}
\begin{aligned}
	&\mathbb{E} \left[ \left.\left( 1+\left| Y_{k}^{j,N} \right|^2 \right) ^p \right\rvert\ \mathcal{F} _{t_{k-1}}\right]\\
	\le & \left\{ 1+\tfrac{2p\left.\left \langle \bar{Y}_{k-1}^{j,N},b\left( \bar{Y}_{k-1}^{j,N},\mu _{k-1}^{\bar{Y}_{k-1}^{j,N},N} \right) h \right.\right \rangle +p\left(4p-3\right)h\left|\left| \sigma \left( \bar{Y}_{k-1}^{j,N},\mu _{k-1}^{\bar{Y}_{k-1}^{j,N},N} \right) \right|\right| ^2}{1+\left| \bar{Y}_{k-1}^{j,N} \right|^2} +\sum_{r=3}^p{C_{p}^{r} \tfrac{Kh}{1+\left|\bar{Y}_{k-1}^{j,N}\right|^2}}\right\} \left( 1+\left| \bar{Y}_{k-1}^{j,N} \right|^2 \right) ^p \\
	&+K_{16}ph\left\{ 1+\tfrac{2\left(p-1\right)\left.\left \langle \bar{Y}_{k-1}^{j,N},b\left( \bar{Y}_{k-1}^{j,N},\mu _{k-1}^{\bar{Y}_{k-1}^{j,N},N} \right) h \right.\right \rangle +\left(p-1\right)\left(4p-7\right)h\left|\left| \sigma \left( \bar{Y}_{k-1}^{j,N},\mu _{k-1}^{\bar{Y}_{k-1}^{j,N},N} \right) \right|\right| ^2}{1+\left| \bar{Y}_{k-1}^{j,N} \right|^2} +\sum_{r=3}^p{C_{p-1}^{r} \tfrac{Kh}{1+\left|\bar{Y}_{k-1}^{j,N}\right|^2}}\right\}\\
    &\times\left( 1+\left| \bar{Y}_{k-1}^{j,N} \right|^2 \right) ^{p-1}\\
	\le&  \left\{ 1+\tfrac{ph\left[ -a_1\left| \bar{Y}_{k-1}^{j,N} \right|^2+a_2\mathcal{W} _2\left( \mu _{k-1}^{\bar{Y}_{k-1}^{j,N},N} \right) ^2+C \right]}{1+\left| \bar{Y}_{k-1}^{j,N} \right|^2} \right\} \left( 1+\left| \bar{Y}_{k-1}^{j,N} \right|^2 \right) ^p \\
	&+K_{21}h\left\{ 1+\tfrac{\left(p-1\right)h\left[ -a_1\left| \bar{Y}_{k-1}^{j,N} \right|^2+a_2\mathcal{W} _2\left( \mu _{k-1}^{\bar{Y}_{k-1}^{j,N},N} \right) ^2+C \right]}{1+\left| \bar{Y}_{k-1}^{j,N} \right|^2} \right\} \left( 1+\left| \bar{Y}_{k-1}^{j,N} \right|^2 \right) ^{p-1} \\
	\leq&  \left\{ 1+\tfrac{ph\left[ -a_1\left| \bar{Y}_{k-1}^{j,N} \right|^2+\tfrac{a_2}{N}\sum_{i=1}^N{\left|\bar{Y}_{k-1}^{i,N}\right|^2}\right]}{1+\left| \bar{Y}_{k-1}^{j,N} \right|^2} \right\} \left( 1+\left| \bar{Y}_{k-1}^{j,N} \right|^2 \right) ^p \\
	&+K_{22}h \left\{ 1+\tfrac{\left(p-1\right)h\left[ -a_1\left| \bar{Y}_{k-1}^{j,N} \right|^2+\tfrac{a_2}{N}\sum_{i=1}^N{\left|\bar{Y}_{k-1}^{i,N}\right|^2}\right]}{1+\left| \bar{Y}_{k-1}^{j,N} \right|^2} \right\} \left( 1+\left| \bar{Y}_{k-1}^{j,N} \right|^2 \right) ^{p-1} \\
	\leq&\left( 1-a_1ph \right)  \left( 1+\left| \bar{Y}_{k-1}^{j,N} \right|^2 \right) ^p+\tfrac{a_2ph}{N}\sum_{i=1}^N{ \left( 1+\left| \bar{Y}_{k-1}^{i,N} \right|^2 \right) \left( 1+\left| \bar{Y}_{k-1}^{j,N} \right|^2 \right) ^{p-1}}\\
&+  K_{23}h\left(1-a_{1}\left(p-1\right)h\right) \left( 1+\left| \bar{Y}_{k-1}^{j,N} \right|^2 \right) ^{p-1}\\
&+K_{23}h\tfrac{a_2\left(p-1\right)h}{N}\sum_{i=1}^N{ \left( 1+\left| \bar{Y}_{k-1}^{i,N} \right|^2 \right) \left( 1+\left| \bar{Y}_{k-1}^{j,N} \right|^2 \right) ^{p-2}} .\\
\end{aligned}
\end{equation}
Applying the expectation operator to both sides of (\ref{5.33}), and observing that the sequence $\left\{ \bar{Y}_{k-1}^{j,N} \right\} _{1\leqslant j\leqslant N}$ is identically distributed, which implies that both the sequences $\left\{ \left( 1+\left| \bar{Y}_{k-1}^{j,N} \right|^2 \right)^p\right\} _{1\leqslant j\leqslant N}$ and $\left\{ \left( 1+\left| \bar{Y}_{k-1}^{j,N} \right|^2 \right)^{p-1}\right\} _{1\leqslant j\leqslant N}$ are also identically distributed, we apply (\ref{5.1}), (\ref{5.4}), and H\"{o}lder's inequality to  derive 
\begin{equation}\tag{5.81}\label{5.34}
\begin{aligned}
&\mathbb{E}\left[\left(1+\left|Y^{j,N}_k\right|^2\right)^p  \right]\\
\le& \left( 1-a_1ph \right) \mathbb{E} \left[ \left( 1+\left| \bar{Y}_{k-1}^{j,N} \right|^2 \right) ^p \right] 
\\
&+\tfrac{a_2ph}{N}\sum_{i=1}^N{\left\{ \mathbb{E} \left[ \left( 1+\left| \bar{Y}_{k-1}^{i,N} \right|^2 \right) ^p \right] \right\} ^{\frac{1}{p}}\left\{ \mathbb{E} \left[ \left( 1+\left| \bar{Y}_{k-1}^{j,N} \right|^2 \right) ^p \right] \right\} ^{\frac{p-1}{p}}}
\\
&+K_{23}h\left(1-a_{1}\left(p-1\right)h\right) \mathbb{E} \left[ \left( 1+\left| \bar{Y}_{k-1}^{j,N} \right|^2 \right) ^{p-1}\right]\\
&+K_{23}h\tfrac{a_2\left(p-1\right)h}{N}\sum_{i=1}^N{\left\{ \mathbb{E} \left[ \left( 1+\left| \bar{Y}_{k-1}^{i,N} \right|^2 \right) ^{p-1} \right] \right\} ^{\frac{1}{p-1}}\left\{ \mathbb{E} \left[ \left( 1+\left| \bar{Y}_{k-1}^{j,N} \right|^2 \right) ^{p-1} \right] \right\} ^{\frac{p-2}{p-1}}}\\
 \leq&\left(1-p\left(a_1-a_2\right) h\right)\mathbb{E}\left[\left(1+\left|\bar{Y}_{k-1}^{j,N}\right|^2\right)^p \right]+K_{24} h\mathbb{E}\left[\left(1+\left|\bar{Y}_{k-1}^{j,N}\right|^2\right)^{p-1}\right]\\
\leq&\left(1-p\left(a_1-a_2\right) h\right)\mathbb{E}\left[\left(1+\left|Y^{j,N}_{k-1}\right|^2\right)^p\right]+K_{24} h\mathbb{E}\left[\left(1+\left|Y^{j,N}_{k-1}\right|^2\right)^{p-1}\right].
\end{aligned}
\end{equation}
Applying Lemma~\ref{lemma5.2} and subsequently taking the supremum over $j$ establishes the following inequality for all positive integers $p$ satisfying $1 \leq p < \left\lfloor \frac{q_0}{2} \right\rfloor$:
\begin{equation}\tag{5.82}\label{5.82}
\underset{1\leqslant j\leqslant N}{\sup}\mathbb{E}\left[\left|Y^{j,N}_k\right|^{2 p}\right] \leq K_6 \underset{1\leqslant j\leqslant N}{\sup}\mathbb{E}\left[\left(1+\left|X_0^{j}\right|^{2 p}\right)\right]= K_6\mathbb{E}\left[\left(1+\left|X_0\right|^{2 p}\right)\right]\leq K_6\mathbb{E}\left[\left(1+\left|X_0\right|^{4 p}\right)\right].
\end{equation}

For non-integer values of $p$, let $p_1 = \left\lfloor p \right\rfloor$ and $\theta = p - p_1 \in \left[ 0,1 \right)$. Since $1 \leq p < \left\lfloor \frac{q_0}{2} \right\rfloor$, it follows that $1 \leq p_1 \leq p < p_1 + 1 \leq \left\lfloor \frac{q_0}{2} \right\rfloor$. Applying H\"older's inequality, we obtain, for any $1 \leq j \leq N$,
\begin{equation}\tag{5.83}\label{5.666}
\begin{aligned}
    \mathbb{E} \left[ \left| Y_{k}^{j,N} \right|^{2p} \right] &=\mathbb{E} \left[ \left| Y_{k}^{j,N} \right|^{2p_1\left( 1-p+p_1 \right) +2\left( p_1+1 \right) \left( p-p_1 \right)} \right] 
\\
&\le \left( \mathbb{E} \left[ \left| Y_{k}^{j,N} \right|^{2p_1} \right] \right) ^{1-\theta}\left( \mathbb{E} \left[ \left| Y_{k}^{j,N} \right|^{2\left( p_1+1 \right)} \right] \right) ^{\theta}.
\end{aligned}
\end{equation}
By the result \eqref{5.82} already established for positive integers, and noting that $2p_1<2\left(p_1+1\right)\leq4p_1\leq4p$, we obtain
\begin{equation}\tag{5.84}\label{5.669}
\begin{aligned}
	\mathbb{E} \left[ \left| Y_{k}^{j,N} \right|^{2p} \right]
	\le& K_{6}^{1-\theta}K_{6}^{\theta}\left( \mathbb{E} \left[ \left( 1+\left| X_0 \right|^{2p_1} \right) \right] \right) ^{1-\theta}\left( \mathbb{E} \left[ \left( 1+\left| X_0 \right|^{2\left(p_1+1\right)} \right) \right] \right) ^{\theta}\\
    \le& K_{6}\left( \mathbb{E} \left[ \left( 1+\left| X_0 \right|^{4p} \right) \right] \right) ^{1-\theta}\left( \mathbb{E} \left[ \left( 1+\left| X_0 \right|^{4p} \right) \right] \right) ^{\theta}\\
	\le& \bar{K}\mathbb{E} \left[ \left( 1+\left| X_0 \right|^{4p} \right) \right] .\\
\end{aligned}
\end{equation}
Taking the supremum over $j$ establishes inequality \eqref{5.16} for all $p$ satisfying $1 \leq p < \left\lfloor \frac{q_0}{2} \right\rfloor$.\\
The theorem is thus proved.
\end{proof} 
\subsection{Proof of Theorem \ref{theorem5.7}}
To establish this theorem, we first introduce several useful lemmas.
\begin{lemma}\label{lemma5.4}
  Assume that  Assumptions \ref{ass2.1}--\ref{ass2.4} are satisfied, then, for $1 \leq\kappa\leq \tfrac{q_0}{2}$ and $0<h \leq h_2$,
\begin{equation}\tag{5.85}\label{5.56}
\underset{1\leqslant j\leqslant N}{\sup}\mathbb{E}\left[\left|X^{j,N}\left(h\right)-X^{j,N}_0\right|^{2}\right] \leq C_{15} \mathbb{E}\left[\left(1+\left|X_0\right|^{2 \kappa }\right)\right] h ,
\end{equation}
\begin{equation}\tag{5.86}\label{5.57}
\underset{1\leqslant j\leqslant N}{\sup}\mathbb{E}\left[\left|X^{j,N}\left(h\right)-X^{j,N}_0\right|^{4}\right] \leq C_{16} \mathbb{E}\left[\left(1+\left|X_0\right|^{4 \kappa }\right)\right] h^2 ,
\end{equation}
where $C_{15}$ and $C_{16}$ are constants independent of $h, N$.
\end{lemma}
\begin{proof} By applying (\ref{2.6}), (\ref{2.15}), and (\ref{2.16}), along with H\"older's inequality and the It\^o isometry, we derive
\begin{equation}\tag{5.87}\label{5.58}
\begin{aligned}
 \mathbb{E} \bigg[\big|X^{j,N}\left(h\right)&-X^{j,N}_0\big|^{2}\bigg] \\
 \leq& 2 h \int_0^h \mathbb{E}\left[\left|b\left(X_s^{j, N}, \mu_s^{X, N}\right)\right|^{2}\right] \mathrm{d} s+2  \int_0^h \mathbb{E}\left[\left|\left|\sigma\left(X_s^{j, N}, \mu_s^{X, N}\right)\right|\right|^{2}\right] \mathrm{d} s \\
 \leq& 2 h \int_0^h\left[K_4\left\{ \mathbb{E}\left[\left|X_s^{j, N}\right|^{2\kappa}\right]+1 \right\} +\tfrac{a_3}{N}\sum_{i=1}^N{\mathbb{E}\left[\left|{X}^{i,N}_s\right|^2\right]}\right] \mathrm{d} s\\
&+2  \int_0^h \left[K_5\left\{ \mathbb{E}\left[\left|X_s^{j, N}\right|^{\kappa+1}\right]+1\right\} +\tfrac{a_4}{N}\sum_{i=1}^N{\mathbb{E}\left[\left|{X}^{i,N}_s\right|^2\right]}\right] \mathrm{d} s.\\
  \end{aligned}
\end{equation}
Since $\left\{ {X}^{i,N}_s \right\} _{1\leqslant i\leqslant N}$ is identically distributed, $\left\{ \left|{X}^{i,N}_s\right|^2 \right\} _{1\leqslant i\leqslant N}$ is also identically distributed.
Applying (\ref{3.25}), we obtain, for all $1\leq j\leq N$,   
\begin{equation}\tag{5.88}\label{5.59}
\begin{aligned}
 \mathbb{E}\left[\left|X^{j,N}\left(h\right)-X^{j,N}_0\right|^{2}\right]  
 \leq& 2 h \int_0^hK_{25}\left\{ \mathbb{E}\left[\left|X_{0}\right|^{2\kappa}\right]+1 \right\} \mathrm{d} s+2  \int_0^h K_{26}\left\{ \mathbb{E}\left[\left|X_{0}\right|^{\kappa+1}\right]+1\right\} \mathrm{d} s \\
\leq&  C_{15} \mathbb{E}\left[\left(1+\left|X_0\right|^{2 \kappa }\right)\right] h.
\end{aligned}
\end{equation}
Taking the supremum over $ j $ in (\ref{5.59}) yields (\ref{5.56}).

By applying (\ref{2.6}), (\ref{2.15}), and (\ref{2.16}), along with H\"older's inequality and the BDG inequality, we derive
\begin{equation}\tag{5.89}\label{5.60}
\begin{aligned}
 \mathbb{E} \bigg[\big|X^{j,N}\left(h\right)&-X^{j,N}_0\big|^{4}\bigg] \\
 \leq& 8 h^{3} \int_0^h \mathbb{E}\left[\left|b\left(X_s^{j, N}, \mu_s^{X, N}\right)\right|^{4}\right] \mathrm{d} s+36h  \int_0^h \mathbb{E}\left[\left|\left|\sigma\left(X_s^{j, N}, \mu_s^{X, N}\right)\right|\right|^{4}\right] \mathrm{d} s \\
\leq& 8 h^{3} \int_0^h K_4^{2} \mathbb{E}\left[\left(\left|X_s^{j, N}\right|^{2\kappa}+1\right)^2\right]+\tfrac{2 K_4 a_3}{N}\sum_{i=1}^N{\mathbb{E}\left[\left(1+\left|{X}^{j,N}_s\right|^{2\kappa}\right)\left|{X}^{i,N}_s\right|^2\right]}\\
&+\tfrac{a_3^{2}}{N}\sum_{i=1}^N{\mathbb{E}\left[\left|{X}^{i,N}_s\right|^4\right]}\mathrm{d} s\\
&+ 36 h \int_0^h K_5^{2} \mathbb{E}\left[\left(\left|X_s^{j, N}\right|^{\kappa+1}+1\right)^2\right]+\tfrac{2 K_5 a_4}{N}\sum_{i=1}^N{\mathbb{E}\left[\left(1+\left|{X}^{j,N}_s\right|^{\kappa+1}\right)\left|{X}^{i,N}_s\right|^2\right]}\\
&+\tfrac{a_4^{2}}{N}\sum_{i=1}^N{\mathbb{E}\left[\left|{X}^{i,N}_s\right|^4\right]}\mathrm{d} s.\\
  \end{aligned}
\end{equation}
Since $\left\{ {X}^{i,N}_s \right\} _{1\leqslant i\leqslant N}$ is identically distributed, $\left\{ \left|{X}^{i,N}_s\right|^4 \right\} _{1\leqslant i\leqslant N}$ is also identically distributed.
Applying (\ref{3.25}), we obtain, for all $1\leq j\leq N$,   
\begin{equation}\tag{5.90}\label{5.61}
\begin{aligned}
 \mathbb{E}\left[\left|X^{j,N}\left(h\right)-X^{j,N}_0\right|^{4}\right]  
 \leq& K_{27} h^4 \mathbb{E}\left[\left(\left|X_{0}\right|^{4\kappa}+1\right)\right]  +K_{27} h^4\mathbb{E}\left[\left(\left|X_{0}\right|^{2\kappa+2}+1\right)\right] \\
 &+ K_{27} h^4 \mathbb{E}\left[\left(\left|X_{0}\right|^{4}+1\right)\right] +K_{27} h^2 \mathbb{E}\left[\left(\left|X_{0}\right|^{2\kappa+2}+1\right)\right]\\
 &+ K_{27} h^2 \mathbb{E}\left[\left(\left|X_{0}\right|^{\kappa+3}+1\right)\right] +K_{27} h^2 \mathbb{E}\left[\left(\left|X_{0}\right|^{4}+1\right)\right]\\
\leq&  C_{16} \mathbb{E}\left[\left(1+\left|X_0\right|^{4 \kappa }\right)\right] h^2.\\
\end{aligned}
\end{equation}
Taking the supremum over $j$ in (\ref{5.61}) yields (\ref{5.57}).
\end{proof} 
The Euler-Maruyama method for the interacting particle system \eqref{interacting} takes the form:
\begin{equation}\tag{5.91}\label{5.64}
Y^{j,N}_{k+1}:={Y}^{j,N}_k+hb\left({Y}^{j,N}_k,\mu^{{Y}^{j,N}_k,N}_k\right)+\sigma\left({Y}^{j,N}_k,\mu^{{Y}^{j,N}_k,N}_k\right) \Delta W^{j}_{k},
\end{equation}
and the corresponding one-step approximation takes the form
\begin{equation}\tag{5.92}\label{5.65}
Y^{j,N}_{E}\left( t,x^{j,N};t+h \right) :=x^{j,N}+\int_t^{t+h}{b\left( x^{j,N},\mu ^{x^{j,N},N} \right)}ds+\int_t^{t+h}{\sigma \left( x^{j,N},\mu ^{x^{j,N},N} \right)}dW^j\left( s \right).
\end{equation}
\begin{lemma}\label{lemmma5.5}
       For $h \in \left(0,1\right]$ and $q \geq 1$, the following inequality holds:
\begin{equation}\tag{5.93}\label{5.62}
\left|x-\psi\left(x\right)\right| \leq 2\left(1+\left|x\right|^{q+1}\right) h^{\frac{q}{2\left(\kappa+2\right)}}.
\end{equation}
\end{lemma}
\begin{proof}
Since $\psi\left(x\right) = \min \left\{ 1, h^{-\tfrac{1}{2\left(\kappa + 2\right)}} \left|x\right|^{-1} \right\} x$, then we obtain
\begin{equation}\tag{5.94}
\begin{aligned}
    \left| x-\psi \left(x\right) \right|&=\mathbf{1}_{\left\{ \left| x \right|\leq h^{-{\frac{1}{2\left( \kappa +2 \right)}}} \right\}}\left(x\right)\left| x-x \right|+\mathbf{1}_{\left\{ \left| x \right|>h^{-{\frac{1}{2\left( \kappa +2 \right)}}} \right\}}\left(x\right)\left| x-h^{-{\frac{1}{2\left( \kappa +2 \right)}}}\tfrac{x}{\left| x \right|} \right|
\\
&\leq \mathbf{1}_{\left\{ \left| x \right|>h^{-{\frac{1}{2\left( \kappa +2 \right)}}} \right\}}\left(x\right)\left| x-h^{-{\frac{1}{2\left( \kappa +2 \right)}}}\tfrac{x}{\left| x \right|} \right|
\\
&\leq 2\left| x \right|\mathbf{1}_{\left\{ \left| x \right|>h^{-{\frac{1}{2\left( \kappa +2 \right)}}} \right\}}\left(x\right)\leq 2\left| x \right|^{q+1}h^{{\frac{q}{2\left( \kappa +2 \right)}}}<2\left( 1+\left| x \right|^{q+1} \right) h^{{\frac{q}{2\left( \kappa +2 \right)}}}.\\
\end{aligned}
\end{equation}
This gives the desired estimate. 
\end{proof}
\begin{lemma}\label{lemma5.6}
Assuming that Assumptions~\ref{ass2.1}--\ref{ass2.4} are satisfied.  For any $1\leq \kappa \leq \tfrac{q_0-3}{2}$ and $0<h \leq h_2$, we obtain the following bounds: 
\begin{equation}\tag{5.95}\label{5.66}
\underset{1\leqslant j\leqslant N}{\sup}\left|\mathbb{E}\left[X^{j,N}\left(t, x^{j,N} ; t+h\right)-Y^{j,N}_{E}\left(t, x^{j,N}; t+h\right)\right]\right| \leq K_{28}\mathbb{E}\left[\left(1+\left|x^{j,N}\right|^{4 \kappa+6}\right)\right] h^{\frac{3}{2}} ,
\end{equation}
\begin{equation}\tag{5.96}\label{5.67}
\underset{1\leqslant j\leqslant N}{\sup}\left|\mathbb{E}\left[Y^{j,N}_{E}\left(t, x^{j,N} ; t+h\right)-Y^{j,N}\left(t, x^{j,N} ; t+h\right)\right]\right|  \leq K_{29}\mathbb{E}\left[\left(1+\left|x^{j,N}\right|^{3 \kappa+7}\right)\right] h^{\frac{3}{2}}, 
\end{equation}
\begin{equation}\tag{5.97}\label{5.68}
\underset{1\leqslant j\leqslant N}{\sup}\left|\mathbb{E}\left[X^{j,N}\left(t, x^{j,N} ; t+h\right)-Y^{j,N}\left(t, x^{j,N} ; t+h\right)\right]\right|  \leq K_{7}\mathbb{E}\left[\left(1+\left|x^{j,N}\right|^{ 4\kappa+6}\right)\right] h^{\frac{3}{2}}.
\end{equation}
\end{lemma}
\begin{proof} To establish the first inequality, we invoke Assumption~\ref{ass2.2}, together with \eqref{2.8}, \eqref{interacting}, and \eqref{5.65}, along with Jensen's inequality and H\"older's inequality, to obtain
\begin{equation}\tag{5.98}\label{5.69}
\begin{aligned}
&\left|\mathbb{E}\left[X^{j,N}\left(t, x^{j,N} ; t+h\right)-Y^{j,N}_E\left(t, x^{j,N} ; t+h\right)\right]\right|\\
\leq& \left|\mathbb{E}\left[\int_t^{t+h} b\left({X}^{j,N}_s,\mu^{X,N}_s\right)-b\left( x^{j,N},\mu ^{x^{j,N},N} \right) \mathrm{d} s\right]\right|\\
\leq & \sqrt{L_1} \int_t^{t+h} \mathbb{E}\left[ \left\{\left( 1+\left|X_{s}^{j,N}\right|^{2\kappa -2}+\left|x^{j,N}\right|^{2\kappa -2} \right) \left|X_{s}^{j,N}-x^{j,N}\right|^2\right.\right.\\
&+\left.\left.\tfrac{1}{N}\sum_{i=1}^N{\left|X_{s}^{i,N}-x^{i,N}\right|^2}\right\}^{\frac{1}{2}}\right] \mathrm{d} s\\
\leq & \sqrt{L_1} \int_t^{t+h} \left\{\mathbb{E}\left[  \left(1+\left|X_{s}^{j,N}\right|^{2\kappa -2}+\left|x^{j,N}\right|^{2\kappa -2}\right)\right]\right\}^{\frac{1}{2}} \left(\mathbb{E}\left[\left|X_{s}^{j,N}-x^{j,N}\right|^{2}\right]\right)^{\frac{1}{2}}\mathrm{d} s\\
 &+ \sqrt{L_1} \int_t^{t+h} \left(\tfrac{1}{N}\sum_{i=1}^N{\mathbb{E}\left[\left|X_{s}^{i,N}-x^{i,N}\right|^2\right]}\right) ^{\frac{1}{2}}\mathrm{d} s.\\
\end{aligned}
\end{equation}
Since the pairs $\left(X_s^{i,N}, x^{i,N}\right)$ for $1 \leq i \leq N$ are solutions to the same interacting particle system \eqref{interacting}, they are identically distributed. Consequently, $\left\{\left|X_s^{i,N} - x^{i,N}\right|^2\right\}_{1 \leq i \leq N}$ is also identically distributed sequence. Taking the supremum over $j$ and utilizing Lemma~\ref{lemma5.4} together with Theorem~\ref{momentboundednessofthesolution}, we derive
\begin{equation}\tag{5.99}\label{5.71}
\begin{aligned}
&\underset{1\leq j\leq N}{\sup}\left|\mathbb{E}\left[X^{j,N}\left(t, x^{j,N} ; t+h\right)-Y^{j,N}_E\left(t, x^{j,N} ; t+h\right)\right]\right|\\
\leq&  \underset{1\leq j\leq N}{\sup} K_{30} \int_t^{t+h} \left\{\mathbb{E}\left[  \left(1+\left|X_{s}^{j,N}\right|^{2\kappa -2}+\left|x^{j,N}\right|^{2\kappa -2}\right)\right]\right\}^{\frac{1}{2}} \left(\mathbb{E}\left[\left|X_{s}^{j,N}-x^{j,N}\right|^{2}\right]\right)^{\frac{1}{2}}\mathrm{d} s\\
\leq & K_{31} \int_t^{t+h}\left\{\mathbb{E}\left[\left(1+\left|x^{j,N}\right|^{2\kappa-2}\right)\right]\right\}^{\frac{1}{2}}\left\{\mathbb{E}\left[\left(1+\left|x^{j,N}\right|^{2\kappa}\right)\right]\right\}^{\frac{1}{2}}\left(s-t\right)^{\frac{1}{2}} \mathrm{d} s \\
\leq&  K_{28}\mathbb{E}\left[\left(1+\left|x^{j,N}\right|^{4 \kappa+6}\right) \right]h^{\frac{3}{2}}.
\end{aligned}
\end{equation}
Furthermore, by leveraging \eqref{2.8}, \eqref{5.4}, \eqref{5.63}, and \eqref{5.65}, along with Assumption~\ref{ass2.2}, H\"older's inequality, and Jensen's inequality, we  derive  for  the second term 
\begin{equation}\tag{5.100}\label{5.72}
\begin{aligned}
& \left|\mathbb{E}\left[Y^{j,N}_{E}\left(t, x^{j,N} ; t+h\right)-Y^{j,N}\left(t, x^{j,N}; t+h\right)\right]\right| \\
\leq& \mathbb{E}\left[\left|x^{j,N}-\psi\left(x^{j,N}\right)\right|\right]+\mathbb{E}\left[\int_t^{t+h}\left| b\left( x^{j,N},\mu ^{x^{j,N},N} \right)-b\left( \psi\left(x^{j,N}\right),\mu ^{\psi\left(x^{j,N}\right),N} \right)\right|\mathrm{d} s\right]\\
\leq &\mathbb{E}\left[\left|x^{j,N}-\psi\left(x^{j,N}\right)\right|\right]\\
&+\sqrt{L_1} \int_t^{t+h} \mathbb{E}\left[ \left( 1+\left|x^{j,N}\right|^{2\kappa -2}+\left|\psi\left(x^{j,N}\right)\right|^{2\kappa -2} \right) \left|x^{j,N}-\psi\left(x^{j,N}\right)\right|^2\right] ^{\frac{1}{2}}\mathrm{d} s\\
&+\sqrt{L_1} \int_t^{t+h} \mathbb{E}\left[\left\{\tfrac{1}{N}\sum_{i=1}^N{\left|x^{i,N}-\psi\left(x^{i,N}\right)\right|^2}\right\}^{\frac{1}{2}}\right] \mathrm{d} s\\
\leq &\mathbb{E}\left[\left|x^{j,N}-\psi\left(x^{j,N}\right)\right|\right]+\sqrt{L_1} \int_t^{t+h} \mathbb{E}\left[ \left( 1+2\left|x^{j,N}\right|^{2\kappa -2}\right)^{\frac{1}{2}} \left|x^{j,N}-\psi\left(x^{j,N}\right)\right|\right] 
\mathrm{d} s\\
&+\sqrt{L_1}\int_t^{t+h} \left(\mathbb{E}\left[\left|x^{j,N}-\psi\left(x^{j,N}\right)\right|^2\right] \right)^{\frac{1}{2}}\mathrm{d} s\\
\leq&2\mathbb{E}\left[\left(1+\left|x^{j,N}\right|^{q_{3}+1}\right)\right] h^{\frac{q_{3}}{2\left(\kappa+2\right)}}+K_{32} \int_t^{t+h} \mathbb{E}\left[\left(1+\left|x^{j,N}\right|^{\kappa -1}\right)\left(1+\left|x^{j,N}\right|^{q_{4}+1}\right)\right]h^{\frac{q_{4}}{2\left(\kappa+2\right)}}\mathrm{d} s\\
&+K_{32} \int_t^{t+h}\left\{\mathbb{E}\left[\left(1+\left|x^{j,N}\right|^{q_{4}+1}\right)^2\right]\right\}^{\frac{1}{2}}h^{\frac{q_{4}}{2\left(\kappa+2\right)}}\mathrm{d} s\\
\leq&2\mathbb{E}\left[\left(1+\left|x^{j,N}\right|^{q_{3}+1}\right)\right] h^{\frac{q_{3}}{2\left(\kappa+2\right)}}+K_{32}\mathbb{E}\left[\left(1+\left|x^{j,N}\right|^{q_{4}+\kappa}\right)\right]h^{\frac{q_{4}}{2\left(\kappa+2\right)}+1}\\
&+K_{32} \mathbb{E}\left[\left(1+\left|x^{j,N}\right|^{2q_{4}+2}\right)\right]h^{\frac{q_{4}}{2\left(\kappa+2\right)}+1}.\\
\end{aligned}
\end{equation}
Taking the supremum over $ j $ in \eqref{5.72} and setting $ q_3 = 3\left(\kappa + 2\right) $ and $ q_4 = \kappa + 2 $, we obtain
\begin{equation}\tag{5.101}\label{5.73}
\begin{aligned}
& \underset{1\leqslant j\leqslant N}{\sup}\left|\mathbb{E}\left[Y^{j,N}_{E}\left(t, x^{j,N} ; t+h\right)-Y^{j,N}\left(t, x^{j,N} ; t+h\right)\right]\right| \\
\leq& 2\mathbb{E}\left[\left(1+\left|x^{j,N}\right|^{3\kappa+7}\right)\right] h^{\frac{3}{2}}+K_{32}\mathbb{E}\left[\left(1+\left|x^{j,N}\right|^{2\kappa+2}\right)\right]h^{\frac{3}{2}}+K_{32}\mathbb{E}\left[\left(1+\left|x^{j,N}\right|^{2\kappa+6}\right)\right]h^{\frac{3}{2}} \\
 \leq& K_{29}\mathbb{E}\left[\left(1+\left|x^{j,N}\right|^{3 \kappa+7}\right)\right] h^{\frac{3}{2}}.
\end{aligned}
\end{equation}
In conclusion, combining \eqref{5.71} and \eqref{5.73} readily yields the bound in \eqref{5.68}, which corresponds to \eqref{5.85} in Theorem~\ref{theorem5.7}. 
This completes the proof.
\end{proof} 

\begin{lemma}\label{lemma5.7}
Assuming that Assumptions~\ref{ass2.1}--\ref{ass2.4} are satisfied, for any $1\leq \kappa \leq \tfrac{q_0-6}{4}$ and $0<h \leq h_2$, we can derive the following estimates:
\begin{equation}\tag{5.102}\label{5.75}
    \underset{1\leqslant j\leqslant N}{\sup} \left\{\mathbb{E}\left[\left|X^{j,N}\left(t, x^{j,N} ; t+h\right)-Y^{j,N}_E\left(t, x^{j,N} ; t+h\right)\right|^{2}\right]\right\}^{\frac{1}{2}} \leq K_{33}\left\{\mathbb{E}\left[\left(1+\left|x^{j,N}\right|^{8\kappa+6}\right)\right]\right\}^{\frac{1}{2}} h,
    \end{equation}
    \begin{equation}\tag{5.103}\label{5.76}
    \underset{1\leqslant j\leqslant N}{\sup} \left\{\mathbb{E}\left[\left|Y^{j,N}_E\left(t, x^{j,N} ; t+h\right)-Y^{j,N}\left(t, x^{j,N} ; t+h\right)\right|^{2}\right]\right\}^{\frac{1}{2}} \leq K_{34}\left\{\mathbb{E}\left[\left(1+\left|x^{j,N}\right|^{6\kappa+8}\right)\right]\right\}^{\frac{1}{2}} h, 
    \end{equation}
    \begin{equation}\tag{5.104}\label{5.77}
    \underset{1\leqslant j\leqslant N}{\sup} \left\{\mathbb{E}\left[\left|X^{j,N}\left(t, x^{j,N} ; t+h\right)-Y^{j,N}\left(t, x^{j,N} ; t+h\right)\right|^{2}\right]\right\}^{\frac{1}{2}} \leq K_{8}\left\{\mathbb{E}\left[\left(1+\left|x^{j,N}\right|^{8\kappa+12 }\right)\right]\right\}^{\frac{1}{2}} h.
    \end{equation}
\end{lemma} 
\begin{proof}  We begin by establishing the first inequality. Applying Assumption~\ref{ass2.2}, noting that the sequence $\left\{ \left|X_s^{i,N} - x^{i,N}\right|^2 \right\}_{1 \leq i \leq N}$ is identically distributed, and utilizing  \eqref{2.8}, \eqref{2.14}, \eqref{interacting}, and \eqref{5.65}, together with the It\^o isometry, H\"older's inequality, Theorem~\ref{momentboundednessofthesolution}, and Lemma~\ref{lemma5.4}, we derive 
\begin{equation}\tag{5.105}\label{5.78}
\begin{aligned}
& \mathbb{E}\left[\left|X^{j,N}\left(t, x^{j,N} ; t+h\right)-Y^{j,N}_E\left(t, x^{j,N} ; t+h\right)\right|^{2}\right] \\
\leq & 2\mathbb{E}\left[\left|\int_t^{t+h} b\left({X}^{j,N}_s,\mu^{X,N}_s\right)-b\left( x^{j,N},\mu ^{x^{j,N},N} \right) \mathrm{d} s\right|^{2}\right]\\
&+2\int_t^{t+h} \mathbb{E}\left[\left|\left|\sigma\left({X}^{j,N}_s,\mu^{X,N}_s\right)-\sigma\left( x^{j,N},\mu ^{x^{j,N},N} \right) \right|\right|^{2}\right]\mathrm{d}s \\
\leq & 2 h\int_t^{t+h} L_1 \mathbb{E}\left[\left(1+\left|X^{j,N}_s\right|^{2 \kappa-2}+\left|x^{j,N}\right|^{2 \kappa-2}\right)\left|X^{j,N}_s-x^{j,N}\right|^{2}+\tfrac{1}{N}\sum_{i=1}^N{\left|X_{s}^{i,N}-x^{i,N}\right|^2}\right] \mathrm{d} s \\
& +2K_{3}\int_t^{t+h} \mathbb{E}\left[\left(1+\left|X^{j,N}_s\right|^{\kappa-1}+\left|x^{j,N}\right|^{\kappa-1}\right)\left|X^{j,N}_s-x^{j,N}\right|^{2}+\tfrac{1}{N}\sum_{i=1}^N{\left|X_{s}^{i,N}-x^{i,N}\right|^2}\right] \mathrm{d} s \\
\leq & K_{35} h\int_t^{t+h}\mathbb{E}\left[\left(1+\left|X^{j,N}_s\right|^{2 \kappa-2}+\left|x^{j,N}\right|^{2 \kappa-2}\right)\left|X^{j,N}_s-x^{j,N}\right|^{2}\right] \mathrm{d} s \\
&+K_{36}\int_t^{t+h} \mathbb{E}\left[\left(1+\left|X^{j,N}_s\right|^{\kappa-1}+\left|x^{j,N}\right|^{\kappa-1}\right)\left|X^{j,N}_s-x^{j,N}\right|^{2}\right] \mathrm{d} s \\
\leq & K_{35} h
\int_t^{t+h} \left\{\mathbb{E}\left[\left(1+\left|X^{j,N}_s\right|^{2 \kappa-2}+\left|x^{j,N}\right|^{2 \kappa-2}\right)^{2}\right]\right\}^{\frac{1}{2 }}\left\{\mathbb{E}\left[\left|X^{j,N}_s-x^{j,N}\right|^{4}\right]\right\}^{\frac{1}{2}}\mathrm{d} s\\
&+K_{36}\int_t^{t+h}\left\{\mathbb{E}\left[\left(1+\left|X^{j,N}_s\right|^{\kappa-1}+\left|x^{j,N}\right|^{\kappa-1}\right)^{2}\right]\right\}^{\frac{1}{2}}\left\{\mathbb{E}\left[\left|X^{j,N}_s-x^{j,N}\right|^{4}\right]\right\}^{\frac{1}{2}}\mathrm{d} s\\
\leq & K_{37}\mathbb{E}\left[\left(1+\left|x^{j,N}\right|^{8 \kappa-4}\right)\right] h^{3}+K_{38}\mathbb{E}\left[\left(1+\left|x^{j,N}\right|^{6 \kappa -2}\right)\right]h^2 \\
\leq &K^{2}_{33}\mathbb{E}\left[\left(1+\left|x^{j,N}\right|^{8 \kappa+6}\right)\right] h^2.\\
\end{aligned}
\end{equation}
Subsequently, taking the supremum over $ j $ yields \eqref{5.75}.

To derive \eqref{5.76}, we invoke Assumption~\ref{ass2.2}, together with  \eqref{2.8}, \eqref{2.14}, \eqref{5.4}, \eqref{5.62}, \eqref{5.63}, and \eqref{5.65}, as well as the It\^o isometry and H\"older's inequality, to obtain
\begin{equation}\tag{5.106}\label{5.80}
\begin{aligned}
  \mathbb{E}&\left[\left|Y^{j,N}_E\left(t, x^{j,N} ; t+h\right)-Y^{j,N}\left(t, x^{j,N} ; t+h\right)\right|^{2}\right] \\
\leq & 3 \mathbb{E}\left[\left|x^{j,N}-\psi\left(x^{j,N}\right)\right|^{2}\right]+3 h\int_t^{t+h} \mathbb{E}\left[\left|b\left( x^{j,N},\mu ^{x^{j,N},N} \right)-b\left( \psi\left(x^{j,N}\right),\mu ^{\psi\left(x^{j,N}\right),N} \right)\right|^{2}\right] \mathrm{d} s\\
&+3\int_t^{t+h}\mathbb{E}\left[\left|\left| \sigma\left( x^{j,N},\mu ^{x^{j,N},N} \right)-\sigma\left( \psi\left(x^{j,N}\right),\mu ^{\psi\left(x^{j,N}\right),N} \right)\right|\right|^{2}\right] \mathrm{d}s  \\
\leq &3\mathbb{E}\left[\left|x^{j,N}-\psi\left(x^{j,N}\right)\right|^{2}\right]+3L_{1}h \int_t^{t+h} \mathbb{E}\left[ \left( 1+\left|x^{j,N}\right|^{2\kappa -2}+\left|\psi\left(x^{j,N}\right)\right|^{2\kappa -2} \right) \left|x^{j,N}-\psi\left(x^{j,N}\right)\right|^2\right.\\
&+\left.\tfrac{1}{N}\sum_{i=1}^N{\left|x^{i,N}-\psi\left(x^{i,N}\right)\right|^2}\right] \mathrm{d} s+3K_{3} \int_t^{t+h} \mathbb{E}\left[ \left( 1+\left|x^{j,N}\right|^{\kappa -1}+\left|\psi\left(x^{j,N}\right)\right|^{\kappa -1} \right) \right.\\
&\times\left.\left|x^{j,N}-\psi\left(x^{j,N}\right)\right|^2+\tfrac{1}{N}\sum_{i=1}^N{\left|x^{i,N}-\psi\left(x^{i,N}\right)\right|^2} \right] \mathrm{d} s\\
\leq &3\mathbb{E}\left[\left|x^{j,N}-\psi\left(x^{j,N}\right)\right|^{2}\right]+{K_{39}}h \int_t^{t+h} \mathbb{E}\left[ \left( 1+\left|x^{j,N}\right|^{2\kappa -2}+\left|\psi\left(x^{j,N}\right)\right|^{2\kappa -2} \right)\right.\\
&\times\left.\left|x^{j,N}-\psi\left(x^{j,N}\right)\right|^2\right] \mathrm{d} s\\
&+{K_{40}} \int_t^{t+h} \mathbb{E}\left[ \left( 1+\left|x^{j,N}\right|^{\kappa -1}+\left|\psi\left(x^{j,N}\right)\right|^{\kappa -1} \right) \left|x^{j,N}-\psi\left(x^{j,N}\right)\right|^2\right] \mathrm{d} s\\
\leq & K_{41}\mathbb{E}\left[\left(1+\left|x^{j,N}\right|^{2q_{5}+2}\right)\right]h^{\frac{q_5}{\kappa+2}}\\
&+K_{42}h\int_t^{t+h}\mathbb{E}\left[\left(1+\left|x^{j,N}\right|^{2 \kappa-2}\right)\left(1+\left|x^{j,N}\right|^{2q_{6}+2}\right)\right]h^{\frac{q_6}{\kappa+2}}\mathrm{d} s\\
&+K_{43}\int_t^{t+h}\mathbb{E}\left[\left(1+\left|x^{j,N}\right|^{\kappa-1}\right)\left(1+\left|x^{j,N}\right|^{2q_{7}+2}\right)\right]h^{\frac{q_7}{\kappa+2}} \mathrm{d} s.\\
\end{aligned}
\end{equation}
In \eqref{5.80}, by setting $ q_5 = q_6 = q_7 = 2\left(\kappa + 2\right) $, we obtain
\begin{equation}\tag{5.107}\label{5.81}
\begin{aligned}
&\mathbb{E}\left[\left|Y^{j,N}_E\left(t, x^{j,N} ; t+h\right)-Y^{j,N}\left(t, x^{j,N} ; t+h\right)\right|^{2}\right] \\
\leq & K_{44}\mathbb{E}\left[\left(1+\left|x^{j,N}\right|^{4\kappa+10}\right)\right]h^2+K_{45} \mathbb{E}\left[\left(1+\left|x^{j,N}\right|^{6\kappa+8}\right)\right]h^4+K_{46}\mathbb{E}\left[ \left(1+\left|x^{j,N}\right|^{5\kappa+9}\right)\right]h^3\\
\leq & K^{2}_{34}\mathbb{E}\left[\left(1+\left|x^{j,N}\right|^{6\kappa+8}\right)\right] h^2.
\end{aligned}
\end{equation}
Taking the supremum over $j$ yields (\ref{5.76}).

Combining \eqref{5.75} and \eqref{5.76} readily establishes \eqref{5.77}, which corresponds to \eqref{5.86} in Theorem~\ref{theorem5.7}. 
This completes the proof of the lemma. 
\end{proof}

\subsection{Proof of Theorem \ref{momentboundsforBEM}}
For this part, we aim to demonstrate the moment boundedness for the backward Euler method.
\begin{remark}
    From (\ref{2.16}), we derive
    \begin{equation}\tag{5.108}\label{6.3}
    \left|\left|\sigma \left( 0,\delta _0 \right) \right|\right|^2\leq K_5\left( \left|0\right|^{\kappa +1}+1 \right) +a_4\mathcal{W} _2\left( \delta _0 \right) ^2 = K_5.
    \end{equation}
     Consequently, we obtain
\begin{equation}\tag{5.109}\label{6.4}
\begin{aligned}
\left|\left|\sigma \left( x,\mu \right) \right|\right|^2
&\leq 2\left[ \left|\left|\sigma \left( x,\mu \right) -\sigma \left( 0,\delta _0 \right) \right|\right|^2+\left|\left|\sigma \left( 0,\delta _0 \right) \right|\right|^2 \right] \\
&\leq 2\left\{ L_{\sigma}\left[ \left|x\right|^2+\mathcal{W} _2\left( \mu \right) ^2 \right] +K_5 \right\}.\\
\end{aligned}
\end{equation}
\end{remark}
 With these results, we are able to finish the proof of Theorem~\ref{momentboundsforBEM}.
\begin{proof}[Proof of Theorem~\ref{momentboundsforBEM}]
Applying (\ref{6.1}) and Assumption~\ref{ass2.4}, we obtain
\begin{equation}\tag{5.110}\label{6.6}
\left( 1+a_1h \right) ^p\left| \hat{X}_{k+1}^{j,N} \right|^{2p}\le \left( \left| \hat{X}_{k}^{j,N}+\sigma \left( \hat{X}_{k}^{j,N},\mu _{k}^{X,N} \right) \Delta W_{k}^{j} \right|^2+a_2h\mathcal{W} _2\left( \mu _{k}^{X,N} \right) ^2+Ch \right) ^p.
\end{equation}
We begin with the case $ p = 1 $. Taking expectations on both sides of \eqref{6.6} and applying \eqref{2.6} and \eqref{6.4}, we obtain, for any $ 1 \leq j \leq N $,
\begin{equation}\tag{5.111}\label{6.7}
\begin{aligned}
&\left( 1+a_1h \right) \mathbb{E} \left[ \left| \hat{X}_{k+1}^{j,N} \right|^2 \right] \\
\le &\mathbb{E} \left[ \left| \hat{X}_{k}^{j,N} \right|^2 \right] +h\mathbb{E} \left[ \left|\left| \sigma \left( \hat{X}_{k}^{j,N},\mu _{k}^{X,N} \right) \right|\right| ^2 \right] +a_2h\mathbb{E} \left[ \left| \hat{X}_{k}^{j,N} \right|^2 \right] +Ch
\\
\le &\mathbb{E} \left[ \left| \hat{X}_{k}^{j,N} \right|^2 \right] +2h\left\{ L_{\sigma}\left[ \mathbb{E} \left[ \left|\hat{X}_{k}^{j,N}\right|^2 \right] +\mathbb{E} \left[ \mathcal{W} _2\left( \mu _{k}^{X,N} \right) ^2 \right] \right] +K_5 \right\} +a_2h\mathbb{E} \left[ \left| \hat{X}_{k}^{j,N} \right|^2 \right] +Ch
\\
\le &\left( 1+\left( 4L_{\sigma}+a_2 \right) h \right) \mathbb{E} \left[ \left| \hat{X}_{k}^{j,N} \right|^2 \right] +\left( 2K_5+C \right) h.\\
\end{aligned}
\end{equation}
Since $0<L_{\sigma}<\tfrac{a_1-a_2}{5\left( 2q_0-1 \right)}<\tfrac{a_1-a_2}{4}$, it follows that $\tfrac{1+\left( 4L_{\sigma}+a_2 \right) h}{1+a_1h}=1-\gamma_{1} h$, where $\gamma_{1} =\tfrac{a_1-a_2-4L_{\sigma}}{1+a_1h}>0$. Applying \eqref{6.7} and Lemma~\ref{lemma5.2}, we obtain
\begin{equation}\tag{5.112}\label{6.8}
\begin{aligned}
	\mathbb{E} \left[ \left| \hat{X}_{k}^{j,N} \right|^2 \right] &\le \left(1-\gamma_{1} h\right)\mathbb{E} \left[ \left| \hat{X}_{k-1}^{j,N} \right|^2 \right] +\tfrac{2K_5+C}{1+a_1h}h\\
	&\le K_{47}\mathbb{E} \left[ \left( 1+\left| X^{j}_0 \right|^2 \right) \right].\\
\end{aligned}
\end{equation}
Taking the supremum over $ j $ yields \eqref{6.5} for the case $ p = 1 $.

For a positive integer  $p > 1$,   applying the binomial expansion yields, for any $1 \leq j \leq N$, 
\begin{equation}\tag{5.113}\label{6.9}
\begin{aligned}
&\mathbb{E} \left[\left. \left( \left| \hat{X}_{k}^{j,N}+\sigma \left( \hat{X}_{k}^{j,N},\mu _{k}^{X,N} \right) \Delta W_{k}^{j} \right|^2+a_2h\mathcal{W} _2\left( \mu _{k}^{X,N} \right) ^2+Ch \right) ^p\right\rvert\ \mathcal{F}_{t_{k}} \right] 
\\
&+\sum_{l=1}^p{C_{p}^{l}\left( Ch \right) ^l\mathbb{E} \left[\left. \left( \left| \hat{X}_{k}^{j,N}+\sigma \left( \hat{X}_{k}^{j,N},\mu _{k}^{X,N} \right) \Delta W_{k}^{j} \right|^2+a_2h\mathcal{W} _2\left( \mu _{k}^{X,N} \right) ^2 \right) ^{p-l}\right\rvert\ \mathcal{F}_{t_{k}} \right]}
\\
=&\mathbb{E} \left[\left. \left( \left| \hat{X}_{k}^{j,N}+\sigma \left( \hat{X}_{k}^{j,N},\mu _{k}^{X,N} \right) \Delta W_{k}^{j} \right|^2+a_2h\mathcal{W} _2\left( \mu _{k}^{X,N} \right) ^2 \right) ^p\right\rvert\ \mathcal{F}_{t_{k}} \right] 
\\
&+\sum_{l=1}^p C_{p}^{l} \left( Ch \right)^l\left\{\mathbb{E}\left[\left. \left|\hat{X}_{k}^{j,N}+\sigma \left( \hat{X}_{k}^{j,N},\mu_{k}^{X,N} \right) \Delta W_{k}^{j} \right|^{2\left(p-l\right)} \right\rvert\ \mathcal{F}_{t_{k}}\right] \right.\\
&\left.+ \sum_{r=1}^{p-l}{C_{p-l}^{r}\left( a_2h \right) ^r\mathbb{E} \left[\left. \left| \hat{X}_{k}^{j,N}+\sigma \left( \hat{X}_{k}^{j,N},\mu _{k}^{X,N} \right) \Delta W_{k}^{j} \right|^{2\left( p-l-r \right)}\mathcal{W} _2\left( \mu _{k}^{X,N} \right) ^{2r} \right\rvert\ \mathcal{F}_{t_{k}}\right]} \right\}\\
:=&I_{1}^{j,N}+\sum_{l=1}^p{C_{p}^{l}\left( Ch \right) ^l\left\{ \mathbb{E} \left[\left. \left| \hat{X}_{k}^{j,N}+\sigma \left( \hat{X}_{k}^{j,N},\mu _{k}^{X,N} \right) \Delta W_{k}^{j} \right|^{2\left(p-l\right)}\right\rvert\ \mathcal{F}_{t_{k}} \right]+ \sum_{r=1}^{p-l}{C_{p-l}^{r}\left( a_2h \right) ^rI_{2}^{j,N}} \right\}}.
\end{aligned}
\end{equation}\\
For the term $ I_1^{j,N} $, by applying \eqref{2.6} along with the binomial expansion, Young's inequality, and the $\mathcal{F}_{t_k}$-measurability of $\hat{X}_k^{j,N}$, we derive
\begin{equation}\tag{5.114}\label{6.10}
\begin{aligned}
&\mathbb{E} \left[ \left.\left( \left| \hat{X}_{k}^{j,N}+\sigma \left( \hat{X}_{k}^{j,N},\mu _{k}^{X,N} \right) \Delta W_{k}^{j} \right|^2+a_2h\mathcal{W} _2\left( \mu _{k}^{X,N} \right) ^2 \right) ^p \right\rvert\ \mathcal{F}_{t_{k}}\right] 
\\
\leq&  \left| \hat{X}_{k}^{j,N} \right|^{2p}  +\sum_{r=1}^p{C_{2p}^{2r}\mathbb{E} \left[\left. \left| \sigma \left( \hat{X}_{k}^{j,N},\mu _{k}^{X,N} \right) \Delta W_{k}^{j} \right|^{2r}\left| \hat{X}_{k}^{j,N} \right|^{2p-2r}\right\rvert\ \mathcal{F}_{t_{k}}  \right]}
\\
&+\sum_{l=1}^p{C_{p}^{l}\left( a_2h \right) ^l \tfrac{l}{p}\mathbb{E} \left[\left. \mathcal{W} _2\left( \mu _{k}^{X,N} \right) ^{2p}\right\rvert\ \mathcal{F}_{t_{k}}  \right]} \\
&+\sum_{l=1}^p{C_{p}^{l}\left( a_2h \right) ^l\tfrac{p-l}{p}\mathbb{E} \left[\left. \left| \hat{X}_{k}^{j,N}+\sigma \left( \hat{X}_{k}^{j,N},\mu _{k}^{X,N} \right) \Delta W_{k}^{j} \right|^{2p} \right\rvert\ \mathcal{F}_{t_{k}} \right] }
\\
\leq& \left| \hat{X}_{k}^{j,N} \right|^{2p} +\sum_{r=1}^p{C_{2p}^{2r}\mathbb{E} \left[ \left.\left| \sigma \left( \hat{X}_{k}^{j,N},\mu _{k}^{X,N} \right) \Delta W_{k}^{j} \right|^{2r}\left| \hat{X}_{k}^{j,N} \right|^{2p-2r} \right\rvert\ \mathcal{F}_{t_{k}} \right]}
\\
&+\sum_{l=1}^p{C_{p}^{l}\left( a_2h \right) ^l \tfrac{l}{p} \left| \hat{X}_{k}^{j,N} \right|^{2p} }\\
&+\sum_{l=1}^p{C_{p}^{l}\left( a_2h \right) ^l\tfrac{p-l}{p}\left\{\left| \hat{X}_{k}^{j,N} \right|^{2p}+\sum_{r=1}^p{C_{2p}^{2r}\mathbb{E} \left[ \left.\left| \sigma \left( \hat{X}_{k}^{j,N},\mu _{k}^{X,N} \right) \Delta W_{k}^{j} \right|^{2r}\left| \hat{X}_{k}^{j,N} \right|^{2p-2r}\right\rvert\ \mathcal{F}_{t_{k}}  \right]} \right\}}.
\end{aligned}
\end{equation}\\
By applying \eqref{2.6}, \eqref{6.4}, the inner product inequality, the binomial expansion, the fact  that $\mathbb{E} \left[ \left.\left|\Delta W_{k}^j\right|^{2r}\right\rvert\ \mathcal{F}_{t_k} \right] = \left(2r - 1\right)!! \, h^r$, along with Jensen's inequality, we derive, for $r \geq 1$,
 \begin{equation}\tag{5.115}\label{6.11}
\begin{aligned}
&\mathbb{E} \left[\left. \left| \sigma \left( \hat{X}_{k}^{j,N},\mu _{k}^{X,N} \right) \Delta W_{k}^{j} \right|^{2r}\left| \hat{X}_{k}^{j,N} \right|^{2p-2r}\right\rvert\ \mathcal{F} _{t_k} \right]\\
\le &2^r\mathbb{E} \left[\left. \left\{ L_{\sigma}\left[ \left|\hat{X}_{k}^{j,N}\right|^2+\tfrac{1}{N}\sum_{i=1}^N{\left|\hat{X}_{k}^{i,N}\right|^2} \right] +K_5 \right\} ^r\left| \hat{X}_{k}^{j,N} \right|^{2p-2r}\left| \Delta W_{k}^{j} \right|^{2r}\right\rvert\ \mathcal{F} _{t_k} \right]\\
\le& 2^rL_{\sigma}^{r}\mathbb{E} \left[\left. \left\{ \tfrac{1}{N}\sum_{i=1}^N{\left( \left|\hat{X}_{k}^{i,N}\right|^2+\left|\hat{X}_{k}^{j,N}\right|^2 \right)} \right\} ^r\left| \hat{X}_{k}^{j,N} \right|^{2p-2r}\left| \Delta W_{k}^{j} \right|^{2r}\right\rvert\ \mathcal{F} _{t_k} \right]\\
&+2^r\sum_{l=1}^r{C_{r}^{l}K_{5}^{l}L_{\sigma}^{r-l}\mathbb{E} \left[\left. \left\{ \tfrac{1}{N}\sum_{i=1}^N{\left( \left|\hat{X}_{k}^{i,N}\right|^2+\left|\hat{X}_{k}^{j,N}\right|^2 \right)} \right\} ^{r-l}\left| \hat{X}_{k}^{j,N} \right|^{2p-2r}\left| \Delta W_{k}^{j} \right|^{2r}\right\rvert\ \mathcal{F} _{t_k} \right]}\\
\le& 2^rL_{\sigma}^{r}\left( 2r-1 \right) !!h^r\tfrac{1}{N}\sum_{i=1}^N{\left( \left|\hat{X}_{k}^{i,N}\right|^2+\left|\hat{X}_{k}^{j,N}\right|^2 \right) ^r\left| \hat{X}_{k}^{j,N} \right|^{2p-2r}}\\
&+2^r\left( 2r-1 \right) !!h^r\sum_{l=1}^r{C_{r}^{l}K_{5}^{l}L_{\sigma}^{r-l}\tfrac{1}{N}\sum_{i=1}^N{\left( \left|\hat{X}_{k}^{i,N}\right|^2+\left|\hat{X}_{k}^{j,N}\right|^2 \right) ^{r-l}\left| \hat{X}_{k}^{j,N} \right|^{2p-2r}}}.\\
\end{aligned}
\end{equation}
Since the sequence $\left\{ \hat{X}_k^{j,N} \right\}_{1 \leq j \leq N}$ is identically distributed, the sequences $\left\{ \left|\hat{X}_k^{j,N}\right|^{2p} \right\}_{1 \leq j \leq N}$ and $\left\{ \left|\hat{X}_k^{j,N}\right|^{2p - 2l} \right\}_{1 \leq j \leq N}$ are also identically distributed. Taking expectations on both sides of \eqref{6.11} and applying H\"older's inequality, we derive
\begin{equation}\tag{5.116}\label{6.13}
\begin{aligned}
&\mathbb{E} \left[ \left| \sigma \left( \hat{X}_{k}^{j,N},\mu _{k}^{X,N} \right) \Delta W_{k}^{j} \right|^{2r}\left| \hat{X}_{k}^{j,N} \right|^{2p-2r} \right] 
\\
\le& 2^rL_{\sigma}^{r}\left( 2r-1 \right) !!h^r\tfrac{1}{N}\sum_{i=1}^N{\left\{ 2^{r-1}\left( \mathbb{E} \left[ \left| \hat{X}_{k}^{i,N} \right|^{2r}\left| \hat{X}_{k}^{j,N} \right|^{2p-2r} \right] +\mathbb{E} \left[ \left| \hat{X}_{k}^{j,N} \right|^{2p} \right] \right) \right\}}
\\
&+2^r\left( 2r-1 \right) !!h^r\sum_{l=1}^r{C_{r}^{l}K_{5}^{l}L_{\sigma}^{r-l}\tfrac{1}{N}\sum_{i=1}^N{\left\{ 2^{r-l-1} \right.}}
\\
 &\times\left. \left( \mathbb{E} \left[ \left| \hat{X}_{k}^{i,N} \right|^{2r-2l}\left| \hat{X}_{k}^{j,N} \right|^{2p-2r} \right] +\mathbb{E} \left[ \left| \hat{X}_{k}^{j,N} \right|^{2p-2l} \right] \right) \right\} 
\\
\le& 2^{2r-1}L_{\sigma}^{r}\left( 2r-1 \right) !!h^r\tfrac{1}{N}\sum_{i=1}^N{\left\{ \left( \mathbb{E} \left[ \left| \hat{X}_{k}^{i,N} \right|^{2p} \right] \right) ^{\frac{r}{p}}\left( \mathbb{E} \left[ \left| \hat{X}_{k}^{j,N} \right|^{2p} \right] \right) ^{\frac{p-r}{p}}+\mathbb{E} \left[ \left| \hat{X}_{k}^{j,N} \right|^{2p} \right] \right\}}
\\
&+2^{2r-l-1}\left( 2r-1 \right) !!h^r\sum_{l=1}^r{C_{r}^{l}K_{5}^{l}L_{\sigma}^{r-l}\tfrac{1}{N}\sum_{i=1}^N{\left\{ \left( \mathbb{E} \left[ \left| \hat{X}_{k}^{i,N} \right|^{2\left( p-l \right)} \right] \right) ^{\frac{r-l}{p-l}} \right.}}
\\
 &\times \left.\left( \mathbb{E} \left[ \left| \hat{X}_{k}^{j,N} \right|^{2\left( p-l \right)} \right] \right) ^{\frac{p-r}{p-l}}+\mathbb{E} \left[ \left| \hat{X}_{k}^{j,N} \right|^{2\left( p-l \right)} \right] \right\} \\
\le& 4^rL_{\sigma}^{r}\left( 2r-1 \right) !!h^r\mathbb{E} \left[ \left| \hat{X}_{k}^{j,N} \right|^{2p} \right] +K_{48}h\mathbb{E} \left[ \left( 1+\left| \hat{X}_{k}^{j,N} \right|^{2p-2} \right) \right].\\
\end{aligned}
\end{equation}
Taking expectations on both sides of \eqref{6.10} and applying \eqref{6.13}, we obtain
\begin{equation}\tag{5.117}\label{6.14}
\begin{aligned}
\mathbb{E} \left[I_{1}^{j,N}\right]\leq & \mathbb{E} \left[\left| \hat{X}_{k}^{j,N} \right|^{2p}\right] +\sum_{r=1}^p{C_{2p}^{2r} \left( 2r-1 \right) !!h^r4^rL_{\sigma}^{r} \mathbb{E} \left[\left| \hat{X}_{k}^{j,N} \right|^{2p}\right]  }\\
&+K_{49}h \mathbb{E} \left[\left(1+\left|\hat{X}_{k}^{j,N}\right|^{2p-2}\right)\right] +\sum_{l=1}^p{C_{p}^{l}\left( a_2h \right) ^l\mathbb{E} \left[\left| \hat{X}_{k}^{j,N} \right|^{2p}\right] }\\
&+\sum_{l=1}^p{C_{p}^{l}\left( a_2h \right) ^l\tfrac{p-l}{p} \sum_{r=1}^p{C_{2p}^{2r} \left( 2r-1 \right) !!h^r4^rL_{\sigma}^{r} \mathbb{E} \left[\left| \hat{X}_{k}^{j,N} \right|^{2p}\right]  } }\\
&+K_{49}h \mathbb{E} \left[\left(1+\left|\hat{X}_{k}^{j,N}\right|^{2p-2}\right)\right].\\
\end{aligned}
\end{equation}
By applying \eqref{6.13} in conjunction with the binomial expansion, we obtain
\begin{equation}\tag{5.118}\label{6.15}
\begin{aligned}
&\mathbb{E} \left[ \left| \hat{X}_{k}^{j,N}+\sigma \left( \hat{X}_{k}^{j,N},\mu _{k}^{X,N} \right) \Delta W_{k}^{j} \right|^{2\left( p-l \right)}\right] 
\\
\le&  \mathbb{E} \left[ \left| \hat{X}_{k}^{j,N} \right|^{2\left( p-l \right)}\right]  +\sum_{r=1}^{p-l}{C_{2\left( p-l \right)}^{2r}\left( 2r-1 \right) !!h^r4^rL_{\sigma}^{r}\mathbb{E} \left[ \left| \hat{X}_{k}^{j,N} \right|^{2\left( p-l \right)}\right]} \\
&+K_{50}h\mathbb{E} \left[  \left(1+\left|\hat{X}_{k}^{j,N}\right|^{2\left( p-l-1 \right)}\right)\right] 
\\
\le& K_{51}h \mathbb{E} \left[ \left(1+\left| \hat{X}_{k}^{j,N} \right|^{2\left( p-1 \right)}\right)\right].
\end{aligned}
\end{equation}\\
For the term $ I_2^{j,N} $, by applying \eqref{2.6} and \eqref{6.13} in conjunction with the binomial expansion and Young's inequality, we derive, for $ l \geq 1 $,
\begin{equation}\tag{5.119}\label{6.16}
\begin{aligned}
&\mathbb{E} \left[I_{2}^{j,N}\right]\\
\le& \tfrac{r}{p-l}\mathbb{E} \left[ \mathcal{W} _2\left( \mu _{k}^{X,N} \right) ^{2\left( p-l \right)} \right]+\tfrac{p-l-r}{p-l}\mathbb{E} \left[ \left| \hat{X}_{k}^{j,N}+\sigma \left( \hat{X}_{k}^{j,N},\mu _{k}^{X,N} \right) \Delta W_{k}^{j} \right|^{2\left( p-l \right)} \right] 
\\
\le&  \mathbb{E} \left[\left| \hat{X}_{k}^{j,N} \right|^{2\left( p-l \right)}\right]  +\tfrac{p-l-r}{p-l}\sum_{s=1}^{p-l}{C_{2\left( p-l \right)}^{2s}\left( 2s-1 \right) !!h^s4^sL_{\sigma}^{s} \mathbb{E} \left[\left| \hat{X}_{k}^{j,N} \right|^{2\left( p-l \right)}\right] }\\
&+K_{52}h\mathbb{E} \left[\left( 1+\left|\hat{X}_{k}^{j,N}\right|^{2\left( p-l-1 \right)} \right)\right]
\\
\le& K_{53}h\mathbb{E} \left[\left( 1+\left| \hat{X}_{k}^{j,N} \right|^{2\left( p-1 \right)} \right)\right].
\end{aligned}
\end{equation}\\
Substituting inequalities \eqref{6.14}, \eqref{6.15}, and \eqref{6.16} into \eqref{6.9} yields 
\begin{equation}\tag{5.120}\label{6.17}
\begin{aligned}
&\left( 1+a_1h \right) ^p\mathbb{E} \left[ \left| \hat{X}_{k+1}^{j,N} \right|^{2p} \right] 
\\
\le& \left[ 1+\sum_{r=1}^p{C_{2p}^{2r}\left( 2r-1 \right) !!h^r4^rL_{\sigma}^{r}}\right.\\
&\left.+\sum_{l=1}^p{C_{p}^{l}\left( a_2h \right) ^l\left( 1+\tfrac{p-l}{p}\sum_{r=1}^p{C_{2p}^{2r}\left( 2r-1 \right) !!h^r4^rL_{\sigma}^{r}} \right)} \right] \mathbb{E} \left[ \left| \hat{X}_{k}^{j,N} \right|^{2p}\right] 
\\
&+K_{54}h \mathbb{E} \left[\left( 1+\left|\hat{X}_{k}^{j,N}\right|^{2p-2} \right)\right]. 
\end{aligned}
\end{equation}\\
Furthermore, utilizing the condition $0 < L_\sigma < \frac{a_1 - a_2}{5\left(2q_0 - 1\right)}$ and $1 \leq p < \left\lfloor \frac{q_0}{2} \right\rfloor \leq\frac{q_0}{2}$, we obtain 
\begin{equation}\tag{5.121}\label{6.18}
\begin{aligned}
&\mathbb{E} \left[ \left| \hat{X}_{k+1}^{j,N} \right|^{2p}\right] 
\\
\le &\left\{ 1-\left[ \tfrac{\left( 1+a_1h \right) ^p-\left( 1+\sum_{r=1}^p{C_{2p}^{2r}\left( 2r-1 \right) !!h^r4^rL_{\sigma}^{r}} \right)}{\left( 1+a_1h \right) ^p} \right] \right\}  \mathbb{E} \left[\left| \hat{X}_{k}^{j,N} \right|^{2p} \right]
\\
&+\left\{ \tfrac{\sum_{l=1}^p{C_{p}^{l}\left( a_2h \right) ^l\left[ 1+\tfrac{p-l}{p}\sum_{r=1}^p{C_{2p}^{2r}\left( 2r-1 \right) !!h^r4^rL_{\sigma}^{r}} \right]}}{\left( 1+a_1h \right) ^p} \right\}  \mathbb{E} \left[\left| \hat{X}_{k}^{j,N} \right|^{2p}\right] 
\\
&+K_{55}h \mathbb{E} \left[\left( 1+\left|\hat{X}_{k}^{j,N}\right|^{2p-2} \right)\right] 
\\
\le& \left\{ 1-\left( \tfrac{\sum_{l=1}^p{C_{p}^{l}h^{l-1}\left[ a_{1}^{l}-a_{2}^{l}\left( 1+\tfrac{p-l}{p}\sum_{r=1}^p{C_{2p}^{2r}\left( 2r-1 \right) !!h^r4^r\tfrac{\left( a_1-a_2 \right) ^r}{5^r\left( 4p-1 \right) ^r}} \right) \right]}}{\left( 1+a_1h \right) ^p} \right) h \right\}  \mathbb{E} \left[\left| \hat{X}_{k}^{j,N} \right|^{2p}\right]  
\\
&+ \tfrac{\sum_{r=1}^p{C_{2p}^{2r}\left( 2r-1 \right) !!h^{r-1}4^r\tfrac{\left( a_1-a_2 \right) ^r}{5^r\left( 4p-1 \right) ^r}}}{\left( 1+a_1h \right) ^p} h  \mathbb{E} \left[\left| \hat{X}_{k}^{j,N} \right|^{2p}\right] 
\\
&+K_{55}h \mathbb{E} \left[\left( 1+\left|\hat{X}_{k}^{j,N}\right|^{2p-2} \right)\right].
\end{aligned}
\end{equation}\\
It is straightforward to verify by mathematical induction that the inequality is valid for every positive integers $p$ satisfying $1 \leq p < \left\lfloor \frac{q_0}{2} \right\rfloor$, under the conditions $a_1 > a_2 \geq 0$ and $0 < h \leq h_2$.
\begin{equation}\tag{5.122}\label{6.19}
\begin{aligned}
\gamma_{2}:&=\tfrac{1}{\left( 1+a_1h \right) ^p}\sum_{l=1}^{p}{\tfrac{p!}{l!\left( p-l \right) !}h^{l-1}}\\
&\times\left[ a_{1}^{l}-a_{2}^{l}\left( 1+\tfrac{p-l}{p}\sum_{r=1}^p{\tfrac{\left( 2p \right) !}{\left( 2r \right) !\left( 2p-2r \right) !}4^r\tfrac{\left( a_1-a_2 \right) ^r}{5^r\left( 4p-1 \right) ^r}\left( 2r-1 \right) !!h^r} \right) \right]\\
&-\tfrac{1}{\left( 1+a_1h \right) ^p}\sum_{r=1}^p{\tfrac{\left( 2p \right) !}{\left( 2r \right) !\left( 2p-2r \right) !}4^r\tfrac{\left( a_1-a_2 \right) ^r}{5^r\left( 4p-1 \right) ^r}\left( 2r-1 \right) !!h^{r-1}}>0.
\end{aligned}
\end{equation}\\
Using this notation, we can rewrite \eqref{6.18} as 
\begin{equation}\tag{5.123}\label{6.20}
\mathbb{E} \left[ \left| \hat{X}_{k+1}^{j,N} \right|^{2p} \right] \le \left(1-\gamma_{2} h\right)\mathbb{E} \left[ \left| \hat{X}_{k}^{j,N} \right|^{2p} \right] +K_{55}h\mathbb{E} \left[ \left( 1+\left| \hat{X}_{k}^{j,N} \right|^{2p-2} \right) \right], 
\end{equation}
where $\gamma_{2}>0$.

By invoking Lemma \ref{lemma5.2} and subsequently taking the supremum with respect to $j$ establishes inequality \eqref{6.5} for all positive integers $p$ satisfying $1 \leq p < \left\lfloor \frac{q_0}{2} \right\rfloor$.  
For non-integer values of $p$, following an approach analogous to the proof of Theorem~\ref{momentboundsforPEM}, we readily establish that inequality \eqref{6.5} holds for all $p$ satisfying $1 \leq p < \left\lfloor \frac{q_0}{2} \right\rfloor$. 
The theorem is thus proved.
\end{proof} 
\subsection{Proof of Theorem \ref{theorem6.7}}
The one-step approximation of the Euler-Maruyama scheme can be expressed as follows:
\begin{equation}\tag{5.124}\label{6.21}
Y^{j,N}_{E}\left( t_{0},X^{j,N}_{0};t_{0}+h \right) :=X^{j,N}_{0}+\int_{t_{0}}^{t_{0}+h}{b\left( X^{j,N}_{0},\mu _{t_{0}}^{X,N} \right)}ds+\int_{t_{0}}^{t_{0}+h}{\sigma \left( X^{j,N}_{0},\mu _{t_{0}}^{X,N} \right)}dW^j\left( s \right).
\end{equation}

Subtracting (\ref{6.21}) from (\ref{6.22}) yields
\begin{equation}\tag{5.125}\label{6.23}
\begin{aligned}
   \hat{X}^{j,N}\left( t_0,X^{j,N}_0;t_0+h \right) :&=Y_{E}^{j,N}\left( t_0,X^{j,N}_0;t_0+h \right)\\
&+\underbrace{\int_{t_0}^{{t_0}+h}{b\left( \hat{X}^{j,N}\left( t_0,X^{j,N}_0;t_0+h \right) ,\mu _{t_0}^{X,N} \right) -b\left( X^{j,N}_0,\mu _{t_0}^{X,N} \right)}\mathrm{d}s}_{:=I_{b}^{j,N}}.
\end{aligned}
\end{equation}

To establish the local weak and strong errors as specified in \eqref{4.19} and \eqref{4.20}, we first introduce the following useful lemmas.

\begin{lemma}\label{lemma6.4}
    Assuming that Assumptions~\ref{ass2.1}--\ref{ass2.4} and \ref{ass6.1} are satisfied, for any $1 \leq \kappa \leq \tfrac{q_0}{4}$ and $0 < h \leq h_2$, we can derive the following estimates:
\begin{equation}\tag{5.126}\label{6.24}
\underset{1\leqslant j\leqslant N}{\mathrm{sup}}\mathbb{E} \left[ \left|\hat{X}^{j,N}\left( t_0,X^{j,N}_0;t_0+h \right) -X^{j,N}_0\right|^2 \right] \le C_{17}\mathbb{E} \left[\left( 1+\left|X_0\right|^{4\kappa}\right) \right] h,
\end{equation}\\
\begin{equation}\tag{5.127}\label{6.25}
\underset{1\leqslant j\leqslant N}{\mathrm{sup}}\mathbb{E} \left[ \left|\hat{X}^{j,N}\left( t_0,X^{j,N}_0;t_0+h \right) -X^{j,N}_0\right|^4 \right] \le C_{18}\mathbb{E} \left[\left( 1+\left|X_0\right|^{8\kappa}\right) \right] h^{2},
\end{equation}\\
where $C_{17}$ and $C_{18}$ are constants independent of $h,N$.
\end{lemma}
\begin{proof}
   Similarly to the proof of Lemma~\ref{lemma5.4}.
\end{proof}
\begin{lemma}\label{lemma6.5}
  Assuming that Assumptions~\ref{ass2.1}--\ref{ass2.4} and \ref{ass6.1} are satisfied, for any $1 \leq \kappa \leq \tfrac{q_0 - 3}{2}$ and $0 < h \leq h_2$, we can derive the following estimates:
\begin{equation}\tag{5.128}\label{6.28}
\underset{1\leqslant j\leqslant N}{\mathrm{sup}}\left| \mathbb{E} \left[ X^{j,N}\left( t_0,X^{j,N}_0;t_0+h \right) -X^{j,N}_0 \right] \right|\le K_{56}\mathbb{E} \left[ \left(1+\left|X_0\right|^{2\kappa}\right) \right] h,
\end{equation}\\
\begin{equation}\tag{5.129}\label{6.29}
\underset{1\leqslant j\leqslant N}{\mathrm{sup}}\left| \mathbb{E} \left[ Y_{E}^{j,N}\left( t_0,X^{j,N}_0;t_0+h \right) -X^{j,N}_0 \right] \right|\le K_{57}\mathbb{E} \left[\left( 1+\left|X_0\right|^{4\kappa +6}\right) \right] h,
\end{equation}\\
where $K_{56}$ and $K_{57}$ are constants independent of $h, N$.
\end{lemma}
\begin{proof}
By applying \eqref{2.6}, \eqref{2.15}, the characteristics of the It\^o integral, Theorem~\ref{momentboundednessofthesolution}, as well as Jensen's inequality, we derive, for any $1 \leq j \leq N$,
\begin{equation}\tag{5.130}\label{6.30}
\begin{aligned}
&\left| \mathbb{E} \left[ X^{j,N}\left( t_0,X^{j,N}_0;t_0+h \right) -X^{j,N}_0 \right] \right|
\\
\le& \int_{t_0}^{t_0+h}{\mathbb{E} \left[ \left| b\left( X^{j,N}\left( t_0,X^{j,N}_0;s \right) ,\mu _{s}^{X,N} \right) \right| \right]}ds
\\
\le& \int_{t_0}^{t_0+h}K_4^{\frac{1}{2}}\mathbb{E} \left[ \left( \left|X^{j,N}\left( t_0,X^{j,N}_0;s \right) \right|^{2\kappa}+1 \right) ^{\frac{1}{2}} \right]+a_3^{\frac{1}{2}}\mathbb{E} \left[ \left( \tfrac{1}{N}\sum_{i=1}^N{\left|X^{i,N}\left( t_0,X^{i,N}_0;s \right) \right|^2} \right) ^{\frac{1}{2}} \right]\mathrm{d}s
\\
\le& K_{56}\mathbb{E} \left[ \left(1+\left|X^{j,N}_0\right|^{2\kappa}\right) \right] h.
\end{aligned}
\end{equation}
This establishes \eqref{6.28}. 
Combining \eqref{5.66} and \eqref{6.28} yields \eqref{6.29}. 
The lemma is thus proved.
\end{proof}

\begin{lemma}\label{lemma6.6}
    Assuming that Assumptions~\ref{ass2.1}--\ref{ass2.4} and \ref{ass6.1} are satisfied, for any $1 \leq \kappa \leq \tfrac{q_0 - 1}{4}$ and $0 < h \leq h_2$, we can derive the following estimates:
\begin{equation}\tag{5.131}\label{6.32}
\underset{1\leqslant j\leqslant N}{\mathrm{sup}}\left|\mathbb{E}\left [X^{j,N}\left( t_0,X^{j,N}_0;t_0+h \right) -\hat{X}^{j,N}\left( t_0,X^{j,N}_0;t_0+h \right) \right]\right|\le K_{10}\mathbb{E} \left[\left( 1+\left|X_0\right|^{8\kappa+2}\right) \right] h^{\frac{3}{2}},
\end{equation}
\begin{equation}\tag{5.132}\label{6.33}
\underset{1\leqslant j\leqslant N}{\mathrm{sup}}\left|\mathbb{E}\left [\hat{X}^{j,N}\left( t_0,X^{j,N}_0;t_0+h \right) \right]\right|\le K_{58}\mathbb{E} \left[ \left(1+\left|X_0\right|^{4\kappa}\right) \right] h,
\end{equation}
where $ K_{10}$ and $ K_{58}$ are constants independent of $h, N$.
\end{lemma}
\begin{proof} 
To establish the result, we apply Assumption~\ref{ass2.2}, together with \eqref{6.23}, \eqref{6.24}, H\"older's inequality, and Theorem~\ref{momentboundsforBEM}, to derive
\begin{equation}\tag{5.133}\label{6.35}
\begin{aligned}
	&\left| \mathbb{E} \left[ I_{b}^{j,N} \right] \right|\\
\le &\sqrt{L_1}\int_{t_0}^{t_0+h}{\mathbb{E}}\left[ \left( 1+\left|\hat{X}^{j,N}\left( t_0,X^{j,N}_0;t_0+h \right)\right |^{2\kappa -2}+\left|X^{j,N}_0\right|^{2\kappa -2} \right) ^{\frac{1}{2}}\right.\\
& \left.\times\left|\hat{X}^{j,N}\left( t_0,X^{j,N}_0;t_0+h \right) -X^{j,N}_0\right| \right] \mathrm{d}s\\
\le &\sqrt{L_1}\int_{t_0}^{t_0+h}{\left( \mathbb{E} \left[ \left( 1+\left|\hat{X}^{j,N}\left( t_0,X^{j,N}_0;t_0+h \right) \right|^{2\kappa -2}+\left|X^{j,N}_0\right|^{2\kappa -2} \right)  \right] \right) ^{\frac{1}{2}}}\\
&\times \left( \mathbb{E} \left[ \left|\hat{X}^{j,N}\left( t_0,X^{j,N}_0;t_0+h \right) -X^{j,N}_0\right|^{2} \right] \right) ^{\frac{1}{2}}\mathrm{d}s\\
\le &K_{59}\mathbb{E} \left[ \left(1+\left|X_0\right|^{8\kappa -4}\right) \right] h^{\frac{3}{2}}.\\
\end{aligned}
\end{equation}
Combining \eqref{5.66}, \eqref{6.23}, and \eqref{6.35}, we obtain, for any $1 \leq j \leq N$,
\begin{equation}\tag{5.134}\label{6.36}
\begin{aligned}
	&\left|\mathbb{E} \left[X^{j,N}\left( t_0,X^{j,N}_0;t_0+h \right) -\hat{X}^{j,N}\left( t_0,X^{j,N}_0;t_0+h \right)\right ]\right|\\
	\le& \left| \mathbb{E} \left[ X^{j,N}\left( t_0,X^{j,N}_0;t_0+h \right) -Y_{E}^{j,N}\left(t_0,X^{j,N}_0;t_0+h\right) \right] \right|+\left| \mathbb{E} \left[ I_{b}^{j,N} \right] \right|\\
	\le& K_{28}\mathbb{E} \left[ \left(1+\left|X_0\right|^{8\kappa +2}\right) \right] h^{\frac{3}{2}}+K_{59}\mathbb{E} \left[ \left(1+\left|X_0\right|^{8\kappa -4}\right) \right] h^{\frac{3}{2}}\\
	\le& K_{10}\mathbb{E} \left[ \left(1+\left|X_0\right|^{8\kappa +2}\right) \right] h^{\frac{3}{2}}.\\
\end{aligned}
\end{equation}
This establishes \eqref{6.32}, which corresponds to \eqref{6.44} in Theorem~\ref{theorem6.7}. 
Analogous to the proof of \eqref{6.28}, \eqref{6.33} is readily established.
\end{proof} 

\begin{lemma}\label{lemma6.7}
    Assuming that Assumptions~\ref{ass2.1}--\ref{ass2.4} and \ref{ass6.1} are satisfied, for any $1 \leq \kappa \leq \tfrac{q_0 - 12}{8}$ and $0 < h \leq h_2$, we can derive the following estimates:
\begin{equation}\tag{5.135}\label{6.38}
\underset{1\leqslant j\leqslant N}{\mathrm{sup}}\left\{ \mathbb{E} \left[ \left| I_{b}^{j,N} \right|^2 \right] \right\} ^{\frac{1}{2}}\le K_{60}\left\{ \mathbb{E} \left[ \left(1+\left|X_0\right|^{16\kappa -8} \right)\right] \right\} ^{\frac{1}{2}}h^{\frac{3}{2}},
\end{equation}
\begin{equation}\tag{5.136}\label{6.39}
\underset{1\leqslant j\leqslant N}{\mathrm{sup}}\left\{ \mathbb{E} \left[ \left|X^{j,N}\left( t_0,X^{j,N}_0;t_0+h \right) -\hat{X}^{j,N}\left( t_0,X^{j,N}_0;t_0+h \right)\right |^2 \right] \right\} ^{\frac{1}{2}}\le K_{11}\left\{ \mathbb{E} \left[ \left(1+\left|X_0\right|^{16\kappa +24}\right) \right] \right\} ^{\frac{1}{2}}h,
\end{equation}
where $I_b^{j,N}$ is as defined in \eqref{6.23}, and $K_{60}$ and $K_{11}$ are constants independent of $h$, $N$.
\end{lemma}
\begin{proof} 
By invoking Assumption~\ref{ass2.2} and \eqref{6.23}, along with the Cauchy-Schwarz inequality, Theorem~\ref{momentboundsforBEM}, Lemma~\ref{lemma6.4}, and H\"older's inequality, we derive, for any $1 \leq j \leq N$,
\begin{equation}\tag{5.137}\label{6.40}
\begin{aligned}
	&\mathbb{E} \left[ \left| I_{b}^{j,N} \right|^2 \right]\\
	\le& h\int_{t_0}^{t_0+h}{\mathbb{E} \left[ \left| b\left( \hat{X}^{j,N}\left( t_0,X^{j,N}_0;t_0+h \right) ,\mu _{t_0}^{X,N} \right) -b\left( X^{j,N}_0,\mu _{t_0}^{X,N} \right) \right|^2 \right]}\mathrm{d}s\\
	\le& L_1h\int_{t_0}^{t_0+h}{\left\{ \mathbb{E} \left[ \left( 1+\left|\hat{X}^{j,N}\left( t_0,X^{j,N}_0;t_0+h \right) \right|^{2\kappa -2}+\left|X^{j,N}_0\right|^{2\kappa -2} \right) ^{2} \right] \right\} ^{\frac{1}{2}}}\\
	&\times \left\{ \mathbb{E} \left[ \left|\hat{X}^{j,N}\left( t_0,X^{j,N}_0;t_0+h \right) -X^{j,N}_0\right|^{4} \right] \right\} ^{\frac{1}{2}}\mathrm{d}s\\
	\le& K_{60}^{2}\mathbb{E} \left[ \left(1+\left|X_0\right|^{16\kappa -8}\right) \right] h^3\,, \\
\end{aligned}
\end{equation}
which establishes \eqref{6.38}.

Next, by invoking Assumptions~\ref{ass2.2} and~\ref{ass6.1}, together with  \eqref{2.8}, \eqref{interacting}, \eqref{6.22}, \eqref{6.23}, and \eqref{6.38}, as well as the It\^o isometry, H\"older's inequality, Theorem~\ref{momentboundednessofthesolution}, and Lemma~\ref{lemma5.4}, the left-hand side of \eqref{6.39} admits the following estimate for any $1 \leq j \leq N$:
\begin{equation}\tag{5.138}\label{6.41}
\begin{aligned} 
	&\mathbb{E} \left[ \left|X^{j,N}\left( t_0,X^{j,N}_0;t_0+h \right) -\hat{X}^{j,N}\left( t_0,X^{j,N}_0;t_0+h \right)\right |^2 \right]\\
	\le& 3h\int_{t_0}^{t_0+h}{\mathbb{E} \left[ \left| b\left( X^{j,N}\left( t_0,X^{j,N}_0;s \right) ,\mu _{s}^{X,N} \right) -b\left( X^{j,N}_0,\mu _{t_0}^{X,N} \right) \right|^2 \right]}\mathrm{d}s+3\mathbb{E} \left[ \left| I_{b}^{j,N} \right|^2 \right]\\
	&+3\int_{t_0}^{t_0+h}{\mathbb{E}}\left[ \left|\left| \sigma \left( X^{j,N}\left( t_0,X^{j,N}_0;s \right) ,\mu _{s}^{X,N} \right) -\sigma \left( X^{j,N}_0,\mu _{t_0}^{X,N} \right) \right|\right| ^2 \right] \mathrm{d}s\\
\le& K_{61}h\int_{t_0}^{t_0+h}{\left\{ \mathbb{E} \left[ \left( 1+\left|X^{j,N}\left( t_0,X^{j,N}_0;s \right) \right|^{2\kappa -2}+\left|X^{j,N}_0\right|^{2\kappa -2} \right) ^{2} \right] \right\} ^{\frac{1}{2}}}\\
&\times \left\{ \mathbb{E} \left[ \left|X^{j,N}\left( t_0,X^{j,N}_0;s \right) -X^{j,N}_0\right|^{4} \right] \right\} ^{\frac{1}{2}}\mathrm{d}s\\
&+K^{2}_{60}\mathbb{E} \left[ \left(1+\left|X_0\right|^{16\kappa -8}\right) \right] h^3+K_{62}\mathbb{E} \left[ \left(1+\left|X_0\right|^{2\kappa}\right) \right] h^2\\
\le& K_{63}\mathbb{E} \left[ \left(1+\left|X_0\right|^{8\kappa -4} \right)\right] h^3+K^{2}_{60}\mathbb{E} \left[ \left(1+\left|X_0\right|^{16\kappa -8}\right) \right] h^3+K_{62}\mathbb{E} \left[ \left(1+\left|X_0\right|^{2\kappa}\right) \right] h^2\\
\le& K^{2}_{11}\mathbb{E} \left[ \left(1+\left|X_0\right|^{16\kappa +24}\right) \right] h^2\,,
\end{aligned}
\end{equation}
This establishes \eqref{6.39}, which corresponds to \eqref{6.45} in Theorem~\ref{theorem6.7}.
The lemma is thus proved.
\end{proof}   

\section*{\textbf{Acknowledgement}}
This research was supported by the Natural Science Foundation of China under Grant No. 12371417, the NSERC Discovery Grant RGPIN 2024-05941, and a Centennial Fund from the University of Alberta. 
The authors  extend their thanks to Prof. Xiaojie Wang, Lei Dai, Xiaoming Wu, and Yuying Zhao for their insightful and helpful comments.
{\color{black} 
}



\end{document}